\documentclass[a4paper,11pt]{amsart}
\usepackage{amssymb}
\usepackage[utf8]{inputenc}
\usepackage[margin=0.8in]{geometry}
\usepackage{hyperref,hypcap}
\hypersetup{colorlinks=true,allcolors=black,bookmarksdepth=3}
\usepackage{mathrsfs,bm}
\usepackage{mathtools}
\usepackage{galois}
\usepackage{tikz}
\tikzstyle{Solid} = [circle,fill,inner sep = 2]
\tikzstyle{Hollow} = [circle,draw,inner sep = 2,line width=1]
\tikzstyle{DSolid} = [rectangle,fill,inner sep = 2.5,line width=1]
\tikzstyle{DHollow} = [rectangle,draw,inner sep = 2.5,line width=1]
\tikzstyle{Semi} = [circle split,rotate=90,draw,inner sep = 2,line width = 1]
\usetikzlibrary{shapes,cd}
\usepackage[inline]{enumitem}
\usepackage{multicol}
\usepackage{array}
\usepackage{subcaption}

\usepackage[backend=bibtex,style=numeric-comp,sorting=nyt,doi=false,isbn=false,url=false,maxbibnames=99,maxcitenames=99]{biblatex}
\usepackage[nonewpage]{imakeidx}
\makeindex[columns=2,title=List of Symbols]

\theoremstyle{plain}
\newtheorem{theorem}{Theorem}[section]
\newtheorem{lemma}[theorem]{Lemma}
\newtheorem{proposition}[theorem]{Proposition}
\newtheorem{corollary}[theorem]{Corollary}
\theoremstyle{definition}
\newtheorem{definition}[theorem]{Definition}
\theoremstyle{remark}
\newtheorem{remark}[theorem]{Remark}

\numberwithin{equation}{section}
\numberwithin{figure}{section}
\numberwithin{table}{section}
\newcommand{\Tk}{T_{\mathrm{kin}}}
\newcommand{\Tt}{T_{\mathrm{trans}}}

\newcommand{\diff}{\mathop{}\!\mathrm{d}}

\newcommand{\Hess}{\bm{\mathrm{H}}}
\DeclareMathOperator{\card}{card}
\DeclareMathOperator{\diam}{diam}
\DeclareMathOperator{\supp}{supp}

\DeclareMathOperator{\ind}{ind}

\newcommand{\wick}[1]{{{:}#1{:}}}

\newcommand{\loc}{\mathrm{loc}}

\newcommand{\reg}{\mathrm{reg}}

\newcommand{\cp}{q}

\allowdisplaybreaks

\begin{document}

\title{Inhomogeneous turbulence for \\the Wick Nonlinear Schr\"odinger Equation}
\author[Z. Hani]{Zaher Hani}
\address[Zaher Hani]{University of Michigan Department of Mathematics}
\email{zhani@umich.edu}
\author[J. Shatah]{Jalal Shatah}
\address[Jalal Shatah]{New York University Courant Institute of Mathematical Sciences}
\email{shatah@cims.nyu.edu}
\author[H. Zhu]{Hui Zhu}
\address[Hui Zhu]{University of Michigan Department of Mathematics}
\curraddr{Imperial College London Department of Mathematics}
\email{zhuhui@umich.edu; hui.zhu@imperial.ac.uk}

\begin{abstract}
We introduce a simplified model for wave turbulence theory---the Wick NLS, of which the main feature is the absence of all self-interactions in the correlation expansions of its solutions.
For this model, we derive several wave kinetic equations that govern the effective statistical behavior of its solutions in various regimes.
In the homogeneous setting, where the initial correlation is translation invariant, we obtain a wave kinetic equation similar to the one predicted by the formal theory.
In the inhomogeneous setting, we obtain a wave kinetic equation that describes the statistical behavior of the wavepackets of the solutions, accounting for both the transport of wavepackets and collisions among them.
Another wave kinetic equation, which seems new in the literature, also appears in a certain scaling regime of this setting and provides a more refined collision picture.
\end{abstract}

\maketitle

\tableofcontents

\section{Introduction}

\emph{Wave turbulence theory} is a broad framework aimed at describing the nonequilibrium statistical mechanics of nonlinear wave systems.
At the core of this theory is a kinetic formalism known as the the \emph{wave kinetic theory}, whose aim is to describe via a collisional paradigm the mesoscopic evolution of a microscopic ensemble of waves.
Emerging as an extension of Boltzmann's \emph{kinetic theory} \cite{Boltzmann1872} for dilute gases, it was first initiated by \citeauthor{Peierls1929} \cite{Peierls1929} for the quantized wave equation, and later independently formulated by \citeauthor{Hasselmann1962} \cite{Hasselmann1962,Hasselmann1963a,Hasselmann1963b} for gravity water waves.
Since then, the subject evolved into a systematic framework for the study of statistical physics of nonlinear wave systems, partially due to works of Zakharov and his collaborators \cite{Zakharov1992} who discovered a close connection to the hydrodynamic turbulence, whence the name ``wave turbulence theory''.
See e.g., \cite{Nazarenko2011,Spohn2008}, for a detailed treatment.

The purpose of this paper is to introduce the \emph{Wick nonlinear Schr\"odinger equation} (WNLS) and conduct a self-contained study of its wave turbulence theory.
A much simplified collisional behavior of this model makes it an excellent ``mathematical laboratory'' for rigorously formulating and testing theoretical foundations of wave turbulence theory.
The WNLS renormalizes the well-known \emph{nonlinear Schr\"odinger equation} (NLS) by eliminating all \emph{self-interactions}, or \emph{normal ordering} the terms in the collisional expansion, as in quantum field theory (QFT).
The algebraic connection between WNLS and the Fock space formalism of QFT enables WNLS to be exactly solvable using \citeauthor{Dyson1949}'s \emph{perturbation method} \cite{Dyson1949} without any remainder.
This leads to two significant simplifications, namely simpler Feynman diagrams and the absence of technical and conceptual difficulties related to the convergence of the Dyson series.
Leveraging these inherent advantages allows for obtaining the effective statistical and kinetic behaviors in various regimes rigorously, with reasonable effort, and highlighting the connections between them.
Our main results will be presented in Theorems \ref{thm:cauchy}, \ref{thm:hom-inhom} and \ref{thm:inhom-2nd}.
They are also summarized (and advertised) in the diagram shown in Figure~\ref{fig:wk-theory}, where the terminologies and notations will be more comprehensible to the readers after completing this introductory section.

\begin{figure}[htb!]
    \begin{tikzpicture}[font=\footnotesize]
        \node(WNLS){WNLS};
        \node[label={[xshift=-4pt]right:{periodic}}](WNLSp) at (5,0) {WNLS};
        \node[label={[xshift=-4pt]right:{homogeneous}}](WK-hom) at (3.7,-1.8) {WK};
        \node[label={[xshift=-4pt]right:{semi-homogeneous}}](WK-semi) at (5.5,-4) {WK};
        \node[label={[xshift=-4pt]right:{inhomogeneous}}](WK-inhom) at (5.5,-5) {WK};
        \node(WK) at (2,-4) {WK};
        \node(WK2) at (-2.5,-4) {WK-2};
        \draw[->] (WNLS) to node[midway,below] {$\beta=\infty$} (WNLSp);
        \draw[->] (WNLSp) to node[pos=0.5,above,sloped]{Fourier} (WK-hom);
        \draw[->] (WNLS) to node[midway,above,sloped]{Microlocalization} node[midway,below,sloped]{$\alpha \le \beta$} (WK);
        \draw[->] (WNLS) to node[midway,above,sloped]{2nd Microlocalization} node[midway,below,sloped]{$\alpha \le \beta \in (1,\infty)$} (WK2);
        \draw[->] (WK) to node[pos=0.5,above,sloped]{$\beta=\infty$} (WK-hom);
        \draw[->] (WK) to node[midway,above,sloped]{$\alpha<\beta<\infty$}(WK-semi);
        \draw[->] (WK) to[out=-50,in=180] node[pos=0.7,above]{$\beta=\alpha$} (WK-inhom);
        \draw[->] (WK2) to node[midway,above,sloped]{Integrate in $\zeta$} (WK);
        \draw[->] (WK-semi) to node[pos=0.4,above,sloped]{At any $x$} (WK-hom);
    \end{tikzpicture}
\caption{Wave turbulence theory for WNLS}
\label{fig:wk-theory}
\end{figure}
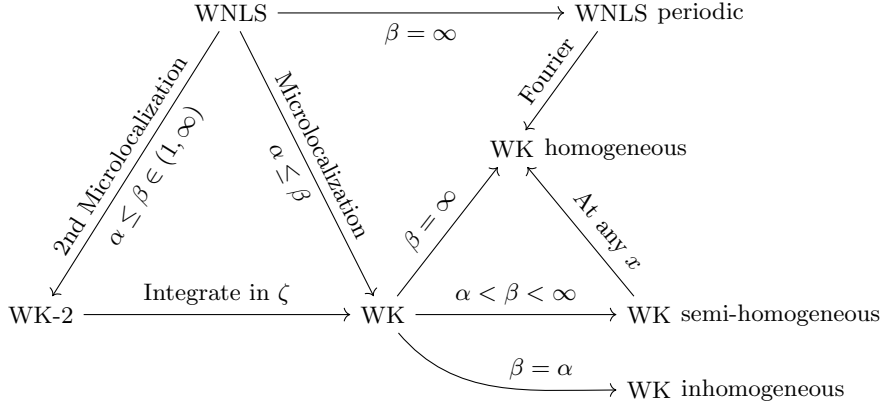

\subsection{Wick NLS}
\label{sec:wick-nls}

Formally, the WNLS is defined by replacing the product in the nonlinearity of NLS with the \emph{Wick product} $\odot$\index{Operations and transforms!Wick product!Wick product $\odot$}, which will be defined in \S\ref{sec:wick-product}.
The equation is given by
\begin{equation}\tag{WNLS}
    \label{eq:WNLS}
    i \partial_t u + \frac{1}{4\pi} \Delta u
    = -\lambda u \odot \overline{u} \odot u ,
\end{equation}
where the parameter $\lambda \in \mathbb{R} \backslash \{0\}$\index{Parameters!Limit parameters!Strength of nonlinearity $\lambda$} is the \emph{strength of nonlinearity}.
We assume that $\lambda$ is small so that we are in the weakly nonlinear setting.
This is one of the parameters involved in the thermodynamic limit that will be taken to obtain the kinetic theory. 
The sign of $\lambda$ is unimportant in our analysis and will henceforth be fixed as \emph{positive} for concreteness. 
We should mention here that a quintic version of this equation has been introduced by De Suzzoni in \cite{DS22} to investigate other interesting questions in wave turbulence theory.

\subsubsection{Wick product}
\label{sec:wick-product}

Following \cite{Janson1997}, we fix a probability space~$\mathfrak{X}$\index{Sets and spaces!Probability spaces!Probability space $\mathfrak{X}$} and a \emph{complex Gaussian Hilbert space} $H$ over $\mathfrak{X}$, i.e., $H$\index{Sets and spaces!Wiener chaos!Gaussian Hilbert space $H$} is a complex Hilbert subspace of $L^2(\mathfrak{X})$ consisting of \emph{complex Gaussian random variables} (or \emph{Gaussians} for simplicity) with zero expected values.
For all $n \in \mathbb{N}$, let $H^{\wick{n}} = P_{n-1}^\perp \cap P_n$\index{Sets and spaces!Wiener chaos!Homogeneous Wiener chaos $H^{\wick{n}}$} where $P_n$ is the closure in $L^2(\mathfrak{X})$ of the set of all polynomials in $\mathbb{C}[H]$ of degrees $\le n$ with the convention that $P_{-1} = \emptyset$.
Elements of $H^{\wick{n}}$ are \emph{complex homogeneous Wiener chaos} of order $n$.
Assuming in addition that $H$ generates the same $\sigma$-algebra as that of $\mathfrak{X}$, there holds the \emph{Wiener--It\^o chaos decomposition} \cite{Wiener1938,Ito1951}: $L^2(\mathfrak{X}) = \bigoplus_{n \ge 0} H^{\wick{n}}$.
Let $\pi_n:L^2(\mathfrak{X}) \to H^{\wick{n}}$\index{Operations and transforms!Chaos projections!Chaos projection $\pi_n$} be the orthogonal projection, then the Wick product $\odot$ is a commutative and associative bilinear operator on $\bigcup_{n \ge 0} P_n$ such that $X \odot Y = \pi_{p+q} (XY)$ when $X \in H^{\wick{p}}$ and $Y \in H^{\wick{q}}$.
\index{Functions and random variables!Random variables!Random variables $X,Y,X^n,Y^n$}
It naturally extends to the direct product
$\mathfrak{H} = \prod_{n \ge 0} H^{\wick{n}}$\index{Sets and spaces!Wiener chaos!Chaos direct product $\mathfrak{H}$}, whose elements are formal sums $\sum_{n \ge 0} X^n$ with $X^n \in H^{\wick{n}}$.
In fact, it suffices to let
\begin{equation}
    \label{eq:def-wick-product}
    \sum_{n \ge 0} X^n \odot \sum_{n \ge 0} Y^n
    = \sum_{n \ge 0} \Bigl(\sum_{p+q=n}X^p \odot Y^q \Bigr).
\end{equation}
The WNLS is then an equation for \emph{generalized} stochastic processes $u : \mathbb{R} \times \mathbb{R}^d \to \mathfrak{H}$.

Under the name \emph{$S$-product}, the Wick product was first used by \citeauthor{Wick1950} \cite{Wick1950} as an algebraic technique to simplify \citeauthor{Dyson1949}'s computation \cite{Dyson1949} of the Heisenberg $S$-matrix, and has ever since been widely utilized in QFT (see e.g., \cite{Simon1974}).
Its closely related (see \S\ref{sec:wick-renorm}) probabilistic counterpart---the Wick product defined by~\eqref{eq:def-wick-product}---was introduced by \citeauthor{Hida1967} \cite{Hida1967} under the white noise framework.
Due to its natural connections with the It\^{o}--Skorohod integral, the Wick product has also been successfully applied to the study of SPDEs \cite{Lindstroem1991,Lindstroem1992}.

\subsubsection{Perturbative expansion}

\label{sec:perturb-exp}

We aim to study the dynamics of WNLS with well-prepared initial data (see \S\ref{sec:ensemble-wave}).
It is usually a highly nontrivial task to prove the long-time existence of solutions with such initial data for an arbitrary nonlinear wave system but it is not the case for WNLS, as Dyson's perturbation method solves the equation exactly.
In fact, write $u = \sum_{n \ge 0} u^n$ where $u^n \coloneq \pi_n u$, then WNLS is equivalent to the following system:
\begin{equation}\tag{WNLS$'$}
    \label{eq:WNLS-component}
    i \partial_t u^{n} + \frac{1}{4\pi} \Delta u^{n} = -\lambda \sum_{n_1+n_2+n_3=n} u^{n_1} \odot \overline {u^{n_2}} \odot u^{n_3}, \quad n \in \mathbb{N}.
\end{equation}

However, $u$ is a generalized stochastic process that lives in white noise distribution spaces such as the Hida and Kondratiev distribution spaces \cite{Hida1975,Kondratiev1993,Kondratiev1995,Kondratiev1996}, all of which are much larger than $L^2(\mathfrak{X})$.
This unfortunate fact is due to the factorial growth, as $n \to \infty$, of the number of possible pairings among $2n$ Gaussians, an estimate that is needed when computing the variance $\mathbb{E}|u^n|^2$ via Wick's probability theorem (or the complex Isserlis theorem). Note that for NLS:
\begin{equation}\tag{NLS}
    \label{eq:NLS}
    i \partial_t u + \frac{1}{4\pi} \Delta u = -\lambda |u|^2 u,
\end{equation}
the same issue requires truncating the Dyson series and estimating the remainders carefully.
This approach has been consistently employed in recent studies involving Feynman diagram expansions, beginning with \cite{Erdoes2008a,Lukkarinen2007} and continuing through more recent works aimed at rigorously justifying wave turbulence theory, which we will review in greater detail later. 
In the context of WNLS, we shall avoid all these technical difficulties by applying a chaos cutoff
\begin{equation*}
    \Pi_n = \sum_{0 \le j \le n} \pi_j: \mathfrak{H} \to \bigoplus_{0 \le j \le n} H^{\wick{j}}
    \index{Operations and transforms!Chaos projections!Chaos cutoff $\Pi_n$}.
\end{equation*}
Specifically, we shall study the variance $\mathbb E |\Pi_n u|^2$ and take the kinetic limit that involves taking $n\to \infty$.
Truncations like this are a common occurrence in theory. For instance, when studying the kinetic theory of discrete lattice systems, it is often necessary to take averages after a finite lattice truncation and then take the full lattice limit afterward.

\subsubsection{Wick renormalization}
\label{sec:wick-renorm}

Perhaps the greatest advantage of WNLS over NLS lies in the two essentially equivalent algebraic properties that can be described as follows:
\begin{enumerate}
    \item Absence of \emph{self-interactions}.
    Formally, solutions to WNLS are \emph{Wick renormalizations} of the Dyson series of NLS.
    This renormalization is executed by replacing all products of Gaussians with the corresponding Wick products:
    $ \prod \mathfrak{g}_j \mapsto \bigodot \mathfrak{g}_j.$\index{Functions and random variables!Random variables!Gaussians $\mathfrak{g}$}
    Equivalently, this amounts to subtracting from $\prod \mathfrak{g}_j$ all \emph{self-interactions}, or \emph{self-pairings}, among its factors of Gaussians (see e.g., \cite[\S3.1]{Janson1997} and \cite{Lukkarinen2016}).
    Consequently, when computing $\mathbb{E} |u^n|^2$, one only considers pairings that connect one Gaussian from $u^n$ and the other from $\overline{u^n}$.

    \item Preservation of \emph{normal ordering}.
    Recall that (see e.g., \cite{Janson1997}), if $\mathfrak{g} \in H^{\wick{1}}$ and $X \in H^{\wick{n}}$, then $\mathfrak{g} X = \mathcal{A}_{\mathfrak{g}}^\dagger X + \mathcal{A}_{\mathfrak{g}} X$, where $\mathcal{A}_{\mathfrak{g}}^\dagger X = \mathfrak{g} \odot X$\index{Operations and transforms!Wick product!Creation operator $\mathcal{A}_{\mathfrak{g}}^\dagger$} and $\mathcal{A}_{\mathfrak{g}} X = \pi_{n-1}(\overline{\mathfrak{g}} X)$\index{Operations and transforms!Wick product!Annihilation operator $\mathcal{A}_{\mathfrak{g}}$} define respectively the so-called creation operator and the annihilation operator.
    The passage from the Wick product to the \emph{Wick ordering}, also known the \emph{normal ordering} (as named in Wick's original work \cite{Wick1950}), is accomplished through the identity
    $\wick{\prod (\mathcal{A}_{\mathfrak{g}_j}^\dagger + \mathcal{A}_{\mathfrak{g}_j})}= \bigodot \mathfrak{g}_j$\index{Operations and transforms!Wick product!Wick ordering $\wick{~}$},
    where the the double-colon operation puts the product into normal order by moving all creation operators to the left of annihilation operators (see e.g., \cite[\S13.2]{Janson1997}).
    Solutions to WNLS, which are polynomials, are normal ordered.
    However, for other models, the normal ordering is usually not preserved in the perturbative expansion.
    For more discussions on this topic, see \cite{Lukkarinen2009}.
\end{enumerate}

The normal ordering has a profound effect on the analysis,  as it removes self-interactions and helps avoid major difficulties.
In fact, self-interactions can generate divergent terms in the Dyson series, and one has to find appropriate cancellations or gauge transformations to get rid of them.
Recent rigorous studies of homogeneous turbulence theory for NLS in \cite{Deng2021,Deng2021b} utilized both techniques, which constitute a significant portion of the analysis there.
It is worth noting that cancellations among divergent structures is fairly robust while gauge transformations are more delicate and dependent on the intrinsic symmetries of the model.
By working with WNLS as a renormalization of NLS, one is free from all such divergent interactions and is able to focus on the structural analysis of wave turbulence theory in various regimes.

\subsection{Ensemble of wavepackets}
\label{sec:ensemble-wave}

To model an ensemble of wavepackets using WNLS, we focus on solutions that are superpositions of wavepackets and satisfy the following assumptions:
\begin{enumerate}
    \item \emph{Discrete wavenumber assumption}: This states that these wavepackets have \emph{distinct} characteristic wavenumbers belonging to a \emph{discrete lattice} $\Lambda \subset \mathbb{R}^d$\index{Sets and spaces!Lattices!General lattice $\Lambda$}. We further assume that the mesh of $\Lambda$ is of order $L^{-1}$ with $L > 0$ being large\index{Parameters!Limit parameters!Large box parameter $L$}. 
    In the homogeneous setting, $L$ is the \emph{large box parameter} as it corresponds to the size of the spatial domain $\mathbb{R}^d / \Lambda$, a large periodic box of size $L$. 
    In general $L$, or more precisely $L^d$, is the parameter that signifies the \emph{number of wavepackets} that are effectively colliding or interacting in the process. 
    \item \emph{Narrow wavepacket assumption}: This states that these wavepackets have spectral width of order $\epsilon$ with $\epsilon \ge 0$\index{Parameters!Limit parameters!Narrow wavepacket parameter} being small, i.e., in the momentum space, each wavepacket is supported within an $\epsilon$-neighborhood of its characteristic wavenumber.
    \begin{enumerate}
        \item When $\epsilon = 0$, the wavepackets become plane waves and the solution $u$ is $\Lambda$-periodic.
        As a result, the energy spectrum (defined in \S\ref{sec:effective-dynamics}) of the solution is translation invariant in the position variable or, as we shall call it, homogeneous (see \S\ref{sec:space-homogeneity}).
        \item When $\epsilon>0$, this assumption is equivalent to the \emph{slow varying envelope approximation} in physics as these wavepackets now have wave lengths of scale $\epsilon^{-1}$.
        In this setting, the energy spectrum is no longer homogeneous, and the parameter $\epsilon$ measures the degree of its inhomogeneity (see \S\ref{sec:space-homogeneity}). 
    \end{enumerate}  
    \item \emph{Phase randomization assumption} or, following \cite{Zakharov1992}, the \emph{phase chaotization assumption}: This states that these wavepackets are \emph{uncorrelated}. 
    This assumption is closely related to the \emph{molecular chaos hypothesis} in the classical kinetic theory \cite{Boltzmann1872, Gallagher2013}.
\end{enumerate}

It is important to note that, while these assumptions are imposed to the solution $u$, in reality, they can only be assumed to hold initially.
Demonstrating the propagation of these properties by the underlying nonlinear wave system is a crucial undertaking, although it is frequently a challenging technical feat.
We will use a typical example of an initial distribution that satisfies the above three assumptions, which we will refer to as \emph{well-prepared} initial data, following the terminology in \cite{Deng2021b}.
Setting $\Lambda = \mathbb{Z}^d_L \coloneqq L^{-1} \mathbb{Z}^d_L$\index{Sets and spaces!Lattices!Integer lattice $\mathbb{Z}^d_L$}, well-prepared initial data are of the form:
\begin{equation}
    \label{eq:ini-data-well-prepared}
    u(0,x) = \frac{1}{L^{d/2}} \sum_{k \in \mathbb{Z}^d_L} \phi(\epsilon x,k) e^{2\pi i k \cdot x} \mathfrak{g}_k,
\end{equation}
where $(\mathfrak{g}_k)_{k \in \mathbb{Z}^d_L}$ is a sequence of independent and standard complex Gaussians in $H^{\wick{1}}$ and $\phi$ is a Schwartz function.
We will sometimes assume that $\phi$ has a compact spectrum in that its Fourier transform with respect to the position variable $x$, denoted by $\widehat{\phi} = \widehat{\phi}(\xi,k)$ is compactly supported in $\mathbb{R}^{2d}$.
Here and throughout this paper the Fourier transform is defined by $\widehat{f}(\xi) = \int e^{-2\pi i x\cdot \xi} f(x) \diff x$\index{Operations and transforms!Phase space transforms!Fourier transform $\widehat{f}$}.
We have imposed restrictive regularity conditions on the profile function $\phi$. While certain decay and smoothness assumptions in $k$ are necessary, we refrain from optimizing them within the scope of this paper. This decision is deliberate, aiming to maintain focus on the novel conceptual aspects of the problem rather than delving into intricate technical details.

To solve WNLS, we choose the ansatz
\begin{equation}
    \label{eq:ansatz-wp}
    u(t,x) = \frac{1}{L^{d/2}} \sum_{k \in \mathbb{Z}^d_L} A_{k}(t,\epsilon x) e^{2\pi i (k \cdot x - |k|^2t /2)},
    \index{Functions and random variables!Wavepacket functions!Amplitude $A_k$}
\end{equation}
and let the profiles $(A_k)_{k\in\mathbb{Z}^d_L}$ solve the following system of Cauchy problems:
\begin{equation}
    \label{eq:WNLS-wp-prof}
    \mathcal{L}_k^\epsilon A_k(t)
        = -\frac{\lambda}{L^d} \sum_{k_1-k_2+k_3=k}
        e^{2\pi i t \Omega(\bm{k})} A_{k_1}(t) \odot \overline{A_{k_2}(t)} \odot A_{k_3}(t),
    \quad k \in \mathbb{Z}^d_L,
\end{equation}
where the differential operator $\mathcal{L}_k^\epsilon = i \partial_t + i \epsilon k \cdot \nabla + (4\pi)^{-1}\epsilon^2 \Delta $ applies to the expansion $A_k = \sum_{n \ge 0} A_k^n$ in a chaos-wise manner, and the phase factor is defined by
\begin{equation}
    \label{eq:def-reso-factor}
    \Omega(\bm{k}) = \frac{1}{2} \bigl( |k_1-k_2+k_3|^2 - |k_1|^2 + |k_2|^2 - |k_3|^2 \bigr),
    \index{Functions and random variables!Resonance functions!Resonance factor $\Omega$}
    \quad
    \bm{k} = (k_1,k_2,k_3).
\end{equation}

Observe that setting $\epsilon = 0$ in~\eqref{eq:WNLS-wp-prof} yields ODEs in $t$ and the profiles $A_k$ become $x$-independent.
In that case, the ansatz~\eqref{eq:ansatz-wp} gives the Fourier series expansion of~$u$ and the initial condition on $A_k$ is given by $A_k(0,x) = \phi(0,k) \mathfrak{g}_k$.
When $\epsilon > 0$, the ansatz~\eqref{eq:ansatz-wp} gives a wavepacket expansion of~$u$ and the initial condition is given by $A_k(0,x) = \phi(x,k) \mathfrak{g}_k$.
For all values of $\epsilon$, one can solve~\eqref{eq:WNLS-wp-prof} iteratively with respect to the chaos order $n$, following the approach explained in \S\ref{sec:perturb-exp}.
Moreover, it is not difficult to see that, because well-prepared initial data live in $H^{\wick{1}}$, these solutions have vanishing chaos of even orders.
These observations are summarized in our first theorem stated below, which examines the global existence of solutions to WNLS and verifies the phase randomization assumption.

\begin{theorem}
    \label{thm:cauchy}
    For all $d \ge 1$, $ \lambda > 0$, $L > 0$, $\epsilon \ge 0$, and $\phi$ is a Schwartz function on $\mathbb{R}^{2d}$, there exists a unique sequence of stochastic processes $(A_k : \mathbb{R} \times \mathbb{R}^d \to \mathfrak{H})_{k\in\mathbb{Z}^d_L}$ such that
    \begin{enumerate}
        \item For all $k \in \mathbb{Z}^d_L$, we have $A_k(0,x) = \phi(x,k) \mathfrak{g}_k, $ if $ \epsilon > 0 $; and $
        A_k(0,x) = \phi(0,k) \mathfrak{g}_k, $ if $\epsilon = 0$;
        \item For all $k \in \mathbb{Z}^d_L$ and $n \in \mathbb{N}$, the stochastic process $A^n_k:\mathbb{R} \times \mathbb{R}^d \to H^{\wick{n}}$ is infinitely differentiable;
        \item Both sides of~\eqref{eq:WNLS-wp-prof} are equal pointwise as generalized stochastic processes in $\mathfrak{H}$.
    \end{enumerate}
    Moreover $A_k^{2n} = 0$ for all $k \in \mathbb{Z}^d_L$ and $n \in \mathbb{N}$; if $(n,k) \ne (n',k')$, then for all $t,t' \in \mathbb{R}$ and $x,x' \in \mathbb{R}^d$,
    \begin{equation}
        \label{eq:cdt-phase-randomization-wp}
        \mathbb{E}[A_{k}^n(t,x)\overline{A_{k'}^{n'}(t',x')}] = 0.
    \end{equation}
    Consequently, let $u:\mathbb{R} \times \mathbb{R}^d \to \mathfrak{H}$ be given by~\eqref{eq:ansatz-wp}, then $u$ is the unique generalized stochastic process which satisfies initial condition~\eqref{eq:ini-data-well-prepared} and pointwisely solve WNLS.
\end{theorem}

\subsection{Effective dynamics}
\label{sec:effective-dynamics}

In order to gain insight into the statistical behavior of solutions to WNLS, we examine the effective dynamics of their correlation functions.
Although the analysis of temporal correlation is also an intriguing problem, we will confine ourselves to the study of the spatial correlation function $\mathbb{E} \{u(t,x+y/2) \overline{ u(t,x-y/2)} \}$\index{Functions and random variables!Correlation functions!Spatial correlation} and frequential correlation function $\mathbb{E}\{\widehat{u}(t,\xi+\eta/2) \overline{\widehat{u}(t,\xi-\eta/2)}\}$\index{Functions and random variables!Correlation functions!Frequential correlation}, where the Fourier transform $\widehat{u}$ is defined in terms of its chaos expansion. 
These two correlation functions are essentially equivalent, as they are related by the formula:
\begin{equation*}
    \begin{aligned}
        \mathbb{E} \mathcal{W}[u(t)](x,\xi) 
        & = \int e^{-2\pi i y \cdot \xi} \mathbb{E} \{u(t,x+y/2) \overline{ u(t,x-y/2)} \} \diff y\\
        & = \int e^{+ 2\pi i x \cdot \eta} \mathbb{E}\{\widehat{u}(t,\xi+\eta/2) \overline{\widehat{u}(t,\xi-\eta/2)}\} \diff \eta.
        \index{Functions and random variables!Energy spectra!Energy spectrum $\mathbb{E}\mathcal{W}$}
    \end{aligned}
\end{equation*}
The function $\mathcal{W}[u]$\index{Operations and transforms!Phase space transforms!Wigner distribution} is the Wigner distribution of $u$ (see e.g., \cite{Wigner1932,Lions1993,Gerard1997}).
It will be specifically defined in \S\ref{sec:energy-spec-non-periodic}, and its expected value will be called the energy spectrum of $u$.

If $u : \mathbb{R} \times \mathbb{R}^d \to \mathfrak{H}$ is a generalized stochastic process, then the energy spectrum $\mathbb{E} \mathcal{W}[u]$ cannot always be defined.
In this case, one needs to implement the chaos cutoff $\Pi_n$ and consider the truncated energy spectrum $\mathbb{E} \mathcal{W}[\Pi_n u]$ instead.

\subsubsection{Scaling regimes}
\label{sec:scaling-regime}

Wave turbulence theory conjectures that an ensemble of wavepackets whose evolution is governed by a system with a weak nonlinearity should manifest collisional behaviors at the \emph{kinetic time order} $\Tk$.
The linear part of the system, which is the underling linear Hamiltonian equation, makes wavepackets propagate at their respective group velocities, as manifested by the second term in equation~\eqref{eq:WNLS-wp-prof} which will become significant at some \emph{transport time order} $\Tt$.
Respectively, these two time orders satisfy
\begin{equation*}
    \Tk \sim \lambda^{-2},
    \index{Parameters!Time orders!Kinetic time order $\Tk$}
    \qquad
    \Tt \sim \epsilon^{-1}.
    \index{Parameters!Time orders!Transport time order $\Tt$}
\end{equation*}
Comparing their scales, we have the trichotomy of regimes:
\begin{enumerate}
    \item $\Tk \sim \Tt$, where collisions and transport of wavepackets are simultaneously observable;
    \item $\Tk \ll \Tt$, where only collisions are observable;
    \item $\Tk \gg \Tt$, where wavepackets disperse to infinity before the kinetic time order and therefore do not effectively collide with each other.
\end{enumerate}

We shall use the name \emph{effective dynamics} to denote any asymptotic dynamics of the energy spectrum at the kinetic time order $\mathbb{E} \mathcal{W}[\Pi_n u(t/\lambda^2)]$ under a certain thermodynamic limit:
\begin{multicols}{2}
\begin{enumerate}[label=(\roman*)]
    \item \emph{Weak nonlinearity limit:} $\lambda \to 0$;
    \item \emph{Large box limit:} $L \to \infty$;
    \item \emph{Narrow wavepacket limit:} $\epsilon \to 0$;
    \item \emph{Infinite chaos limit:} $n \to \infty$.
\end{enumerate}
\end{multicols}
Clearly, only the first two scenarios of the trichotomy of time orders correspond to nontrivial effective dynamics.
We shall therefore require that $\Tk \lesssim \Tt$, or equivalently $ \epsilon \lesssim \lambda^2 $.
To simplify the discussion, we will focus on the limit regimes prescribed by the scaling laws
\begin{equation}
    \label{eq:scaling-law}
    \lambda^{-2} = L^\alpha,
    \index{Parameters!Scaling laws!Kinetic scaling parameter}
    \qquad
    \epsilon^{-1} = L^\beta,
    \index{Parameters!Scaling laws!Transport scaling parameter}
\end{equation}
where $0 < \alpha, \beta \le \infty$, and state the results in terms of the fixed parameters $\alpha, \beta$. At the endpoints, we use the conventions that $\alpha = \infty$ corresponds to the linear case $\lambda = 0$, and $\beta = \infty$ corresponds to the periodic (or homogeneous) case $\epsilon = 0$. We remark that in the notation of \cite{Deng2021b, Deng2021, Deng2023}, the quantity $\gamma=\alpha/2$ represents the so-called \emph{scaling law}, which plays a decisive role in the analysis, and is the analog of the Boltzmann--Grad scaling law in Boltzmann's kinetic theory \cite{Grad1958,Lanford1975,Gallagher2013}. 

With these notations, the aforementioned trichotomy corresponds respectively to the relations: \begin{enumerate*}
    \item $\alpha = \beta$,
    \item $\alpha < \beta$,
    \item $\alpha > \beta$,
\end{enumerate*}
and we shall consider only the cases where $\alpha \le \beta$.
In \cite{Deng2021,Deng2021b, Deng2023}, the case $\beta = \infty$ is studied for $0<\alpha\leq 2$, and the kinetic time order is reached, with an additional arithmetic genericity assumption of the lattice $\Lambda$ in the case $\alpha=2$.
This genericity assumption is needed for number theoretic reasons whenever $\alpha\geq 2$.
To avoid imposing such assumptions, we shall restrict the regime to $0< \alpha <2$ for our main results (Theorems~\ref{thm:hom-inhom} and~\ref{thm:inhom-2nd}).

\subsubsection{Space homogeneity}
\label{sec:space-homogeneity}

The phase randomization condition~\eqref{eq:cdt-phase-randomization-wp} implies the computation:
\begin{equation}
    \label{eq:E-spec-wp-expansion}
    \mathbb{E} \mathcal{W}[\Pi_n u(t)](x,\xi)
    = \begin{cases}
        \displaystyle
        \frac{1}{L^d} \sum_{k \in \mathbb{Z}^d_L} 
        \mathbb{E} \mathcal{W}\bigl[\Pi_n A_{k}(t)\bigr]\Bigl(\epsilon x,\frac{\xi-k}{\epsilon}\Bigr)
        \frac{1}{\epsilon^d}, & \epsilon > 0;\\
        \displaystyle
        \frac{1}{L^d} \sum_{k \in \mathbb{Z}^d_L} \mathbb{E}\bigl|\Pi_n A_k(t)\bigr|^2 \bm{\delta}_k(\xi), & \epsilon = 0.
    \end{cases}
\end{equation}
In the periodic case $\epsilon=0$, the energy spectrum $\mathbb{E} \mathcal{W}[\Pi_n u]$ is position-independent, leading to an effective dynamics that we refer to as \emph{(space) homogeneous}.
However, when $\epsilon>0$, the effective dynamics may vary with respect to the position variable $x$, which we refer to as \emph{(space) inhomogeneous}.
In \S\ref{sec:wave-kinetic-eq}, we further distinguish the inhomogeneous setting by introducing a \emph{(space) semi-homogeneous} setting, where the effective dynamics is inhomogeneous but the transport of wavepackets is not observable. 

The homogeneous setting has been extensively investigated in various models, as evidenced by a number of studies including \cite{Buckmaster2021,Deng2021,Deng2021a,Deng2021b,Collot2020,Collot2019,Staffilani2021,Ma2022, Deng2023, Faou23, DH23-2}.
However, progress in the inhomogeneous setting has been limited, with the exception of recent work by \citeauthor{Ampatzoglou2021} \cite{Ampatzoglou2021} on NLS with a quadratic nonlinearity, which only reaches subcritical times.
In the stochastic setting, where time-dependent forcing is permitted in the equation, \citeauthor{Hannani2022} \cite{Hannani2022} have made contributions to the study of a KdV-type equation, reaching the kinetic time.
In an ongoing project, we aim to extend the results in \cite{Deng2021, Deng2023} for NLS to the inhomogeneous setting, thus contributing to the understanding of effective dynamics in this regime.

\subsubsection{Wave kinetic equation}
\label{sec:wave-kinetic-eq}

A formal analysis of wave turbulence theory for WNLS leads to the conjecture that the effective dynamics of the energy spectrum at the kinetic time order is given by the following wave kinetic equation:
\begin{equation}\tag{WK}
    \label{eq:WK}
    (\partial_t + \bm{1}_{\alpha=\beta} k \cdot \nabla) W_k
    = 2 \int_{\mathscr{D}_k}  \bm{\delta}\bigl(\Omega(\bm{k})\bigr) \prod_{1 \le j \le 3} W_{k_j} \diff{\bm{k}}, \quad k\in \mathbb{R}^d,
    \index{Functions and random variables!Kinetic limits!Kinetic limit $W_k$}
\end{equation}
where $W_k$ are functions of $(t,x)$, the factor $\Omega(\bm{k})$ is defined in \eqref{eq:def-reso-factor}, and $\mathscr{D}_k$ is the set of all $\bm{k} = (k_1,k_2,k_3) \in \mathbb{R}^{3d}$ such that $k_1-k_2+k_3=k$.
The Kronecker delta $\bm{1}_{\alpha=\beta}$ can be thought of as a ``switch'' that toggles the transport effect on and off in the equation, depending on the scaling regimes that we are working with.
As a result, the WK equation provides a combined description of the effective dynamics in both homogeneous and inhomogeneous settings.
It is important to compare WK with the wave kinetic equation for NLS which, in the homogeneous setting, is given by:
\begin{equation}
    \label{eq:WK-nls}
    \frac{\diff}{\diff t} N_k = 2 \int_{\mathscr{D}_k} \bm{\delta}\bigl(\Omega(\bm{k})\bigr) N_k N_{k_1} N_{k_2} N_{k_3} \Bigl(\frac{1}{N_k} - \frac{1}{N_{k_1}} + \frac{1}{N_{k_2}} - \frac{1}{N_{k_3}}\Bigr) \diff \bm{k}, \quad  k \in \mathbb{R}^d.
    \index{Functions and random variables!Kinetic limits!Homogeneous kinetic limit $N_k$}
\end{equation}
Within~\eqref{eq:WK-nls}, there exist four cubic integrands in the collision integral: $N_{k_1} N_{k_2} N_{k_3}$, $N_k N_{k_2} N_{k_3}$, $N_k N_{k_1} N_{k_3}$, and $N_k N_{k_1} N_{k_2}$. 
In WK however, there is only one such integrand, specifically $W_{k_1} W_{k_2} W_{k_3}$, which takes on the same form as the first integrand in \eqref{eq:WK-nls}. 
The reason for this difference, which will become more evident in our derivation in \S\ref{sec:hom} and \S\ref{sec:inhom}, is due to the absence of self-interactions in WNLS, for the last three integrands in \eqref{eq:WK-nls} all arise from self-interactions.

We are now ready to state our first main result on the wave turbulence theory of WNLS:

\begin{theorem}
    \label{thm:hom-inhom}
    Let $d \ge 3$,  $\alpha \in (0,2)$, $\beta \in [\alpha,\infty]$, and assume that $\widehat{\phi} \in C_c^\infty(\mathbb{R}^{2d})$.
    Let $(A_k)_{k \in \mathbb{Z}^d_L}$ be the global solution to~\eqref{eq:WNLS-wp-prof} given by Theorem~\ref{thm:cauchy}.
    Then there exists $T > 0$ and a smooth map $[-T,T] \times \mathbb{R}^{2d} \ni (t,x,k) \mapsto W_k(t,x) \in \mathbb{R}$ such that, under the scaling law~\eqref{eq:scaling-law}, there holds
    \begin{equation}
        \label{eq:A-W-convergence}
        \lim_{n \to \infty} \lim_{L\to \infty} \sup_{|t| \le T} \sup_{k \in \mathbb{Z}^d_L} \bigl\| \mathbb{E}\bigl|\Pi_n A_k\bigr|^2\bigl(t/\lambda^2,\cdot \bigr) - W_k(t, \cdot) \bigr\|_{L^\infty} = 0.
    \end{equation}
    Moreover $(W_k)_{k \in \mathbb{R}^d}$ solves WK with initial condition:
    \begin{equation}
        \label{eq:ini-WK}
        W_k(0,x) = \bigl|\phi(x,k)\bigr|^2, \  \text{if}\ \beta < \infty; \quad
        W_k(0,x) = \bigl|\phi(0,k)\bigr|^2, \  \text{if}\ \beta = \infty.
    \end{equation}
    Next, let $u$ be given by~\eqref{eq:ansatz-wp}.
    For all $\chi \in C_c^\infty(\mathbb{R}^{2d})$, define
    \begin{equation*}
        \langle n|\chi|n \rangle(t)= \iint \chi(x,\xi) \mathbb{E}\mathcal{W}[\Pi_n u] \bigl(t/\lambda^2,\bm{1}_{\epsilon \ne 0} \cdot x/\epsilon,\xi\bigr) \diff x \diff \xi,
        \index{Functions and random variables!Energy spectra!Observation $\langle n\vert \chi \vert n\rangle$}
    \end{equation*}
    where we shall omit the position variable from the energy spectrum $\mathbb{E}\mathcal{W}[\Pi_n u]$ when $\epsilon = 0$ as it is $x$-independent in that case.
    Then, under the scaling law~\eqref{eq:scaling-law}, there holds
    \begin{equation}
        \label{eq:E-spec-wigner-lim}
        \lim_{n \to \infty} \lim_{L \to \infty} \sup_{|t| \le T} \Bigl| \langle n|\chi|n \rangle(t) - \iint \chi(x,k) W_k(t,x) \diff x \diff k\Bigr| = 0.
    \end{equation}
\end{theorem}

The results of Theorem~\ref{thm:hom-inhom} can be summarized as follows and also in Table~\ref{tab:scaling-homogeneity}.
If $\beta = \infty$, then we are in the periodic case and the effective dynamics are homogeneous and independent of $x$.
When $\beta = \alpha$, i.e., $\Tk = \Tt$, one sees the transport of wavepackets and the effective dynamics is clearly inhomogeneous.
When $\alpha < \beta$, i.e., $\Tk < \Tt$, even though one does not observe the transport of wavepackets, the effective dynamics is not homogeneous.
In that case, the trace of $W_k$ at each position satisfies the homogeneous wave kinetic equation (i.e., the one without the transport term), but the initial condition depends on the position.
We shall therefore refer to this scenario as the semi-homogeneous setting.

\begin{table}[htb!]
    \begin{tabular}{|c|c|c|c|}
        \hline
        Scaling law & $\beta=\alpha$ & $\alpha < \beta < \infty$ & $\beta = \infty$ \\
        \hline
        Collision & Yes & Yes & Yes \\
        \hline
        Transport & Yes & No & No \\
        \hline
        Space homogeneity & Inhomogeneous & Semi-homogeneous & Homogeneous\\
        \hline
    \end{tabular}
    \caption{Relation between scaling law and space homogeneity}
    \label{tab:scaling-homogeneity}
\end{table}

In the limit~\eqref{eq:A-W-convergence}, we take the limit $n\to\infty$ at the very end so as to avoid the infinite chaos divergence discussed in \S\ref{sec:wick-renorm}.
In fact, the proof will show that almost all terms in the Dyson series of $\mathbb{E}|\Pi_n A_k|^2$ vanish after letting $L\to\infty$.
The remaining terms sum up to the $n$-th Duhamel iterate of WK and thus converge to its solution as $n\to\infty$.
It is natural to consider the weak convergence of the energy spectrum $\mathbb{E}\mathcal{W}[\Pi_n u]$ and it is generally impossible to give the weak convergence~\eqref{eq:E-spec-wigner-lim} a strong sense.
Indeed, as $\epsilon \to 0$, the support of the wavepacket ensemble shrinks---in the momentum space---to the lattice $\mathbb{Z}^d_L$;
equivalently, as per~\eqref{eq:E-spec-wp-expansion}, one immediately observes that this convergence involves a measure theoretical approximation of $\mathbb{E}\mathcal{W}[\Pi_n u]$ by a sum of Dirac masses.

\subsubsection{Finer effective dynamics}
\label{sec:finer-eff-dynamics}

A major novelty of this work is the discovery of a more sophisticated effective dynamics of the energy spectrum.
This is reflected in a refined version of the WK equation, which we believe has not been previously discovered or explored in the literature. 

The kinetic limit exhibited in Theorem~\ref{thm:hom-inhom} involves the following approximation as $\epsilon \to 0$:
\begin{equation*}
    \mathbb{E} \mathcal{W}\bigl[\Pi_n A_{k}(t)\bigr]\Bigl( x,\frac{\xi-k}{\epsilon}\Bigr)\frac{1}{\epsilon^d}
    \approx  \mathbb{E}\bigl|\Pi_n A_k\bigr|^2(t,x) \bm{\delta}_k(\xi),
\end{equation*}
which results in the loss of all frequential information for the profile $A_k$.
Surprisingly, we have discovered that it is possible to retain this information when $\epsilon>0$.
To reveal such finer structures, we need to widen the narrow wavepackets by introducing a subordinate wavenumber
\begin{equation*}
    \zeta = \frac{\xi-k}{\epsilon},
\end{equation*}
and consider the asymptotic dynamics of the energy spectrum $\mathbb{E} \mathcal{W}[\Pi_n A_{k}(t)](x,\zeta)$ instead of the variance $\mathbb{E} |\Pi_n A_k|^2(t,x)$.
Equivalently, this amounts to test the solution $u$ against observables of the form $\chi(x,\xi,(\xi-k)/\epsilon)$.
It is worth noting that while $\mathbb{E}|\Pi_n A_k|^2$ describes the mass distribution of $\Pi_n u$ in the phase space $\mathbb{R}^{2d}$, the energy spectrum $\mathbb{E} \mathcal{W}[\Pi_n A_k]$ captures a more detailed mass distribution in $\mathbb{R}^{3d}$ by performing a \emph{second microlocalization} at the wavenumber $k \in \mathbb{Z}^d_L$.
We shall show that such second microlocalization can be retained in the kinetic limit, and is reflected in the following WK-2 equation as a refinement to WK:
\begin{equation}\tag{WK-2}
    \label{eq:WK2}
    (\partial_t + \bm{1}_{\alpha=\beta} k \cdot \nabla) E_{k,\zeta}
    = 2 \int_{\mathscr{D}_k}  \bm{\delta}(\Omega(\bm{k})) \Bigl( \int_{\mathscr{D}_\zeta}  \prod_{1 \le j \le 3} E_{k_j,\zeta_j} \diff \bm{\zeta} \Bigr) \diff{\bm{k}}, \quad
    (k,\zeta) \in \mathbb{R}^{2d},
    \index{Functions and random variables!Kinetic limits!2nd microlocal kinetic limit $E_{k,\zeta}$}
\end{equation}
where $E_{k,\zeta}$ is a function of $(t,x)$.
Note that in WK-2, the collision integral describes not only collisions in terms of the wavenumber variable $k$, but also in terms of the subordinate wavenumber variable $\zeta$. As such, it describes a more elaborate energy transfer mechanism among wavepackets.
Of course, this is a phenomena unique to ensembles of wavepackets that are not plane waves.

\begin{theorem}
    \label{thm:inhom-2nd}
    Let $d \ge 3$, let $\alpha\in(0,2) $, $\beta \in [\alpha,\infty]$, and assume that $\widehat{\phi} \in C_c^\infty(\mathbb{R}^{2d})$.
    Let $(A_k)_{k \in \mathbb{Z}^d_L}$ be the global solution to~\eqref{eq:WNLS-wp-prof} given by Theorem~\ref{thm:cauchy}.
    Then there exists $T > 0$ and a smooth map $[-T,T] \times \mathbb{R}^{3d} \ni (t,x,k,\zeta) \mapsto E_{k,\zeta}(t,x)$, such that under the scaling law~\eqref{eq:scaling-law}, there holds
    \begin{equation}
        \label{eq:A-E-convergence}
        \lim_{n \to \infty} \lim_{L \to \infty} \sup_{|t| \le T} \sup_{k \in \mathbb{Z}^d_L} \sup_{\zeta \in \mathbb{R}^d} \bigl\|\mathbb{E}\mathcal{W}\bigl[\Pi_n A_k\bigl(t/\lambda^2\bigr)\bigr](\cdot,\zeta) - E_{k,\zeta}(t, \cdot) \bigr\|_{L^\infty} = 0.
    \end{equation}
    Moreover $(E_{k,\zeta})_{(k,\zeta)\in \mathbb{R}^{2d}}$ solves WK-2 with initial condition
    \begin{equation}
        \label{eq:ini-WK2}
        E_{k,\zeta}(0,x) = \mathcal{W}[\phi(\cdot,k)](x,\zeta).
    \end{equation}
    Next, let $u$ be given by~\eqref{eq:ansatz-wp} and assume in addition that $\beta > 1$.
    For all $\chi \in C_c^\infty(\mathbb{R}^{3d})$, define
    \begin{equation*}
        \langle n|\chi|n \rangle_k(t)
        =  \iint \chi\Bigl(x,\xi,\frac{\xi-k}{\epsilon}\Bigr) \mathbb{E} \mathcal{W}[\Pi_n u]\bigl(t/\lambda^2,x/\epsilon,\xi\bigr) \diff x \diff \xi.
        \index{Functions and random variables!Energy spectra!2nd microlocal observation $\langle n\vert\chi\vert n \rangle_k$}
    \end{equation*}
    Then under the scaling law~\eqref{eq:scaling-law}, there holds
    \begin{equation}
        \label{eq:u-second-ml}
        \lim_{n \to \infty} \lim_{L \to \infty} \sup_{|t| \le T} \sup_{k \in \mathbb{Z}^d_L} \Bigl| L^d
        \langle n|\chi|n \rangle_k(t)
        - \iint \chi(x,k,\zeta) E_{k,\zeta}(t,x) \diff x \diff \zeta \Bigr| = 0.
    \end{equation}    
\end{theorem}

In the limit~\eqref{eq:u-second-ml}, we require the condition $\beta > 1$ so that, after a chaos cutoff, the wavepackets are essentially disjoint in the momentum space. This allows us to extract the information of any wavepacket by second microlocalizing $u$ at the corresponding wavenumber. 
From a physical point of view, recall that the case when $\beta=\infty$ corresponds to the homogeneous periodic limit, so one can understand the regime $\beta>1$ as being close to homogeneous.
This disjointedness condition allows one to separate the wavepackets from others in the Fourier space and is vital in the second microlocal procedure that extracts information of $u$ from that of $A_k$.
One can also consider WK-2 as a lift of WK to $\mathbb{R}^{3d}$, in that WK can be recovered by formally integrating WK-2 in $\zeta$:
\begin{equation*}
    W_k = \int E_{k,\zeta} \diff \zeta.
\end{equation*}
Let us also remark that the weak convergence~\eqref{eq:u-second-ml} can be given a strong sense because, by zooming into the second microlocal regime, the measure theoretical approximation as discussed below Theorem~\ref{thm:hom-inhom} can be more finely characterized as a pointwise approximation of profile functions.
Indeed, our proof suggests that, if $\beta > 1$, then for some $R > 0$,
\begin{equation*}
    \lim_{n\to\infty} \lim_{L\to\infty} \sup_{|t|\le T} \sup_{k \in \mathbb{Z}^d_L} \sup_{x \in \mathbb{R}^d} \sup_{|\zeta| \le nR} |(\epsilon L)^d \mathbb{E} \mathcal{W}[\Pi_n u](t/\lambda^2,x/\epsilon,k+\epsilon \zeta) - E_{k,\zeta}(t,x)|
    = 0.
\end{equation*}

Let us conclude by revisiting Figure~\ref{fig:wk-theory}, which provides a summary and visualization of our main results.
To the best of our knowledge, this is the first attempt to elucidate the connections among various regimes of the wave turbulence theory.
Moreover, our work unveils for the first time the second microlocal kinetic limit WK-2.

\subsection*{Organization of the paper}

In~\S\ref{sec:tree-couple}, we review some combinatorial preliminaries---trees and couples---for the diagrammatic expansion of solutions to WNLS and their energy spectra. 
In~\S\ref{sec:oscillatory}, we analyze some oscillatory functionals which arise from iterating WNLS, and estimate their asymptotic behaviors in the large box limit.
In~\S\ref{sec:lattice-count-est}, we obtain a lattice counting estimate which is essential in proving that most of the terms from the iteration give negligible contributions to the effective dynamics.
In~\S\ref{sec:diagram}, we give diagrammatic expansions to solutions of WNLS and their energy spectra.
In~\S\ref{sec:hom}, we prove Theorem~\ref{thm:hom-inhom} in the homogeneous setting.
In~\S\ref{sec:inhom}, we prove Theorems~\ref{thm:hom-inhom} and~\ref{thm:inhom-2nd} in the inhomogeneous and semi-homogeneous setting.
Readers are also advised to refer to a list of symbols provided at the end of the paper for their convenience.

\section*{Acknowledgement}

The authors were partly supported by a Simons Collaboration Grant on Wave Turbulence.
The first author was also partly supported by NSF grant DMS-1654692. 
The authors thank the anonymous referees for their careful review of the manuscript, and their helpful suggestions to improve the exposition.

\section{Trees and couples}

\label{sec:tree-couple}

In this section, we introduce some combinatorial preliminaries for the Feynman diagrammatic expansion of solutions to WNLS and their energy spectra. 
Similar concepts have been used in previous works such as \cite{Lukkarinen2007,Lukkarinen2009,Deng2021b,Deng2021,Deng2021a}.
However, necessary and nontrivial adaptations shall be made to fit the Wick renormalized equation and the inhomogeneous setting we are working with.
After introducing the notions of trees and couples in \S\ref{sec:tree} for this problem, 
we introduce the new concepts of \emph{conjugate nodes} and \emph{irreducible couples} which are essential in the characterization of \emph{decorations} on couples presented in \S\ref{sec:deco} and in the study of topological structures of couples presented in \S\ref{sec:topo}.

\subsection{Ternary trees}

\label{sec:tree}

In \S\ref{sec:cauchy}, we will solve the Cauchy problem of WNLS by iteratively apply Duhamel's principle.
Due to the cubic nature of the nonlinearity, this iterative scheme can be completely illustrated using ternary trees.

\begin{definition}
    \label{def:tree}
    \emph{A ternary tree} $\tau$\index{Trees and couples!Ternary trees!Ternary tree $\tau$} is a rooted tree where each non-leaf (or branching) node $\mathfrak{b}$\index{Trees and couples!Nodes!Branching node $\mathfrak{b}$} has three children which we distinguish as left, middle, and right or $(\mathfrak{b}[1],\mathfrak{b}[2],\mathfrak{b}[3])$\index{Trees and couples!Nodes!Child nodes $\mathfrak{b}[j]$}.
    We say that a node $\mathfrak{m}$ is a descendent of another node $\mathfrak{n}$ (or that $\mathfrak{n}$ is an ancestor of $\mathfrak{m}$) if $\mathfrak{m}$ belongs to the subtree rooted at $\mathfrak{n}$.
    This gives a strict partial order $\prec$\index{Trees and couples!Ternary trees!Partial order $\prec$} on $\tau$ under which there exists a unique maximum element given by the root node $\mathfrak{r}^\tau$\index{Trees and couples!Nodes!Root node $\mathfrak{r}^\tau$} and for every $\mathfrak{m} \in  \tau \backslash \{\mathfrak{r}^\tau\}$, the set $\{\mathfrak{n}\in \tau : \mathfrak{m} \prec \mathfrak{n}\} $ admits a unique minimum $\mathfrak{m}^p$, which is the parent node $\mathfrak{m}^p$\index{Trees and couples!Nodes!Parent node $\mathfrak{m}^p$} of $\mathfrak{m}$.
    If $\mathfrak{n}\in \tau$ is a minimum node, then $\mathfrak{n}$ is \emph{a leaf node} of $\tau$, otherwise it is \emph{a branching node} of $\tau$.
    The set of all leaf nodes of $\tau$ is denoted by $\mathfrak{L}^\tau$\index{Trees and couples!Nodes!Set of leaf nodes $\mathfrak{L}^\tau$, $\mathfrak{L}^\cp$} and the set of all branching nodes of $\tau$ is denoted by $\mathfrak{B}^\tau$\index{Trees and couples!Nodes!Set of branching nodes $\mathfrak{B}^\tau$, $\mathfrak{B}^\cp$}.
    
    For all $n \in \mathbb{N}$, we define $\mathscr{T}_n^+$ resp.\ $\mathscr{T}_n^-$ to be set of all positive resp.\ negative ternary trees of order $n$ and set $\mathscr{T}_n = \mathscr{T}^+_n \cup \mathscr{T}^-_n$.\index{Trees and couples!Ternary trees!Sets of ternary trees $\mathscr{T}_n$, $\mathscr{T}^\pm_n$, $\mathscr{T}^\pm$}
    We also denote
    \begin{equation*}
        \mathscr{T}^\pm = \bigcup_{n \in \mathbb{N}} \mathscr{T}_n^\pm, \quad
        \mathscr{T} = \bigcup_{n \in \mathbb{N}} \mathscr{T}_n.
    \end{equation*}
\end{definition}

\begin{definition}
    \label{def:ter-tree-signed}
    \emph{A sign function}\index{Trees and couples!Ternary trees!Sign function $\iota$} on a ternary tree $\tau$ is a map $\iota : \tau \to \{+1,-1\}$ such that $\iota_{\mathfrak{b}[j]} = \iota_{\mathfrak{b}} (-1)^{j+1}$ for all $\mathfrak{b} \in \mathfrak{B}^\tau$ and for all $1 \le j \le 3$.
    A ternary tree equipped with a sign function is \emph{a signed ternary tree}.
    If $\tau$ is a signed ternary tree and $\mathfrak{n} \in \tau$, then $\mathfrak{n}$ is positive resp.\ negative if $\iota_{\mathfrak{n}} = +1$ resp.\ $\iota_{\mathfrak{n}} = -1$.
    The sign of $\tau$ is $\iota_\tau = \iota_{\mathfrak{r}^\tau}$.
    The ternary tree $\tau$ is positive resp.\ negative ternary if $\iota_\tau = +1$ resp.\ $\iota_\tau = -1$.
    In addition, we define the \emph{polarity}\index{Trees and couples!Ternary trees!Polarity $\varsigma_\tau$} of $\tau$ to be the complex number
    \begin{equation*}
        \varsigma_\tau = \prod_{\mathfrak{b} \in \mathfrak{B}^\tau} i \iota_{\mathfrak{b}}.
    \end{equation*}
\end{definition}

In Table~\ref{tab:eg-T-n} we list all elements in $\mathscr{T}^\pm_n$ when $n = 1,2,3$.
In these figures, we use a solid bullet $\bullet$ denote a positive node and use a hollow circle $\circ$ to denote a negative node.
For every branching node, its three children numbered $1,2,3$ are respectively the left node, the middle node, and the right node below it that are directly connected to it.
\begin{table}[htb!]
    \begin{tabular}{c|c|c|c}
        $\mathscr{T}^\sigma_n$ & $n=0$ & $n=1$ & $n=2$ \\
        \hline
        $\sigma = +$ &
        \begin{tikzpicture}[baseline=(root),scale=0.5,font=\footnotesize]
            \node(root)[Solid]{};
        \end{tikzpicture}
        &
        \begin{tikzpicture}[baseline=(root),scale=0.5,font=\footnotesize]
            \node(root)[Solid]{}
                child {node[Solid]{}}
                child {node[Hollow]{}}
                child {node[Solid]{}};
        \end{tikzpicture}
        &
        \begin{tikzpicture}[baseline=(root),scale=0.5,font=\footnotesize]
            \node(root)[Solid]{}
                child {node[Solid]{}
                    child {node[Solid]{}}
                    child {node[Hollow]{}}
                    child {node[Solid]{}}}
                child {node[Hollow]{}}
                child {node[Solid]{}};
        \end{tikzpicture}
        \begin{tikzpicture}[baseline=(root),scale=0.5,font=\footnotesize]
            \node(root)[Solid]{}
                child {node[Solid]{}}
                child {node[Hollow]{}
                    child {node[Hollow]{}}
                    child {node[Solid]{}}
                    child {node[Hollow]{}}}
                child {node[Solid]{}};
        \end{tikzpicture}
        \begin{tikzpicture}[baseline=(root),scale=0.5,font=\footnotesize]
            \node(root)[Solid]{}
                child {node[Solid]{}}
                child {node[Hollow]{}}
                child {node[Solid]{}
                    child {node[Solid]{}}
                    child {node[Hollow]{}}
                    child {node[Solid]{}}
                };
        \end{tikzpicture}\\[4.2em]
        \hline
        $\sigma = -$ &
        \begin{tikzpicture}[baseline=(root),scale=0.5,font=\footnotesize]
            \node(root)[Hollow]{};
        \end{tikzpicture}
        &
        \begin{tikzpicture}[baseline=(root),scale=0.5,font=\footnotesize]
            \node(root)[Hollow]{}
                child {node[Hollow]{}}
                child {node[Solid]{}}
                child {node[Hollow]{}};
        \end{tikzpicture}
        &
        \begin{tikzpicture}[baseline=(root),scale=0.5,font=\footnotesize]
            \node(root)[Hollow]{}
                child {node[Hollow]{}
                    child {node[Hollow]{}}
                    child {node[Solid]{}}
                    child {node[Hollow]{}}}
                child {node[Solid]{}}
                child {node[Hollow]{}};
        \end{tikzpicture}
        \begin{tikzpicture}[baseline=(root),scale=0.5,font=\footnotesize]
            \node(root)[Hollow]{}
                child {node[Hollow]{}}
                child {node[Solid]{}
                    child {node[Solid]{}}
                    child {node[Hollow]{}}
                    child {node[Solid]{}}}
                child {node[Hollow]{}};
        \end{tikzpicture}
        \begin{tikzpicture}[baseline=(root),scale=0.5,font=\footnotesize]
            \node(root)[Hollow]{}
                child {node[Hollow]{}}
                child {node[Solid]{}}
                child {node[Hollow]{}
                    child {node[Hollow]{}}
                    child {node[Solid]{}}
                    child {node[Hollow]{}}
                };
        \end{tikzpicture}
    \end{tabular}
    \caption{The set $\mathscr{T}^\pm_n$ when $n=0,1,2$}
    \label{tab:eg-T-n}
\end{table}

\begin{definition}[New trees from old]
    \label{def:tree-str}
    If $\tau,\tau'\in \mathscr{T}$ and $\mathfrak{l} \in \mathfrak{L}^\tau$ such that $\iota_{\mathfrak{l}} = \iota_{\tau'}$, then by attaching $\mathfrak{r}^{\tau'}$ to $\mathfrak{l}$ we obtain a new signed ternary tree $\tau \otimes_{\mathfrak{l}} \tau'$.
    Conversely, if $\tau \in \mathscr{T}$ and $\mathfrak{n} \in \tau$, then $\check{\tau}_\mathfrak{n} = \{\mathfrak{m} \in \tau : \mathfrak{m} \preceq \mathfrak{n}\}$ and $ \hat{\tau}_{\mathfrak{n}} = \{\mathfrak{m} \in \tau : \mathfrak{m} \not\prec \mathfrak{n}\}$\index{Trees and couples!Ternary trees!Subtrees $\hat{\tau}_{\mathfrak{n}}$, $\check{\tau}_{\mathfrak{n}}$} are both signed ternary trees and we have $\tau = \check{\tau}_{\mathfrak{n}} \otimes_{\mathfrak{n}} \hat{\tau}_{\mathfrak{n}}$. See Figure \ref{fig:tree-planting}(A).

    For all $\tau_1,\tau_2,\tau_3 \in \mathscr{T}$ with $(\iota_{\tau_1},\iota_{\tau_2},\iota_{\tau_3}) = (\pm 1, \mp 1, \pm 1)$, there exists a unique $\tau \coloneqq \otimes (\tau_1, \tau_2, \tau_3) \in \mathscr{T}^{\pm}$\index{Trees and couples!Ternary trees!Tree product $\oplus$} such that $(\tau_{\mathfrak{r}^\tau[1]},\tau_{\mathfrak{r}^\tau[2]},\tau_{\mathfrak{r}^\tau[3]}) = (\tau_1,\tau_2,\tau_3)$.
    Moreover, if $\tau_j \in \mathscr{T}_{n_j}$ for $1 \le j \le 3$, then $\tau \in \mathscr{T}^\pm_{n}$ where $n = n_1+n_2+n_3+1$. See Figure \ref{fig:tree-planting}(B).
\end{definition}

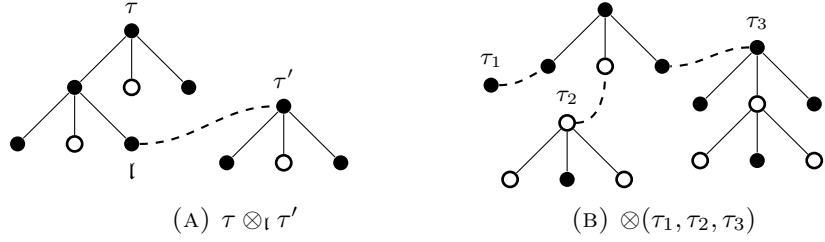
\begin{figure}[htb!]
    \begin{subfigure}{0.35\textwidth}
        \begin{tikzpicture}[baseline=(root),scale=0.5,font=\footnotesize]
            \node(root)[Solid,label=above:{$\tau$}]{}
                    child {node[Solid]{}
                        child {node[Solid]{}}
                        child {node[Hollow]{}}
                        child {node(l) [Solid,label=below:{$\mathfrak{l}$}]{}}}
                    child {node[Hollow]{}}
                    child {node [Solid]{}};
            \node(r) at (4,-2) [Solid,label=above:{$\tau'$}]{}
                    child {node[Solid]{}}
                    child {node[Hollow]{}}
                    child {node[Solid]{}};
            \draw[dashed,thick] (l) to[out=0,in=180] (r);
        \end{tikzpicture}
        \caption{$\tau \otimes_{\mathfrak{l}} \tau'$}
    \end{subfigure}
    \begin{subfigure}{0.3\textwidth}
        \begin{tikzpicture}[baseline=(root),scale=0.5,font=\footnotesize]
            \node(root) at (0,0) [Solid]{}
                child {node(c1)[Solid]{}}
                child {node(c2)[Hollow]{}}
                child {node(c3)[Solid]{}};
            \node(r1) at (-3,-2) [Solid,label=above:{$\tau_1$}]{};
            \node(r2) at (-1,-3) [Hollow,label=above:{$\tau_2$}]{}
                child {node[Hollow]{}}
                child {node[Solid]{}}
                child {node[Hollow]{}};
            \node(r3) at (4,-1) [Solid,label=above:{$\tau_3$}]{}
                    child {node[Solid]{}}
                    child {node[Hollow]{}
                        child {node[Hollow]{}}
                        child {node[Solid]{}}
                        child {node[Hollow]{}}}
                    child {node[Solid]{}};
            \draw[dashed,thick] (r1) to[out=0,in=210]  (c1);
            \draw[dashed,thick] (r2) to[out=0,in=-90]  (c2);
            \draw[dashed,thick] (r3) to[out=180,in=0]  (c3);
        \end{tikzpicture}
        \caption{$\otimes(\tau_1,\tau_2,\tau_3)$}
    \end{subfigure}
    \caption{Tree planting}
    \label{fig:tree-planting}
\end{figure}

\subsection{Couples}

\label{sec:couple}

When computing the energy spectrum of the solution, we use Wick's theorem~\eqref{eq:wick-identity} and reduce the calculations to all possible pairings between ternary trees.
This leads to the definition of couples.
As a consequence of the Wick renormalization, no self-interactions shall appear in pairings, in contrast to the non-renormalized model studied in~\cite{Deng2021b,Deng2021}.

\begin{definition}
    \label{def:pairing}
    For all $n \in \mathbb{N}$, we denote $\mathscr{T}^*_n = \mathscr{T}^+_n \times \mathscr{T}^-_n$\index{Trees and couples!Tree pairs!Sets of tree pairs $\mathscr{T}^*_n$, $\mathscr{T}^*$}.
    We also denote
    $\mathscr{T}^* = \bigcup_{n \in \mathbb{N}} \mathscr{T}^*_n$\index{Trees and couples!Tree pairs!Tree pair $\tau^*$}.
    If $\tau^* = (\tau^+,\tau^-) \in \mathscr{T}^*$, then we denote $\mathfrak{L}^{\tau^*} = \mathfrak{L}^{\tau^+} \cup \mathfrak{L}^{\tau^-}$, $\mathfrak{B}^{\tau^*} = \mathfrak{B}^{\tau^+} \cup \mathfrak{B}^{\tau^-}$ and $\mathfrak{r}^{\tau^*} = \{\mathfrak{r}^{\tau^+},\mathfrak{r}^{\tau^-}\}$.
    We also denote $\varsigma_{\tau^*} = \varsigma_{\tau^+} \varsigma_{\tau^-}$\index{Trees and couples!Tree pairs!Polarity for tree pairs $\varsigma_{\tau^*}$}.
    A \emph{pairing} $\wp$\index{Trees and couples!Couples!Pairing $\wp$} of $\tau^* \in \mathscr{T}^*$ is a partition of $\mathfrak{L}^{\tau^*}$ into doubletons, which are called \emph{leaf pairs}, such that for any $\mathfrak{p}\in \wp$, there holds that $\mathfrak{p} \not\subset \mathfrak{L}^{\tau^\pm}$ and $\iota(\mathfrak{p}) = \{+1,-1\}$ for all $\mathfrak{p} \in \wp$. 
    The set of all pairings of $\tau^*$ is denoted by $\mathfrak{P}^{\tau^*}$\index{Trees and couples!Couples!Set of pairings $\mathfrak{P}^{\tau^*}$}.
\end{definition}

\begin{definition}
    \label{def:cp}
    A \emph{couple} is an element $\cp = (\tau^*,\wp)$\index{Trees and couples!Couples!Couple $\cp$} with $\tau^* \in \mathscr{T}^*$ and $\wp \in \mathfrak{P}_{\tau^*}$.
    If $\tau^* \in \mathscr{T}^*_n$, then $\cp$ has \emph{order}~$n$.
    There is a unique couple with order zero and it is called the trivial couple.
    The set of all couples with order $n$ is denoted by $\mathscr{K}_n$\index{Trees and couples!Couples!Sets of couples $\mathscr{K}_n$, $\mathscr{K}_n$}. 
    We also denote $\mathscr{K} = \bigcup_{n \in \mathbb{N}} \mathscr{K}_{n}.$
    If $\cp = (\tau^*,\wp) \in \mathscr{K}$, then we shall denote $\mathfrak{B}^\cp = \mathfrak{B}^{\tau^*}$\index{Trees and couples!Couples!Sets of branching nodes for couples $\mathfrak{B}^\cp$}, $\mathfrak{L}^\cp = \mathfrak{L}^{\tau^*}$\index{Trees and couples!Couples!Sets of leaf nodes for couples $\mathfrak{L}^\cp$} and $\varsigma_\cp = \varsigma_{\tau^*}$\index{Trees and couples!Couples!Polarity for couples $\varsigma_\cp$}.
    By abusing the notation, we shall also use $\cp$ to denote the finite set $ \tau^+ \cup \tau^-$ when there is no ambiguity.
\end{definition}

In Table~\ref{tab:eg-cp}, we list all elements in $\mathscr{K}_0$ and $\mathscr{K}_1$.
In these figures, we use dotted curves to connect leaf nodes paired to each other.
\begin{table}[htb!]
    \begin{tabular}{c|c}
        $\mathscr{K}_0$ & $\mathscr{K}_1$\\
        \hline
        \begin{tikzpicture}[baseline=(r+),scale=0.5,font=\footnotesize]
            \node(r+) at (0,0) [Solid]{};
            \node(r-) at (2,0) [Hollow]{};
            \draw[-,dotted] (r+) to[out=-45,in=225] (r-);
        \end{tikzpicture}
        & \begin{tikzpicture}[baseline=(r+),scale=0.5,font=\footnotesize]
            \node(r+) at (-1,-7) [Solid]{}
                child {node(a+)[Solid]{}}
                child {node(b-)[Hollow]{}}
                child {node(c+)[Solid]{}};
            \node(r-) at (3.5,-7) [Hollow]{}
                child {node(a-)[Hollow]{}}
                child {node(b+)[Solid]{}}
                child {node(c-)[Hollow]{}};
            \draw[-,dotted] (a+) to[out=-45,in=225] (a-);
            \draw[-,dotted] (b-) to[out=-45,in=225] (b+);
            \draw[-,dotted] (c+) to[out=-45,in=225] (c-);
        \end{tikzpicture}
        \begin{tikzpicture}[baseline=(r+),scale=0.5,font=\footnotesize]
            \node(r+) at (-1,-7) [Solid]{}
                child {node(a+)[Solid]{}}
                child {node(b-)[Hollow]{}}
                child {node(c+)[Solid]{}};
            \node(r-) at (3.5,-7) [Hollow]{}
                child {node(c-)[Hollow]{}}
                child {node(b+)[Solid]{}}
                child {node(a-)[Hollow]{}};
            \draw[-,dotted] (a+) to[out=-30,in=210] (a-);
            \draw[-,dotted] (b-) to[out=-30,in=210] (b+);
            \draw[-,dotted] (c+) to[out=-30,in=210] (c-);
        \end{tikzpicture}
    \end{tabular}
    \caption{The set $\mathscr{K}_n$ when $n=0,1$}
    \label{tab:eg-cp}
\end{table}
\subsubsection{Conjugate nodes and irreducible couples}
There is a factorization structure for the set of all couples, which allows us to form general couples from \emph{irreducible couples}.
In fact, as we shall see in Lemma~\ref{lem:cp-attaching}, by attaching a couple to a leaf pair $\mathfrak{p}$ of another couple, one obtains a new couple.
In this new couple, elements of $\mathfrak{p}$ may no longer be leaves but we can still consider them as ``paired'' in the sense that they exhibit some behaviors that all leaf pairs have.
To make precise of this observation, we introduce \emph{conjugate nodes}.

\begin{definition}
    \label{def:conju-node}
    Let $\cp = (\tau^*,\wp) \in \mathscr{K}$.
    For all $\mathfrak{p} = \{\mathfrak{l}^+,\mathfrak{l}^-\}\in \wp$, denote by $\mathscr{A}_{\mathfrak{p}}$ the set of all $\mathfrak{n} \in \cp$ such that either $\mathfrak{n} \succeq \mathfrak{l}^+$ or $\mathfrak{n} \succeq \mathfrak{l}^-$.
    Two nodes $\mathfrak{m},\mathfrak{n} \in \cp$ are \emph{conjugate} and we denote $\mathfrak{m} \sim \mathfrak{n}$\index{Trees and couples!Couples!Conjugate relation $\sim$} if for all $\mathfrak{p} \in \wp$ we have either $\{\mathfrak{m},\mathfrak{n}\} \subset \mathscr{A}_{\mathfrak{p}}$ or $\{\mathfrak{m},\mathfrak{n}\} \subset \cp \backslash \mathscr{A}_{\mathfrak{p}}$.
\end{definition}

\begin{lemma}
    \label{lem:conju-node-char}
    If $\cp = (\tau^*,\wp) \in \mathscr{K}$ and $\mathfrak{m},\mathfrak{n} \in \cp$ such that $\mathfrak{m} \ne \mathfrak{n}$ and $\mathfrak{m} \sim \mathfrak{n}$, then $\mathfrak{m}$ and $\mathfrak{n}$ belong to different trees and have opposite signs.
    Moreover, if $\mathfrak{h} \in \cp$ such that $\mathfrak{m} \sim \mathfrak{n} \sim \mathfrak{h}$, then $\mathfrak{h} \in \{\mathfrak{m},\mathfrak{n}\}$.
    Therefore, every conjugate class contains either one or two nodes.
\end{lemma}
\begin{proof}
    Because $\mathfrak{m} \ne \mathfrak{n}$, there exists $\mathfrak{l} \in \mathfrak{L}^\cp$ such that $\mathfrak{l} \prec \mathfrak{m}$ and $\mathfrak{l} \not\prec \mathfrak{n}$.
    Let $\mathfrak{p} = \{\mathfrak{l},\mathfrak{l}'\} \in \wp$.
    Because $\mathfrak{m} \sim \mathfrak{n}$ and $\mathfrak{m} \in \mathscr{A}_{\mathfrak{p}}$, we have $\mathfrak{n} \in \mathscr{A}_{\mathfrak{p}}$.
    Since $\mathfrak{l} \not\prec \mathfrak{n}$, we must have $\mathfrak{l}' \in \mathfrak{n}$.
    Now that $\mathfrak{l}$ and $\mathfrak{l}'$ belong to different trees, the nodes $\mathfrak{m}$ and $\mathfrak{n}$ must also belong to different trees.
    Therefore the pairing $\wp$ assigns every leaf $\mathfrak{l} \prec \mathfrak{m}$ uniquely to a leaf $\mathfrak{l}' \in \mathfrak{n}$ with an opposite sign and vice versa.
    We then deduce that $\mathfrak{m}$ and $\mathfrak{n}$ must have opposite signs.
    If $\mathfrak{m} \sim \mathfrak{n} \sim \mathfrak{h}$, then at least two nodes among $\mathfrak{m}$,$\mathfrak{n}$, and $\mathfrak{h}$ have the same sign and these two nodes must therefore be equal.
\end{proof}

\begin{definition}
    If $\cp \in \mathscr{K}$, then $\mathfrak{C}^\cp$ is the set of all conjugate classes of $\cp$ with $\pi^\cp : \cp \to \mathfrak{C}^\cp$ being the canonical quotient map.
    We decompose $\mathfrak{C}^\cp = \mathfrak{C}^\cp_1 \cup \mathfrak{C}^\cp_2$ where $\mathfrak{C}^\cp_1$ consists of singletons and $\mathfrak{C}^\cp_2$ consists of doubletons.\index{Trees and couples!Couples!Conjugate class $\mathfrak{C}^\cp$, $\mathfrak{C}^\cp_1$, $\mathfrak{C}^\cp_2$}
\end{definition}

\begin{lemma}
    \label{lem:cp-attaching}
    Let $\cp_j = (\tau^*_j,\wp_j) \in \mathscr{K}$ with $j \in \{0,1\}$ and let $\mathfrak{p} = \{\mathfrak{l}^+,\mathfrak{l}^-\} \in \wp_0$.
    Let $\sigma_\pm = \pm$ if $\mathfrak{l}^\pm \in \tau_0^\pm$ and let $\sigma_\pm = \mp$ if $\mathfrak{l}^\pm \in \tau_0^\mp$.
    By attaching (and thus identifying) $\mathfrak{r}^{\tau_1^+}$ with $\mathfrak{l}^+$ and $\mathfrak{r}^{\tau_1^-}$ with $\mathfrak{l}^-$, we obtain a new couple $\cp_0 \otimes_{\mathfrak{p}} \cp_1 = (\tau^*,\wp)$ with $\mathfrak{C}^{\cp_0 \otimes_{\mathfrak{p}} \cp_1} = \mathfrak{C}^{\cp_0} \cup \mathfrak{C}^{\cp_1}$\index{Trees and couples!Couples!Couple product $\otimes_{\mathfrak{p}}$} by letting
    \begin{equation*}
        \tau^\pm = \tau_0^\pm \otimes_{\mathfrak{l}^{\sigma_\pm}} \tau_1^{\sigma_{\pm}}, \quad
        \wp =\wp_0 \cup \wp_1 \backslash \{\mathfrak{p}\}.
    \end{equation*}
    Conversely, let $\cp = (\tau^*,\wp) \in \mathscr{K}$ and let $\mathfrak{c} = \{\mathfrak{n}^+,\mathfrak{n}^-\} \in \mathfrak{C}^\cp_2$.
    Let $\sigma_\pm = \pm$ if $\mathfrak{n}^\pm \in \tau^\pm$ and let $\sigma_\pm = \mp$ if $\mathfrak{n}^\pm \in \tau^\mp$.
    Denote $\hat{\tau}^\pm = \hat{\tau}^{\sigma_\pm}_{\mathfrak{n}^\pm}$ and $\check{\tau}^\pm = \check{\tau}^{\sigma_\pm}_{\mathfrak{n}^\pm} $ and let
    \begin{equation*}
        \hat{\cp}_{\mathfrak{c}}
        = (\hat{\tau}^*,\wp \backslash 2^{\mathfrak{L}^{\hat{\tau}^*}} \cup \{\mathfrak{c}\}), \quad
        \check{\cp}_{\mathfrak{c}}
        = (\check{\tau}^*,\wp \cap 2^{\mathfrak{L}^{\hat{\tau}^*}})
    \end{equation*}
    Then $\hat{\cp}_{\mathfrak{c}}, \check{\cp}_{\mathfrak{c}} \in \mathscr{K}$ and $\cp = \check{\cp}_{\mathfrak{c}} \otimes_{\mathfrak{c}} \hat{\cp}_{\mathfrak{c}}$.
\end{lemma}
\begin{proof}
    The first half of the lemma is self explanatory.
    To prove the second half, it suffices to observe that for all $\mathfrak{p} \in \wp$, we have either $\mathfrak{p} \in 2^{\mathfrak{L}^{\hat{\tau}^*}}$ or $\mathfrak{p} \in \wp \backslash 2^{\mathfrak{L}^{\hat{\tau}^*}}$.
\end{proof}

\begin{definition}
    \label{def:cp-irreducible}
    A couple $\cp = (\tau^*,\wp) \in \mathscr{K} \backslash \mathscr{K}_0$ is \emph{irreducible} if $\mathfrak{C}^\cp_2 = \{\mathfrak{r}^{\tau^*}\} \cup \wp$.
    The set of irreducible subcouples of $\cp \in \mathscr{K}$, denoted by $\mathscr{I}_{\cp}$\index{Trees and couples!Couples!Sets of irreducible subcouples $\mathscr{I}_\cp$} is defined recursively as follows:
    \begin{enumerate}
        \item If $\cp \in \mathscr{K}_0$, then $\mathscr{I}_\cp = \emptyset$.
        \item If $\cp$ is irreducible, then $\mathscr{I}_\cp = \{\cp\}$.
        \item If $\cp = \cp_1 \otimes_{\mathfrak{c}} \cp_2$ for some $\mathfrak{c} \subset \mathfrak{C}^{\cp}_2$, then $\mathscr{I}_\cp = \mathscr{I}_{\cp_1} \cup \mathscr{I}_{\cp_2}$.
    \end{enumerate}
\end{definition}

One verifies that $\mathscr{I}_\cp$ is well-defined because it is independent of the choice of $\mathfrak{c} \in \mathfrak{C}^\cp_2$.
In Figure~\ref{fig:cp-fac} we give two examples of couples factorized into irreducible subcouples.
In these examples, nodes connected by dashed lines are attached together and nodes connected by dotted lines are conjugate to each other.

\begin{figure}[htb!]
    \centering
    \begin{subfigure}{0.4\textwidth}
        \begin{tikzpicture}[baseline=(root+),scale=0.5,font=\footnotesize]
            \node(root+)[Solid]{}
                child {node[Solid]{}
                    child {node(n6+)[Solid]{}}
                    child {node(n5-)[Hollow]{}}
                    child {node[Solid]{}
                        child {node(p+)[Solid]{}}
                        child {node(n4-)[Hollow]{}}
                        child {node(n3+)[Solid]{}}}}
                child {node(n2-)[Hollow]{}}
                child {node(n1+)[Solid]{}};
            \node(root-) at (4.5,0) [Hollow]{}
                child {node(n1-)[Hollow]{}}
                child {node[Solid]{}
                    child {node(n2+)[Solid]{}}
                    child {node(n3-)[Hollow]{}}
                    child {node[Solid]{}
                        child {node(n5+)[Solid]{}}
                        child {node(p-)[Hollow]{}}
                        child {node(n4+)[Solid]{}}}}
                child {node(n6-)[Hollow]{}};
            \node(r+) at (0,-6.5) [Solid]{}
                child {node(a+)[Solid]{}}
                child {node(b-)[Hollow]{}}
                child {node(c+)[Solid]{}};
            \node(r-) at (4.5,-6.5) [Hollow]{}
                child {node(a-)[Hollow]{}}
                child {node(b+)[Solid]{}}
                child {node(c-)[Hollow]{}};
            \draw[dotted] (n1+) to (n1-);
            \draw[dotted] (n2-) to (n2+);
            \draw[dotted] (n3+) to (n3-);
            \draw[dotted] (n4-) to[out=-30,in=210]  (n4+);
            \draw[dotted] (n5-) to[out=30,in=150]  (n5+);
            \draw[dotted] (n6+) to[out=75,in=165]  (n6-);
            \draw[dotted] (root+) to[out=15,in=165] (root-);
            \draw[dotted] (r+) to[out=15,in=165] (r-);
            \draw[dotted] (p+) to[out=-30,in=210] (p-);
            \draw[dashed,thick] (r+) to[out=135,in=-90] (p+);
            \draw[dashed,thick] (r-) to[out=45,in=-90] (p-);
            \draw[dotted] (a+) to[out=-30,in=210] (a-);
            \draw[dotted] (b-) to[out=-30,in=210]  (b+);
            \draw[dotted] (c+) to[out=-30,in=210] (c-);
        \end{tikzpicture}
        \caption{Conjugate nodes}
        \label{fig:conju-nodes}
    \end{subfigure}
    \begin{subfigure}{0.55\textwidth}
        \begin{tikzpicture}[baseline=(root+),scale=0.5,font=\footnotesize]
            \node(r+) at (0,0) [Solid]{}
                child {node(a+)[Solid]{}}
                child {node(b-)[Hollow]{}}
                child {node(c+)[Solid]{}};
            \node(r-) at (4.5,0) [Hollow]{}
                child {node(c-)[Hollow]{}}
                child {node(b+)[Solid]{}}
                child {node(a-)[Hollow]{}};

            \node(r1+) at (-3.75,-3) [Solid]{}
                child {node(a1+)[Solid]{}}
                child {node(b1-)[Hollow]{}}
                child {node(c1+)[Solid]{}};
            \node(r1-) at (0.75,-3) [Hollow]{}
                child {node(a1-)[Hollow]{}}
                child {node(b1+)[Solid]{}}
                child {node(c1-)[Hollow]{}};

            \node(r2+) at (4,-5) [Solid]{}
                child {node(a2+)[Solid]{}}
                child {node(b2-)[Hollow]{}}
                child {node(c2+)[Solid]{}};
            \node(r2-) at (8.5,-5) [Hollow]{}
                child {node(a2-)[Hollow]{}}
                child {node(b2+)[Solid]{}}
                child {node(c2-)[Hollow]{}};

            \node(r3+) at (-3.75,-6.5) [Solid]{}
                child {node(a3+)[Solid]{}}
                child {node(b3-)[Hollow]{}}
                child {node(c3+)[Solid]{}};
            \node(r3-) at (0.75,-6.5) [Hollow]{}
                child {node(c3-)[Hollow]{}}
                child {node(b3+)[Solid]{}}
                child {node(a3-)[Hollow]{}};

            \draw[dashed,thick] (r1+) to[out=60,in=180] (a+);
            \draw[dashed,thick] (r1-) to[out=0,in=-120] (a-);
            \draw[dashed,thick] (r2+) to[out=45,in=-45] (b+);
            \draw[dashed,thick] (r2-) to[out=150,in=-30] (b-);
            \draw[dashed,thick] (r3+) to[out=150,in=-90] (a1+);
            \draw[dashed,thick] (r3-) to[out=150,in=-90] (a1-);

            \draw[dotted] (r+) to[out=15,in=165] (r-);
            \draw[dotted] (a+) to[out=-30,in=210] (a-);
            \draw[dotted] (b-) to[out=-30,in=210]  (b+);
            \draw[dotted] (c+) to[out=-30,in=210] (c-);
            \draw[dotted] (r1+) to[out=15,in=165] (r1-);
            \draw[dotted] (a1+) to[out=-30,in=210] (a1-);
            \draw[dotted] (b1-) to[out=-30,in=210]  (b1+);
            \draw[dotted] (c1+) to[out=-30,in=210] (c1-);
            \draw[dotted] (r2+) to[out=15,in=165] (r2-);
            \draw[dotted] (a2+) to[out=-30,in=210] (a2-);
            \draw[dotted] (b2-) to[out=-30,in=210]  (b2+);
            \draw[dotted] (c2+) to[out=-30,in=210] (c2-);
            \draw[dotted] (r3+) to[out=15,in=165] (r3-);
            \draw[dotted] (a3+) to[out=-30,in=210] (a3-);
            \draw[dotted] (b3-) to[out=-30,in=210]  (b3+);
            \draw[dotted] (c3+) to[out=-30,in=210] (c3-);
        \end{tikzpicture}
        \caption{A regular couple}
        \label{fig:reg-cp}
    \end{subfigure}
    \caption{Examples of couple factorization}
    \label{fig:cp-fac}
\end{figure}
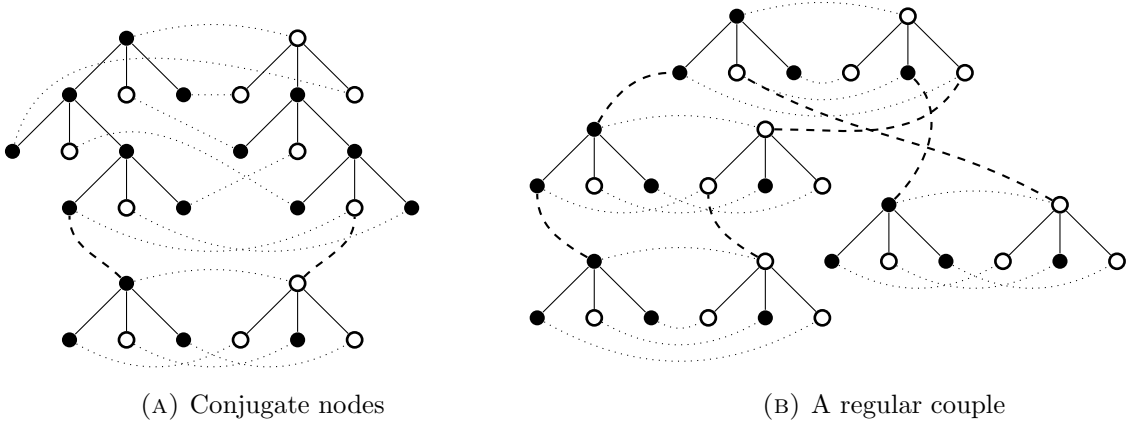

\subsubsection{Regular couples}

In~\cite{Deng2021}, regular couples were defined as the leading diagrams of the Feynman diagrammatic expansions, which makes them the only couples that contribute to the wave kinetic equations. In fact, it is shown that couples that are not regular either have a negligible contribution or cancel among themselves. For WNLS, only couples that are not self-interactive are taken into considerations.
Therefore, in this paper, regular couples are defined as follows.

\begin{definition}
    The regular index of a couple $\cp \in \mathscr{K}$ is defined by
    \begin{equation*}
        \ind(\cp) = \card(\mathscr{I}_\cp \backslash \mathscr{K}_1).\index{Trees and couples!Couples!Regular index $\ind(\cp)$}
    \end{equation*}
    A couple $\cp$ is \emph{regular} if $\ind(\cp) = 0$.
    The set of all regular couples is denoted by $\mathscr{K}^\reg$.
    For all $n \in \mathbb{N}$,  we denote $\mathscr{K}^\reg_n = \mathscr{K}^\reg \cap \mathscr{K}_n$.\index{Trees and couples!Couples!Sets of regular couples $\mathscr{K}_n$, $\mathscr{K}$}
\end{definition}

Clearly $\mathscr{K}^\reg_0 = \mathscr{K}_0$ and $\mathscr{K}^\reg_1 = \mathscr{K}_1$.
If $n \ge 2$, then $\mathscr{K}^\reg_n$ is a proper subset of $\mathscr{K}_n$.
In Figure~\ref{fig:reg-cp}, we give an example of a regular couple.

\begin{lemma}
    \label{lem:reg-cp-str}
    Let $\cp = (\tau^*,\wp) \in \mathscr{K}$.
    For all $\sigma \in \{+1,-1\}$ and for all $j \in \{1,2,3\}$, let $ \mathfrak{c}^\sigma_{j} = \{\mathfrak{r}^+_{j},\mathfrak{r}^-_{2+ \sigma (j-2)}\}$.
    Then $\cp \in \mathscr{K}^\reg$ if and only if either one of the following statements holds:
    \begin{enumerate}
        \item $\{\mathfrak{c}^+_1,\mathfrak{c}^+_2,\mathfrak{c}^+_3\} \subset \mathfrak{C}^\cp$ and $\{\check{\cp}_{\mathfrak{c}^+_1},\check{\cp}_{\mathfrak{c}^+_2},\check{\cp}_{\mathfrak{c}^+_3}\} \subset \mathscr{K}^\reg$,
        \item $\{\mathfrak{c}^-_1,\mathfrak{c}^-_2,\mathfrak{c}^-_3\} \subset \mathfrak{C}^\cp$ and $\{\check{\cp}_{\mathfrak{c}^-_1},\check{\cp}_{\mathfrak{c}^-_2},\check{\cp}_{\mathfrak{c}^-_3}\} \subset \mathscr{K}^\reg$.
    \end{enumerate}
    Conversely, for all $\cp_1,\cp_2,\cp_3 \in \mathscr{K}^\reg$ and for all $\sigma \in \{+ 1,-1\}$, there exists a unique regular couple $\cp\coloneqq\otimes^\sigma(\cp_1,\cp_2,\cp_3)$ such that $\check{\cp}_{\mathfrak{c}^\sigma_j} = \cp_j$ for all $ j \in \{1,2,3\} $.
\end{lemma}
\begin{proof}
    Suppose that $\cp \in \mathscr{K}^\reg_n$.
    The lemma clearly holds when $n = 1$.
    The general case follows by induction and the definition of the regular couple.
\end{proof}

\begin{remark}
    \label{rmk:I-q-ordering}
    The lemma above describes the structure of regular couples.
    As a consequence, if $\cp \in \mathscr{K}^\reg$, then we can naturally equip $\mathscr{I}_\cp$ with a strict partial order which makes it a tree (in fact $\mathscr{I}_\cp$ will be a ternary tree with child nodes not ordered).
    To do this, we first give a strict partial order on $\mathfrak{C}^\cp_2$ by setting $\{\mathfrak{b}_1^+,\mathfrak{b}_1^-\} \prec \{\mathfrak{b}_2^+,\mathfrak{b}_2^+\}$ if either $\mathfrak{b}_1^\pm \prec \mathfrak{b}_2^\pm$ or $\mathfrak{b}_1^\pm \prec \mathfrak{b}_2^\mp$.
    Then the strict partial order on $\mathscr{I}_\cp$ is defined via the identification $\mathscr{I}_\cp \ni r \mapsto \mathfrak{B}^r \in \mathfrak{C}^\cp_2$.
\end{remark}

\begin{lemma}
    \label{lem:tree-regcp-counting}
    For all $n \in \mathbb{N}$, we have
    \begin{equation*}
        \card (\mathscr{T}^\pm_n) =\frac{1}{2^n}\card (\mathscr{K}^\reg_n) = \frac{1}{2n+1}\binom{3n}{n}
        \lesssim 3^{3n}/ 2^{2n}.
    \end{equation*}
\end{lemma}
\begin{proof}
    Note that $\card(\mathscr{T}^\pm_0) = \card(\mathscr{K}^\reg_0) = 1$.
    By Definition~\ref{def:tree-str} and Lemma~\ref{lem:reg-cp-str}, we have
    \begin{equation*}
        \card (\mathscr{T}^\pm_n)
        = \sum_{n_1+n_2+n_3=n-1} \prod_{1 \le j \le 3} \card (\mathscr{T}^\pm_{n_j}),
        \quad
        \card (\mathscr{K}^\reg_n)
        = 2 \sum_{n_1+n_2+n_3=n-1} \prod_{1 \le j \le 3} \card (\mathscr{K}^\reg_{n_j}).
    \end{equation*}
    Therefore, the identity in the lemma follows from the definition of the generalized Catalan number, and the estimate follows from Stirling's approximation. 
\end{proof}

\subsection{Decorations}

\label{sec:deco}

We introduce two types of decorations, type $\mathscr{D}$ and type $\mathscr{C}$.
Type $\mathscr{D}$ decorations have been used in~\cite{Deng2021b,Deng2021,Deng2021a} to describe collisions among Fourier modes.
However, type $\mathscr{D}$ decorations are insufficient in describing interactions among wavepackets due to the space inhomogeneity.
We therefore introduce type $\mathscr{C}$ decorations as a canonical algebraic complement of $\mathscr{D}$ decorations to complete the picture of interactions, as will be shown in Lemma~\ref{lem:cp-deco-wigner}.

In the following of this section, we let $\mathcal{G}$ be an additive Abelian group.

\begin{definition}
    \label{def:tree-deco}
    If $\tau \in \mathscr{T}$, then $\mathscr{D}^\tau(\mathcal{G})$\index{Trees and couples!Decorations!Type $\mathscr{D}$ decorations $\mathscr{D}^\tau$, $\mathscr{D}^{\tau^*}$, $\mathscr{D}^\cp$} and $\mathscr{C}^\tau(\mathcal{G})$\index{Trees and couples!Decorations!Type $\mathscr{C}$ decorations $\mathscr{C}^\tau$, $\mathscr{C}^{\tau^*}$, $\mathscr{C}^\cp$} are respectively sets of all maps from $\tau$ to $\mathcal{G}$ such that: if $\bm{\zeta} \in \mathscr{D}^\tau(\mathcal{G})$, $\bm{\eta} \in \mathscr{C}^\tau(\mathcal{G})$, and $\iota$ is any sign function on $\tau$, then for all $\mathfrak{b} \in \mathfrak{B}^\tau$, 
    \begin{equation}
        \label{eq:tree-deco-def}
        \iota_{\mathfrak{b}} \bm{\zeta}_{\mathfrak{b}} = \sum_{\mathfrak{n}^p = \mathfrak{b}} \iota_{\mathfrak{n}} \bm{\zeta}_{\mathfrak{n}},\quad
        \bm{\eta}_{\mathfrak{b}} = \sum_{\mathfrak{n}^p = \mathfrak{b}} \bm{\eta}_{\mathfrak{n}}.
    \end{equation}
    If $\zeta,\eta \in \mathcal{G}$, then $\mathscr{D}^\tau_\zeta(\mathcal{G})$ is the set of all $\bm{\zeta} \in \mathscr{D}^\tau(\mathcal{G}) $ with $\bm{\zeta}_{\mathfrak{r}^\tau} = \zeta$, and $\mathscr{C}^\tau_\eta(\mathcal{G})$ is the set of all $\bm{\eta} \in \mathscr{C}^\tau(\mathcal{G})$ with $\bm{\eta}_{\mathfrak{r}^\tau} = \eta$.

    If $\tau^* \in \mathscr{T}^*$, then $\mathscr{D}^{\tau^*}(\mathcal{G})$ is the set of all maps $\bm{\zeta} : \cp \to \mathcal{G}$ with $\bm{\zeta}|_{\tau^\pm} \in \mathscr{D}^{\tau^\pm}$ and $\mathscr{C}^{\tau^*}(\mathcal{G})$ is the set of all maps $\bm{\eta} : \cp \to \mathcal{G}$ with $\bm{\eta}|_{\tau^{\pm}} \in \mathscr{C}^{\tau^\pm}(\mathcal{G})$.
    If $\zeta^\pm \in \mathcal{G}$, then $\mathscr{D}^{\tau^*}_{\zeta^+,\zeta^-}(\mathcal{G})$ is the set of all $\bm{\zeta} \in \mathscr{D}^{\tau^*}(\mathcal{G})$ such that $\bm{\zeta}_{\mathfrak{r}^{\tau^\pm}} = \zeta^\pm$.
    If $\cp = (\tau^*,\wp) \in \mathscr{K}$, then $\mathscr{D}^\cp(\mathcal{G})$ is the set of $\bm{\zeta} \in \mathscr{D}^{\tau^*}(\mathcal{G})$ with $\bm{\zeta}_{\mathfrak{l}^+} = \bm{\zeta}_{\mathfrak{l}^-}$ for all $\mathfrak{p} = \{\mathfrak{l}^+,\mathfrak{l}^-\} \in \wp$, and $\mathscr{C}^\cp(\mathcal{G})$ is the set of all $\bm{\eta} \in \mathscr{C}^{\tau^*}(\mathcal{G})$ with $\bm{\eta}_{\mathfrak{l}^+} = \bm{\eta}_{\mathfrak{l}^-}$ for all $\mathfrak{p} = \{\mathfrak{l}^+,\mathfrak{l}^-\} \in \wp$.
    If $\zeta,\eta \in \mathcal{G}$, then $\mathscr{D}^\cp_\zeta(\mathcal{G})$ is the set of all $\bm{\zeta} \in \mathscr{D}^\cp(\mathcal{G})$ with $\bm{\zeta}_{\mathfrak{r}^{\tau^+}} = \bm{\zeta}_{\mathfrak{r}^{\tau^-}} = \zeta$ and $\mathscr{C}^\cp_\eta(\mathcal{G})$ is the set of all $\bm{\eta} \in \mathscr{C}^\cp(\mathcal{G})$ with $\bm{\eta}_{\mathfrak{r}^{\tau^+}} = \bm{\eta}_{\mathfrak{r}^{\tau^-}} = \eta$.
\end{definition}

\begin{lemma}
    \label{lem:deco-formula}
    If $\tau \in \mathscr{T}$ and $\bm{\zeta} \in \mathscr{D}^\tau(\mathcal{G})$, $\bm{\eta} \in \mathscr{C}^\tau(\mathcal{G})$, then for all $\mathfrak{n} \in \tau$, we have
    \begin{equation}
        \label{eq:deco-formula-tree}
        \iota_{\mathfrak{n}} \bm{\zeta}_{\mathfrak{n}}
        = \sum_{\mathfrak{l} \in \mathfrak{L}^\tau}  \iota_{\mathfrak{l}} \bm{\zeta}_{\mathfrak{l}} \bm{1}_{\mathfrak{n}\succeq \mathfrak{l}}, \quad
        \bm{\eta}_{\mathfrak{n}}
        = \sum_{\mathfrak{l} \in \mathfrak{L}^\tau}  \bm{\eta}_{\mathfrak{l}} \bm{1}_{\mathfrak{n}\succeq \mathfrak{l}}.
    \end{equation}
    Consequently, if $\cp \in \mathscr{K}$ and $\bm{\zeta} \in \mathscr{D}^\cp(\mathcal{G})$, $\bm{\eta} \in \mathscr{C}^\cp(\mathcal{G})$, then for all $\mathfrak{n} \in \cp$, we have
    \begin{equation}
        \label{eq:deco-formula-cp}
        \iota_{\mathfrak{n}} \bm{\zeta}_{\mathfrak{n}}
        = \sum_{\mathfrak{p} = \{\mathfrak{l}^+,\mathfrak{l}^-\} \in \wp} (\bm{1}_{\mathfrak{n} \succeq \mathfrak{l}^+} - \bm{1}_{\mathfrak{n} \succeq \mathfrak{l}^-}) \bm{\zeta}^\flat_{\mathfrak{p}}, \quad
        \bm{\eta}_{\mathfrak{n}}
        = \sum_{\mathfrak{p} = \{\mathfrak{l}^+,\mathfrak{l}^-\} \in \wp} (\bm{1}_{\mathfrak{n} \succeq \mathfrak{l}^+} + \bm{1}_{\mathfrak{n} \succeq \mathfrak{l}^-}) \bm{\eta}^\flat_{\mathfrak{p}},
    \end{equation}
    where for all $\mathfrak{p} = \{\mathfrak{l}^+,\mathfrak{l}^-\} \in \wp$, we denote $\bm{\zeta}^\flat_{\mathfrak{p}} = \bm{\zeta}_{\mathfrak{l}^+} = \bm{\zeta}_{\mathfrak{l}^-}$ and $\bm{\eta}^\flat_{\mathfrak{p}} = \bm{\eta}_{\mathfrak{l}^+} = \bm{\eta}_{\mathfrak{l}^-}$.
\end{lemma}
\begin{proof}
    The identity~\eqref{eq:deco-formula-tree} follows from~\eqref{eq:tree-deco-def} and the mathematical induction on the order of the tree.
    To prove~\eqref{eq:deco-formula-tree}, it suffices to note that if $\cp = (\tau^*,\wp) \in \mathscr{K}$ and $\mathfrak{n} \in \cp$, then
    \begin{align*}
        \iota_{\mathfrak{n}} \bm{\zeta}_{\mathfrak{n}}
        & = \sum_{\mathfrak{l} \in \mathfrak{L}^{\tau^+}}  \iota_{\mathfrak{l}} \bm{\zeta}_{\mathfrak{l}} \bm{1}_{\mathfrak{n}\succeq \mathfrak{l}}
        + \sum_{\mathfrak{l} \in \mathfrak{L}^{\tau^-}}  \iota_{\mathfrak{l}} \bm{\zeta}_{\mathfrak{l}} \bm{1}_{\mathfrak{n}\succeq \mathfrak{l}}
        = \sum_{\mathfrak{p} = \{\mathfrak{l}^+,\mathfrak{l}^-\} \in \wp} (\iota_{\mathfrak{l}^+} \bm{\zeta}_{\mathfrak{l}^+} \bm{1}_{\mathfrak{n}\succeq \mathfrak{l}^+}  + \iota_{\mathfrak{l}^-} \bm{\zeta}_{\mathfrak{l}^-} \bm{1}_{\mathfrak{n}\succeq \mathfrak{l}^-}),\\
        \bm{\eta}_{\mathfrak{n}}
        & = \sum_{\mathfrak{l} \in \mathfrak{L}^{\tau^+}} \bm{\eta}_{\mathfrak{l}^+} + \sum_{\mathfrak{l} \in \mathfrak{L}^{\tau^-}} \bm{\eta}_{\mathfrak{l}^-} 
        = \sum_{\mathfrak{p} = \{\mathfrak{l}^+,\mathfrak{l}^-\} \in \wp} (\bm{\eta}_{\mathfrak{l}^+} \bm{1}_{\mathfrak{n}\succeq \mathfrak{l}^+}  + \bm{\eta}_{\mathfrak{l}^-} \bm{1}_{\mathfrak{n}\succeq \mathfrak{l}^-}). \qedhere
    \end{align*}
\end{proof}

\begin{corollary}
    \label{cor:conju-node-char-deco}
    If $\cp = (\tau^*,\wp) \in \mathscr{K}$ and $\mathfrak{m},\mathfrak{n} \in \cp$, then $\mathfrak{m} \sim \mathfrak{n}$ if and only if either one of the following two statements hold:
    \begin{enumerate}
        \item $\bm{\zeta}_{\mathfrak{m}} = \bm{\zeta}_{\mathfrak{n}}$ for all additive Abelian group $\mathcal{G}$ and for all $\bm{\zeta} \in \mathscr{D}^\cp(\mathcal{G})$;
        \label{item:conju-equiv-D}
        \item $\bm{\eta}_{\mathfrak{m}} = \bm{\eta}_{\mathfrak{n}}$ for all additive Abelian group $\mathcal{G}$ and for all $\bm{\eta} \in \mathscr{C}^\cp(\mathcal{G})$.
        \label{item:conju-equiv-C}
    \end{enumerate}
    Consequently, if $\cp \in \mathscr{K}$ and $\mathfrak{c} \in \mathfrak{C}^\cp_2$, then for all Abelian group $\mathcal{G}$ and for all $\bm{\zeta} \in \mathscr{D}^\cp(\mathcal{G})$, $\bm{\eta} \in \mathscr{C}^\cp(\mathcal{G})$, we have $\bm{\zeta}|_{\hat{\cp}_{\mathfrak{c}}} \in \mathscr{D}^{\hat{\cp}_{\mathfrak{c}}}(\mathcal{G})$, $\bm{\zeta}|_{\check{\cp}_{\mathfrak{c}}} \in \mathscr{D}^{\check{\cp}_{\mathfrak{c}}}(\mathcal{G})$, $\bm{\eta}|_{\hat{\cp}_{\mathfrak{c}}} \in \mathscr{C}^{\hat{\cp}_{\mathfrak{c}}}(\mathcal{G})$, $\bm{\eta}|_{\check{\cp}_{\mathfrak{c}}} \in \mathscr{C}^{\check{\cp}_{\mathfrak{c}}}(\mathcal{G})$.
\end{corollary}
\begin{proof}
    By Lemma~\ref{lem:conju-node-char}, if $\mathfrak{m} \sim \mathfrak{n}$ and $\mathfrak{m} \ne \mathfrak{n}$, then $\mathfrak{m}$ and $\mathfrak{n}$ have opposite signs and belong to opposite trees.
    Therefore, for all $\mathfrak{p} = \{\mathfrak{l}^+,\mathfrak{l}^-\} \in \wp$, we have
    \begin{equation*}
        \iota_{\mathfrak{m}}(\bm{1}_{\mathfrak{m}\succeq \mathfrak{l}^+} - \bm{1}_{\mathfrak{m}\succeq \mathfrak{l}^-}) = \iota_{\mathfrak{n}}(\bm{1}_{\mathfrak{n}\succeq \mathfrak{l}^+} - \bm{1}_{\mathfrak{n}\succeq \mathfrak{l}^-}),\quad
        \bm{1}_{\mathfrak{m} \succeq \mathfrak{l}^+} + \bm{1}_{\mathfrak{m} \succeq \mathfrak{l}^-} = \bm{1}_{\mathfrak{n} \succeq \mathfrak{l}^+} + \bm{1}_{\mathfrak{m} \succeq \mathfrak{l}^-}.
    \end{equation*}
    We thus obtain~\eqref{item:conju-equiv-D} and~\eqref{item:conju-equiv-C} by Lemma~\ref{lem:deco-formula}.
    Conversely, let $\bm{\zeta} \in \mathscr{D}^\cp(\mathbb{Z}^\wp)$ resp.\ $\bm{\eta} \in \mathscr{C}^\cp(\mathbb{Z}^\wp)$ be such that $(\bm{\zeta}^\flat_{\mathfrak{p}})_{\mathfrak{p}'} = (\bm{\eta}^\flat_{\mathfrak{p}})_{\mathfrak{p}'}  = \bm{1}_{\mathfrak{p} = \mathfrak{p}'}$ for all $\mathfrak{p},\mathfrak{p}' \in \wp$.
    If $\mathfrak{m},\mathfrak{n} \in \cp$ is such that $\bm{\zeta}_{\mathfrak{m}} = \bm{\zeta}_{\mathfrak{n}}$ resp.\ $\bm{\eta}_{\mathfrak{m}} = \bm{\eta}_{\mathfrak{n}}$, then by Lemma~\ref{lem:deco-formula}, for all $\mathfrak{p} = \{\mathfrak{l}^+,\mathfrak{l}^-\}\in \wp$, we have
    \begin{equation*}
        \begin{aligned}
            \iota_{\mathfrak{m}} (\bm{1}_{\mathfrak{m}\succeq \mathfrak{l}^+} - \bm{1}_{\mathfrak{m}\succeq \mathfrak{l}^-}) = (\bm{\zeta}_{\mathfrak{m}})_{\mathfrak{p}} & = (\bm{\zeta}_{\mathfrak{n}})_{\mathfrak{p}} =\iota_{\mathfrak{n}} (\bm{1}_{\mathfrak{n}\succeq \mathfrak{l}^+} - \bm{1}_{\mathfrak{n}\succeq \mathfrak{l}^-}),\\
            \text{resp.}\ 
            \bm{1}_{\mathfrak{m}\succeq \mathfrak{l}^+} + \bm{1}_{\mathfrak{m}\succeq \mathfrak{l}^-} = (\bm{\eta}_{\mathfrak{m}})_{\mathfrak{p}} & = (\bm{\eta}_{\mathfrak{n}})_{\mathfrak{p}} = \bm{1}_{\mathfrak{n}\succeq \mathfrak{l}^+} + \bm{1}_{\mathfrak{n}\succeq \mathfrak{l}^-}.
        \end{aligned}
    \end{equation*}
    Therefore $\bm{1}_{\mathfrak{m}\succeq \mathfrak{l}^+} \pm \bm{1}_{\mathfrak{m}\succeq \mathfrak{l}^-} $ and $ \bm{1}_{\mathfrak{n}\succeq \mathfrak{l}^+} \pm \bm{1}_{\mathfrak{n}\succeq \mathfrak{l}^-} $ are either both vanishing, so that $\{\mathfrak{m},\mathfrak{n}\} \subset \cp \backslash \mathscr{A}_\mathfrak{p}$, or both non-vanishing, so that $\{\mathfrak{m},\mathfrak{n}\} \subset \mathscr{A}_\mathfrak{p}$.
    We thus obtain $\mathfrak{m} \sim \mathfrak{n}$ from Definition~\ref{def:conju-node}.
\end{proof}

As a consequence of Corollary~\ref{cor:conju-node-char-deco}, elements in $\mathscr{D}^\cp(\mathcal{G})$ or $\mathscr{C}^\cp(\mathcal{G})$ can be pushed down by the quotient map $\pi^\cp$ to maps from $\mathfrak{C}^\cp$ to $\mathcal{G}$.
In fact, if $\bm{\zeta} \in \mathscr{D}^\cp(\mathcal{G})$ and $\bm{\eta} \in \mathscr{C}^\cp(\mathcal{G})$, then these maps extend $\bm{\zeta}^\flat$ and $\bm{\eta}^\flat$ (which are defined in Lemma~\ref{lem:deco-formula}) from $\wp$ to the whole $\mathfrak{C}^\cp$.
With an abuse of notation, let us denote by $\bm{\zeta}^\flat$ and $\bm{\eta}^\flat$ the maps from $\mathfrak{C}^\cp$ to $\mathcal{G}$ such that
\begin{equation*}
    \bm{\zeta}^\flat \comp \pi^\cp = \bm{\zeta}, \quad
    \bm{\eta}^\flat \comp \pi^\cp = \bm{\eta}.\index{Trees and couples!Decorations!Pushforward to conjugate classes $\flat$}
\end{equation*}

\subsection{Topological structure}

\label{sec:topo}

We give two topological structures on every couple.
They are closely related to each other and crucial in connecting the linear relations of resonance functions defined \S\ref{sec:reso-function} with the irreducible factorization of couples.

\begin{definition}
    \label{def:cp-topo}
    If $\cp = (\tau^*,\wp) \in \mathscr{K}$, then $\mathfrak{T}^q$ is the topology on $\cp$ generated by $\{\mathscr{A}_{\mathfrak{p}}\}_{\mathfrak{p} \in \wp}$ and $\mathfrak{I}^\cp$ is the topology on $\mathfrak{B}^\cp$ generated by $\{\mathfrak{B}^{r}\}_{r \in \mathscr{I}_\cp}$.
\end{definition}

\begin{lemma}
    \label{lem:Kolmogorov-space}
    If $\cp = (\tau^*,\wp) \in \mathscr{K}$, then $(\mathfrak{C}^\cp,\pi^\cp \mathfrak{T}_\cp)$ is a Kolmogorov space.
    In fact, if $\mathfrak{c},\mathfrak{c}' \in \mathfrak{C}^\cp$ and $\mathfrak{c} \ne \mathfrak{c}'$, then there exists $\mathfrak{p} \in \wp$ such that $\mathscr{A}_{\mathfrak{p}}$ separates $\mathfrak{c}$ and $\mathfrak{c}'$, i.e., we have either $\mathfrak{c} \subset \mathscr{A}_{\mathfrak{p}}$ and $\mathfrak{c}' \cap \mathscr{A}_{\mathfrak{p}} = \emptyset$ or $\mathfrak{c}' \subset \mathscr{A}_{\mathfrak{p}}$ and $\mathfrak{c} \cap \mathscr{A}_{\mathfrak{p}} = \emptyset$.
\end{lemma}
\begin{proof}
    This follows directly from Definition~\ref{def:conju-node} and Definition~\ref{def:cp-topo}.
\end{proof}

\begin{lemma}
    If $\cp = (\tau^*,\wp) \in \mathscr{K}$, then $\mathfrak{U}^\cp = \{\mathscr{A}_{\mathfrak{p}_1} \cap \mathscr{A}_{\mathfrak{p}_2}\}_{\mathfrak{p}_1,\mathfrak{p}_2 \in \wp}$ is a base for $\mathfrak{T}^\cp$.
    In fact, for all nonempty $P \subset \wp$, we have
    $\bigcap_{\mathfrak{p} \in P} \mathscr{A}_{\mathfrak{p}} \in \mathfrak{U}^\cp.$
\end{lemma}
\begin{proof}
    By mathematical induction, it suffices prove the lemma when $P$ is a tripleton.
    Write $P = \{\mathfrak{a},\mathfrak{b},\mathfrak{c}\}$.
    Observe that for any $\sigma \in \{+,-\}$, at least two nodes among $\min\{\mathscr{A}_{\mathfrak{a}} \cap \mathscr{A}_{\mathfrak{b}} \cap \tau^\sigma\}$, $\min\{\mathscr{A}_{\mathfrak{b}} \cap\mathscr{A}_{\mathfrak{c}} \cap \tau^\sigma\}$ and $\min\{\mathscr{A}_{\mathfrak{c}} \cap \mathscr{A}_{\mathfrak{a}} \cap \tau^\sigma\}$ are the same.
    Therefore, there exists subsets $C^+$ and $C^-$ of $\{\{\mathfrak{a},\mathfrak{b}\},\{\mathfrak{b},\mathfrak{c}\},\{\mathfrak{c},\mathfrak{a}\}\}$ with $\card(C^\pm) = 2$ such that if $Q^\pm = \{\mathfrak{f}^\pm,\mathfrak{g}^\pm\} \in C^\pm$, then
    \begin{equation*}
        \mathscr{A}_{\mathfrak{a}} \cap \mathscr{A}_{\mathfrak{b}} \cap \mathscr{A}_{\mathfrak{c}} \cap \tau^\pm
        = \mathscr{A}_{\mathfrak{f}^\pm} \cap \mathscr{A}_{\mathfrak{g}^\pm} \cap \tau^\pm.
    \end{equation*}
    Because $C^\pm$ are subsets of a tripleton and $\card(C^\pm) = 2$, we always have $C^+ \cap C^- \ne \emptyset$.
    Let $\{\mathfrak{p}_1,\mathfrak{p}_2\} \in C^+ \cap C^-$, then $\mathscr{A}_\mathfrak{a} \cap \mathscr{A}_{\mathfrak{b}} \cap \mathscr{A}_{\mathfrak{c}} \cap \tau^\pm = \mathscr{A}_{\mathfrak{p}_1} \cap \mathscr{A}_{\mathfrak{p}_2} \cap \tau^\pm$.
    Consequently, we have $\mathscr{A}_\mathfrak{a} \cap \mathscr{A}_{\mathfrak{b}} \cap \mathscr{A}_{\mathfrak{c}}  = \mathscr{A}_{\mathfrak{p}_1} \cap \mathscr{A}_{\mathfrak{p}_2} \in \mathfrak{U}^\cp$.
\end{proof}

In Figure~\ref{fig:top-base} we show an example of $\mathscr{A}_{\mathfrak{p}_1} \cap \mathscr{A}_{\mathfrak{p}_2}$, where $\mathfrak{p}_1$ and $\mathfrak{p}_2$ are marked by double dotted lines and $\mathscr{A}_{\mathfrak{p}_1} \cap \mathscr{A}_{\mathfrak{p}_2}$ consist of square nodes.
\begin{figure}[htb!]
    \begin{tikzpicture}[baseline=(root+),scale=0.5,font=\footnotesize]
        \node(root+)[DSolid]{}
            child {node[DSolid]{}
                child {node(n6+)[Solid]{}}
                child {node(n5-)[Hollow]{}}
                child {node[Solid]{}
                    child {node(p+)[Solid]{}}
                    child {node(n4-)[Hollow]{}}
                    child {node(n3+)[Solid]{}}}}
            child {node(n2-)[Hollow]{}}
            child {node(n1+)[Solid]{}};
        \node(root-) at (4.5,0) [DHollow]{}
            child {node(n1-)[Hollow]{}}
            child {node[DSolid]{}
                child {node(n2+)[Solid]{}}
                child {node(n3-)[Hollow]{}}
                child {node[Solid]{}
                    child {node(n5+)[Solid]{}}
                    child {node(p-)[Hollow]{}}
                    child {node(n4+)[Solid]{}}}}
            child {node(n6-)[Hollow]{}};
        \draw[dotted] (n1+) to (n1-);
        \draw[dotted] (n2-) to (n2+);
        \draw[dotted,double,thick] (n3+) to (n3-);
        \draw[dotted] (n4-) to[out=-30,in=210]  (n4+);
        \draw[dotted,double,thick] (n5-) to[out=30,in=150]  (n5+);
        \draw[dotted] (n6+) to[out=75,in=165]  (n6-);
        \draw[dotted] (root+) to[out=15,in=165] (root-);
        \draw[dotted] (p+) to[out=-30,in=210] (p-);
    \end{tikzpicture}
    \caption{An example of $\mathscr{A}_{\mathfrak{p}_1} \cap \mathscr{A}_{\mathfrak{p}_2}$}
    \label{fig:top-base}
\end{figure}

\begin{lemma}
    \label{lem:topo-linear-char}
    Let $\cp \in \mathscr{K}$.
    For all $\mathfrak{n} \in \cp$, let $\Gamma^\cp_{\mathfrak{n}}$\index{Functions and random variables!Resonance functions!Matrices for dispersion relation $\Gamma^\cp_{\mathfrak{n}}$} be the $\wp \times \wp$ matrix such that
    \begin{equation*}
        \Gamma^\cp_{\mathfrak{n}}(\mathfrak{p}_1,\mathfrak{p}_2)
        = (\bm{1}_{\mathfrak{n} \succeq \mathfrak{l}_1^+} - \bm{1}_{\mathfrak{n} \succeq \mathfrak{l}_1^-}) (\bm{1}_{\mathfrak{n} \succeq \mathfrak{l}_2^+} - \bm{1}_{\mathfrak{n} \succeq \mathfrak{l}_2^-})
    \end{equation*} 
    for all $\mathfrak{p}_1 = \{\mathfrak{l}_1^+,\mathfrak{l}_1^-\} $ and $\mathfrak{p}_2 = \{\mathfrak{l}_2^+,\mathfrak{l}_2^-\}$ in $\wp$.
    For all $\lambda:\cp\to\mathbb{C}$, the following three statements are equivalent:
    \begin{enumerate*}
        \item $\sum_{\mathfrak{n}\in \mathfrak{c}} \lambda_{\mathfrak{n}} = 0$ for all $\mathfrak{c} \in \mathfrak{C}^\cp$;
        \label{item:linear-rel-C}
        \item $\sum_{\mathfrak{n} \in \mathscr{B}} \lambda_{\mathfrak{n}} = 0$ for all $\mathscr{B} \in \mathfrak{U}^\cp$;
        \label{item:linear-rel-U}
        \item $\sum_{\mathfrak{n} \in \cp} \lambda_{\mathfrak{n}} \Gamma^\cp_{\mathfrak{n}} = 0$.
        \label{item:linear-rel-Gamma}
    \end{enumerate*}
\end{lemma}
\begin{proof}
    Observe that if $\mathfrak{p}_1 = \{\mathfrak{l}^+_1,\mathfrak{l}^-_1\},\mathfrak{p}_2 = \{\mathfrak{l}^+_2,\mathfrak{l}^-_2\} \in \wp$ and $\mathscr{B} = \mathscr{A}_{\mathfrak{p}_1} \cap \mathscr{A}_{\mathfrak{p}_2} \in \mathfrak{U}^\cp$, then
    \begin{equation*}
        \sum_{\mathfrak{n} \in \cp} \lambda_{\mathfrak{n}} \Gamma^\cp_{\mathfrak{n}}(\mathfrak{p}_1,\mathfrak{p}_2)
        = \sum_{\mathfrak{n} \in \mathscr{B}} \lambda_{\mathfrak{n}} (\bm{1}_{\mathfrak{n} \succeq \mathfrak{l}_1^+} - \bm{1}_{\mathfrak{n} \succeq \mathfrak{l}_1^-}) (\bm{1}_{\mathfrak{n} \succeq \mathfrak{l}_2^+} - \bm{1}_{\mathfrak{n} \succeq \mathfrak{l}_2^-})
        = \iota_{\mathfrak{p}_1} \iota_{\mathfrak{p}_2} \sum_{\mathfrak{n} \in \mathscr{B}} \lambda_{\mathfrak{n}},
    \end{equation*}
    where $\iota_{\mathfrak{p}_j} = \bm{1}_{\mathfrak{l}^+_j \in \tau^+}$ for $j \in \{1,2\}$.
    Consequently, the statements~\eqref{item:linear-rel-U} and~\eqref{item:linear-rel-Gamma} are equivalent and it suffices to show that teh statements~\eqref{item:linear-rel-C} and~\eqref{item:linear-rel-U} are equivalent.
    By Definition~\ref{def:conju-node} and Lemma~\ref{lem:conju-node-char}, every $\mathscr{B} \in\mathfrak{U}^\cp$ is a disjoint union of elements in $\mathfrak{C}^\cp$.
    Therefore~\eqref{item:linear-rel-C} implies~\eqref{item:linear-rel-U}.

    Next we show that~\eqref{item:linear-rel-U} implies~\eqref{item:linear-rel-C}.
    Let $S$ be the set of all $\mathscr{B} \in \mathfrak{U}^\cp $ such that $\sum_{\mathfrak{n} \in \mathfrak{c}} \lambda_{\mathfrak{n}} \ne 0$ for some $\mathfrak{c} \in \pi^\cp \mathscr{B}$.
    Suppose that~\eqref{item:linear-rel-C} is untrue, then $S \ne \emptyset$.
    In fact, if $\mathfrak{c} \in \mathfrak{C}^\cp$ violates~\eqref{item:linear-rel-C} and $\mathfrak{c} \in \mathscr{A}_\mathfrak{p}$ for some $\mathfrak{p} \in \wp$, then $ \mathscr{A}_\mathfrak{p} \cap \mathscr{A}_\mathfrak{p} \in S$.
    Now fix any $\mathscr{B} \in S$ which minimizes $\card(\pi^\cp \mathscr{B})$ and fix any $\mathfrak{c} \in \pi^\cp \mathscr{B}$ which violates~\eqref{item:linear-rel-C}.
    The condition~\eqref{item:linear-rel-U} implies the existence of an element $\mathfrak{c}' \in \pi^\cp \mathscr{B} \backslash \{\mathfrak{c}\}$ which also violates~\eqref{item:linear-rel-C}.
    By Lemma~\ref{lem:Kolmogorov-space}, there exists $\mathfrak{p} \in \wp$ such that $\mathscr{A}_{\mathfrak{p}}$ separates $\mathfrak{c}$ and $\mathfrak{c}'$.
    Assume without losing generality that $\mathfrak{c} \in \mathscr{A}_{\mathfrak{p}}$.
    Then $\mathfrak{c} \subset \mathscr{B} \cap \mathscr{A}_{\mathfrak{p}} \in S$.
    Now that $\mathfrak{c}' \notin \pi^\cp(\mathscr{B} \cap \mathscr{A}_{\mathfrak{p}})$, we thus have
    \begin{equation*}
        \card(\pi^\cp(\mathscr{B} \cap \mathscr{A}_{\mathfrak{p}})) < \card(\pi^\cp \mathscr{B}),
    \end{equation*}
    which contradicts the minimizing property of $\mathscr{B}$.
\end{proof}

\begin{corollary}
    \label{cor:reso-matrix-linear-rel}
    Let $\cp \in \mathscr{K}$ and let $\mu$ be a map from $\mathfrak{B}^\cp$ to $\mathbb{C}$.
    Then
    \begin{equation*}
        \sum_{\mathfrak{b} \in \mathfrak{B}^\cp} \mu_{\mathfrak{b}} \Bigl(\iota_{\mathfrak{b}} \Gamma^\cp_{\mathfrak{b}} - \sum_{\mathfrak{n}^p = \mathfrak{b}} \iota_{\mathfrak{n}} \Gamma^\cp_{\mathfrak{n}}\Bigr) = 0
    \end{equation*}
    if and only if $\mu$ is locally constant with respect to the topology $\mathfrak{I}^\cp$.
    This means that for all $r \in \mathscr{I}_\cp$ and for all $\mathfrak{b}_1,\mathfrak{b}_2 \in \mathfrak{B}^{r} \subset \mathfrak{B}^\cp$, we have $\mu_{\mathfrak{b}_1} = \mu_{\mathbf{b}_2}$.
\end{corollary}
\begin{proof}
    Using the convention that $\mu_{\mathfrak{n}^p} = 0$ if $\mathfrak{n}$ is a root and $\mu_{\mathfrak{n}} = 0$ if $\mathfrak{n} \in \mathfrak{L}^\cp$. Then
    \begin{equation*}
        \sum_{\mathfrak{b} \in \mathfrak{B}^\cp} \mu_{\mathfrak{b}} \Bigl( \iota_{\mathfrak{b}} \Gamma^{\cp}_{\mathfrak{b}} - \sum_{\mathfrak{n}^p = \mathfrak{b}} \iota_{\mathfrak{n}} \Gamma^{\cp}_{\mathfrak{n}} \Bigr)
        = \sum_{\mathfrak{n} \in \cp} \iota_{\mathfrak{n}} (\mu_{\mathfrak{n}} - \mu_{\mathfrak{n}^\cp}) \Gamma^\cp_{\mathfrak{n}}.
    \end{equation*}
    By Lemma~\ref{lem:topo-linear-char}, the above sum is zero if and only if for all $\mathfrak{c} \in \mathfrak{C}^\cp$ we have
    \begin{equation}
        \label{eq:linear-rel-coeff}
        \sum_{\mathfrak{n} \in \mathfrak{c}} \iota_{\mathfrak{n}}  (\mu_{\mathfrak{n}} - \mu_{\mathfrak{n}^\cp}) = 0,
    \end{equation}
    or equivalently, if $\mathfrak{c} = \{\mathfrak{n}\} \in \mathfrak{C}^\cp_1$, then $\mu_{\mathfrak{n}} = \mu_{\mathfrak{n}^\cp}$, and if $\mathfrak{c} = \{\mathfrak{n}^+,\mathfrak{n}^-\} \in \mathfrak{C}^\cp_2$, then $\mu_{\mathfrak{n}^+} - \mu_{\mathfrak{n}^-} = \mu_{(\mathfrak{n}^+)^p} - \mu_{(\mathfrak{n}^-)^p}$.
    When $\cp = (\tau^*,\wp)$ is irreducible, by Definition~\ref{def:cp-irreducible}, this is equivalent to the conditions that $\mu_{\mathfrak{r}^{\tau^+}} = \mu_{\mathfrak{r}^{\tau^-}}$ and $ \mu_{\mathfrak{b}} = \mu_{\mathfrak{r}^{\tau^\pm}}$ for all $\mathfrak{b} \in \mathfrak{B}^{\tau^\pm}$, which is clearly equivalent to $\mu$ being constant on $\mathfrak{B}^\cp$.
    The general case follows by mathematical induction on the order of $\cp$.
    Assume that~\eqref{eq:linear-rel-coeff} implies that $\mu$ is locally constant with respect to $\mathfrak{I}^\cp$ when the order of $\cp$ is $< n$.
    Now, let $\cp \in \mathscr{K}_n$ be an irreducible couple and write $\cp = \cp_1 \otimes_{\mathfrak{c}} \cp_2$ for some $\mathfrak{c} \in \mathfrak{C}^\cp_2$ such that both $\cp_1$ and $\cp_2$ are nontrivial and $\cp_2$ is irreducible, then by the induction hypothesis, the function $\mu|_{\mathfrak{B}^{\cp_1}}$ is locally constant with respect to $\mathfrak{I}^{\cp_1}$.
    Moreover, writing $\cp_2 = (\tau^*_2,\wp_2)$, then $\mu_{\mathfrak{b}} = \mu_{\mathfrak{r}^{\tau^\pm_2}}$ when $\mathfrak{b} \in \mathfrak{B}^{\tau^\pm_2}$ and $\mu_{\mathfrak{r}^{\tau^+_2}} - \mu_{\mathfrak{r}^{\tau^-_2}} = \mu_{(\mathfrak{r}^{\tau^+_2})^p} - \mu_{(\mathfrak{r}^{\tau^-_2})^p} = 0 $ because the $(\mathfrak{r}^{\tau^\pm_2})^p \in \mathfrak{B}^{\cp_1}$.
    This shows that $\mu|_{\mathfrak{B}^\cp_2}$ is constant and therefore $\mu$ is locally constant with respect to $\mathfrak{I}^\cp$.
\end{proof}

\section{Oscillatory sums and integrals}

\label{sec:oscillatory}

The purpose of this section is to analyze the oscillatory functionals defined in Definition~\ref{def:functional-S}.
Estimates and asymptotic behaviors of these functionals are essential in the study of the large box limit $L \to \infty$ for regular couples.

\begin{definition}
    \label{def:functional-S}
    For $n \ge 1$ and $d \ge 1$, denote for simplicity $X^{n,d} = \mathbb{R}^{n} \times (\mathbb{R}^{2d})^n$.
    For all $L \in (0,\infty]$, the linear function $\mathcal{S}_L$\index{Operations and transforms!Oscillatory sums and integrals!Oscillatory sum/integral operator $\mathcal{S}_L$} on $ \mathscr{S}(X^{n,d}) $ is defined by
    \begin{equation*}
        \mathcal{S}_L(\Phi)
        =
        \begin{cases}
            \displaystyle
            \frac{1}{L^{2nd}} \sum_{\bm{z} \in (\mathbb{Z}^{2d}_L)^n} \int_{\mathbb{R}^n} e^{2\pi i \bm{s} \cdot \varpi(\bm{z})} \Phi(\bm{s},\bm{z}) \diff \bm{s}, & L < \infty; \\
            \displaystyle
            \iint_{X^{n,d}} e^{2\pi i \bm{s} \cdot \varpi(\bm{z})} \Phi(\bm{s},\bm{z}) \diff \bm{s} \diff \bm{z},                                       & L = \infty,
        \end{cases}
    \end{equation*}
    where the phase factor $\varpi = (\varpi_j)_{1\le j \le n}: (\mathbb{R}^{2d})^n \to \mathbb{R}^n$ is defined by setting $\varpi_j(\bm{z}) = \bm{x}_j \cdot \bm{y}_j$ if $\bm{z} = (\bm{x}_j,\bm{y}_j)_{1 \le j \le n}$ and $\bm{s} \cdot \varpi(\bm{z})$ is defined as an inner product in $\mathbb{R}^n$.
\end{definition}

In this section, we shall fix $\alpha \in (0,2)$ and write $\gamma = L^\alpha$ for simplicity.
We define the scaling operator $\theta_L $ on $ C^\infty(X^{n,d}) $ by setting
$\theta_L \Phi(\bm{s},\bm{z}) = \Phi(\bm{s}/\gamma,\bm{z}).$\index{Operations and transforms!Oscillatory sums and integrals!Scaling operator $\theta_L$}
We also denote by $\theta_\infty$ the trace operator defined by $\theta_\infty \Phi(\bm{z}) = \Phi(0,\bm{z})$.

\subsection{Estimates of upper bounds}

We first obtain upper bounds for $\mathcal{S}_L$ in both the cases $L<\infty$ and $L = \infty$.
The following lemma will be frequently used.

\begin{lemma}
    \label{lem:gauss-sum-fourier-decomp}
    For all $\bm{\xi},\bm{z} \in (\mathbb{R}^{2d})^n$, let $\psi_{\bm{\xi}}(\bm{z}) = e^{2\pi i \bm{\xi} \cdot \bm{z}}$.
    For $\Phi \in C^\infty(X^{n,d})$ and $\chi \in C_c^\infty(X^{n,d})$, denote
    $\widehat{\Phi}^\chi_{\bm{\xi}}(\bm{s},\bm{z}) = \chi(\bm{s},\bm{z}) \widehat{\Phi}(\bm{s},\bm{\xi})$,
    where $\widehat{\Phi}(\bm{s},\cdot)$ is the Fourier transform of $\Phi(\bm{s},\cdot)$.
    Then for all $L \in (0,\infty]$, we have the identity
    \begin{equation*}
        \mathcal{S}_L(\chi\Phi)
        = \int_{(\mathbb{R}^{2d})^n} \mathcal{S}_L(\psi_{\bm{\xi}}\widehat{\Phi}^\chi_{\bm{\xi}}) \diff \bm{\xi}.
    \end{equation*}
\end{lemma}
\begin{proof}
    It suffices to plug the following identity into $\mathcal{S}_L(\chi\Phi)$:
    \begin{equation*}
        (\chi \Phi)(\bm{s},\bm{z})
        = \chi(\bm{s},\bm{z}) \int_{(\mathbb{R}^{2d})^n} e^{2\pi i \bm{\xi} \cdot \bm{z}} \widehat{\Phi}(\bm{s},\bm{\xi})\diff \bm{\xi}.
        \qedhere
    \end{equation*}
\end{proof}

When $L < \infty$, we estimate the oscillatory sum $\mathcal{S}_L$ functional by transforming it into expressions involving Gauss sums.
Then we use the following Hua's lemma borrowed from~\cite{Deng2021}.

\begin{lemma}
    \label{lem:gauss-sum-def-est}
    For all $s,r \in \mathbb{R}$, $h \in \mathbb{Z}$ and $x \ge 0$, we denote by $G_h(s,r,x)$ the Gauss sum
    \begin{equation*}
        G_h(s,r,x) = \sum_{h \le n \le h + \lfloor x \rfloor} e^{\pi i (s n^2 + rn)}.
        \index{Functions and random variables!Number theoretical functions!Gauss sum $G_h$}
    \end{equation*}
    Then for all $N \in \mathbb{N} \backslash \{0\}$, we have the uniform estimates
    \begin{equation*}
        \sup_{r \in \mathbb{R}} \sup_{h \in \mathbb{Z}} \|G_h(\cdot,r,N)\|_{L^4}^4 \lesssim N^2 \log(1+N), \quad
        \sup_{r \in \mathbb{R}} \sup_{h \in \mathbb{Z}} \|G_h(\cdot,r,N)\|_{L^6}^6 \lesssim N^4.
    \end{equation*}
\end{lemma}

\begin{proposition}
    \label{prop:gauss-sum-upper-bound}
    Let $d \ge 3$, $n \ge 1$, $\delta \in (0,2-\alpha)$.
    If $L \ge 1$, then for all $(\bm{u},\bm{\xi}) \in X^{n,d}$ and all $\Phi \in L^\infty W^{2d,1}\cap L^1L^\infty$ such that $\supp \Phi \subset \mathbb{R}^n \times B$ for some bounded $B \subset (\mathbb{R}^{2d})^n$, we have
    \begin{equation}
        \label{eq:gauss-sum-upper-bound-pre}
        \sup_{(\bm{u},\bm{\xi}) \in X^{n,d}} |\mathcal{S}_L(\psi_{\bm{\xi}} \theta_L \Phi_{\bm{u}})|
        \lesssim C^n \ell(B)^{2nd}(\|\Phi\|_{L^\infty W^{2d,1}} + \gamma \|\Phi\|_{L^1_{|\bm{s}|\ge L^{\delta}}L^\infty})
    \end{equation}
    for some universal constant $C > 0$.
    In~\eqref{eq:gauss-int-upper-bound-pre}, we denote $\Phi_{\bm{u}}(\bm{s},\bm{z}) = e^{2\pi i \bm{s} \cdot \bm{u}} \Phi(\bm{s},\bm{z})$ and $\ell(B) = \langle 1 + \diam B \rangle$.
    Consequently if $\Phi \in \mathscr{S}'(X^{n,d})$ is such that $\supp \Phi \subset Q \times B$ for some bounded subsets $Q \subset \mathbb{R}^n$, then for sufficiently large $L > 0$, we have
    \begin{equation}
        \label{eq:gauss-sum-upper-bound}
        | \mathcal{S}_L(\theta_L \Phi)| \lesssim C^n \ell(B)^{4nd} \|\widehat{\Phi}\|_{L^\infty L^1}.
    \end{equation}
\end{proposition}
\begin{proof}
    To prove~\eqref{eq:gauss-sum-upper-bound-pre}, it suffices to focus on the case $n=1$.
    In fact, the general case follows by applying the $n=1$ estimate $n$ times.
    Let $(u,\xi) \in X^{1,d}$.
    Let $\chi \in C_c^\infty(\mathbb{R})$ be such that $\chi(s) = 1$ when $|s| \le 1$ and let $\chi^\delta(s) = \chi(s/L^{\delta})$.
    Then we write $\mathcal{S}_L(\psi_{\bm{\xi}} \theta_L \Phi_{\bm{u}}) = P + Q$, where
    \begin{equation*}
        P = \mathcal{S}_L\bigl((1-\chi^\delta) \psi_{\xi} \theta_L \Phi_u\bigr), \quad
        Q = \mathcal{S}_L(\chi^\delta \psi_{\xi} \theta_L \Phi_u),
    \end{equation*}
    and estimate these two integrals separately.    
    By the definition of $\mathcal{S}_L$ and $\theta_L$, we have
    \begin{equation*}
        \begin{aligned}
            |P|
            & \le \frac{1}{L^{2d}} \sum_{z \in \mathbb{Z}^{2d}_L}  \int \bigl| (1-\chi)(s/(\gamma/L^\delta)) \Phi(s/\gamma,z)\bigr| \diff s= \frac{\gamma}{L^{2d}} \sum_{z \in \mathbb{Z}^{2d}_L} \int |(1-\chi)(s/L^\delta)\Phi(s,z)| \diff s\\
            & \lesssim \frac{\gamma}{L^{2d}} \sum_{z \in \mathbb{Z}^{2d}_L} \|\Phi(\cdot,z)\|_{L^1_{|s|\ge L^\delta}}
            \lesssim \gamma \ell(B)^{2d} \|\Phi\|_{L^1_{|s|\ge L^\delta} L^\infty}.
        \end{aligned}
    \end{equation*}
    To estimate $Q$, define the linear transform
        $\varphi : \mathbb{R}^{2d} \to \mathbb{R}^{2d}$,
        $(p,q) \mapsto (x,y) = (p+q,p-q)/2,$
    and note that $(\varpi \comp \varphi)(p,q) = (|p|^2 - |q|^2)/2$.
    Let $B_L = \{Lz: z \in B\}$, let $\tau = (s + \gamma u)/L^2$ and denote $(p,q) = \varphi(z)$, $(\zeta,\eta) = \varphi^\dagger(\xi)/L$, where $\varphi^\dagger$ is the adjoint of $\xi$, then
    \begin{equation*}
        \begin{aligned}
            Q
            & = \frac{1}{L^{2d}} \sum_{z \in B_L} \int e^{2\pi i s( \varpi(z)/L^2 + u)} e^{2\pi i \xi \cdot z/L} \chi\Bigl(\frac{s}{\gamma L^\delta}\Bigr) \Phi\Bigl(\frac{s}{\gamma},\frac{z}{L}\Bigr) \diff s\\
            & = \frac{L^2}{L^{2d}} \sum_{z \in B_L} \int e^{2\pi i \tau (\varpi(z) + uL^2)} e^{2\pi i \xi\cdot (z) /L}
            \chi\Bigl(\frac{\tau L^2}{\gamma L^{\delta}}\Bigr) \Phi\Bigl(\frac{\tau L^2}{\gamma},\frac{z+z_0}{L}\Bigr) \diff \tau\\
            & = \frac{L^2}{L^{2d}} \sum_{(p,q) \in \varphi(B_L)} \int e^{\pi i \tau (|p|^2-|q|^2)} e^{2\pi i (\zeta \cdot p + \eta \cdot q)}
            e^{2\pi i \tau u L^2}
            \chi\Bigl(\frac{\tau L^2}{\gamma L^{\delta}}\Bigr) \Phi\Bigl(\frac{\tau L^2}{\gamma},\frac{\varphi(p,q)}{L}\Bigr) \diff \tau.
        \end{aligned}
    \end{equation*}
    Note that $\ell(\varphi(B_L)) \le L \ell(B)$.
    We can choose $(p_0,q_0) \in \mathbb{Z}^d$ such that
    \begin{equation*}
        \varphi(B_L) \subset [(p_0,q_0) - \ell(B) L,(p_0,q_0) + \ell(B) L]^{2d}.
    \end{equation*}
    Let $\tilde{\zeta} = 2(\zeta + \tau p_0)$, $\tilde{\eta} = 2(\eta + \tau q_0)$, $\varrho = - \tau(|p_0|^2-|q_0|^2) + 2\zeta \cdot p_0 + 2 \eta \cdot q_0 + 2 \tau uL^2$, and let
    \begin{equation}
        \label{eq:integrand-gauss-sum-int}
        \tilde{\Phi}(\tau,p,q) = e^{\pi i \varrho} \chi\Bigl(\frac{\tau L^2}{\gamma L^\delta}\Bigr) \Phi\Bigl(\frac{\tau L^2}{\gamma}, \frac{\varphi(p+p_0,q+q_0)}{L}\Bigr)
    \end{equation}
    then we can write
    \begin{equation*}
        Q = \frac{L^2}{L^{2d}}
        \sum_{(p+p_0,q+q_0) \in \varphi(B_L)}
        \int e^{\pi i \tau(|p|^2-|q|^2)} e^{\pi i \tilde{\zeta}\cdot p + \tilde{\eta}\cdot q} \tilde{\Phi}(\tau,p,q) \diff \tau.
    \end{equation*}
    Let $\Gamma_1 = (2 \mathbb{Z}^d) \times \mathbb{Z}^d$, $\Gamma_2 = \mathbb{Z}^d \times (2\mathbb{Z}^d)$, $\Gamma_3 = (2 \mathbb{Z}^d) \times (2\mathbb{Z}^d)$, and let $\iota_1 = 1$, $\iota_2 = -1$, $\iota_3 = 2$.
    Let $h \in \mathbb{Z}$ be such that $h < -L\ell(B)$.
    Then we have $Q = \sum_{1 \le m \le 3} \iota_m Q_m$, where
    \begin{equation}
        \label{eq:int-by-part-gauss-sum}
        \begin{aligned}
            Q_m
            & = \frac{L^2}{L^{2d}}
            \sum_{(p,q) \in \Gamma_m}
            \int e^{\pi i \tau(|p|^2-|q|^2)} e^{\pi i \tilde{\zeta}\cdot p + \tilde{\eta}\cdot q} \tilde{\Phi}(\tau,p,q) \diff \tau\\
            & = \frac{L^2}{L^{2d}} \iint_{\mathbb{R}^{2d}} \int \Lambda \tilde{\Phi}(\tau,w) \prod_{1\le j\le d} G_h(\alpha_m^2 \tau,\alpha_m \tilde{\zeta}_j,p_j) G_h(-\beta_m^2 \tau,\beta_m\tilde{\eta}_j,q_j) \diff p \diff q \diff \tau,
        \end{aligned}
    \end{equation}
    where $\alpha_m,\beta_m \in \{1,2\}$ and $\Lambda = \prod_{1 \le j \le d} \partial_{p_j} \partial_{q_j}$.
    Note that if $(\tau,p,q) \in \supp \tilde{\Phi} $, then $|\tau| \lesssim \gamma L^{\delta-2} \lesssim 1$ and $|p|+|q| \lesssim \ell(B) L$.
    Using the trivial bound $\|G_h(\cdot,\cdot,N)\|_{L^\infty} \le N$ for all $N \in \mathbb{N}$, Lemma~\ref{lem:gauss-sum-def-est}, H\"older's inequality, and the condition $d \ge 3$, we have
    \begin{equation*}
        \begin{aligned}
            |Q_m|
            & \lesssim \frac{\ell(B)^{2d-6}}{L^{4}}
            \iint \|\Lambda \tilde{\Phi}(\cdot,w)\|_{L^\infty} \int \prod_{1\le j\le 3} |G_h(\alpha_j^2 \sigma, \alpha_j \tilde{\zeta}_j,p_j) G_h(-\beta_j^2 \sigma, \beta_j \tilde{\eta}_j,q_j)| \diff \tau \diff p \diff q \\
            & \lesssim \ell(B)^{2d-2} \|\Lambda \tilde{\Phi}\|_{L^\infty L^1}
            \lesssim \ell(B)^{2d} \|\Phi\|_{L^\infty W^{2d,1}}.
        \end{aligned}
    \end{equation*}

    To prove~\eqref{eq:gauss-sum-upper-bound}, let $\chi \in C_c^\infty(X^{n,d})$ be such that $\chi \Phi = \Phi$.
    For all $L$ and all  $L^\delta \ge \langle \ell(Q) \rangle$ $\bm{\xi} \in (\mathbb{R}^{2d})^n$, we have the following estimate and conclude by Lemma~\eqref{lem:gauss-sum-fourier-decomp}:
    \begin{equation*}
        |\mathcal{S}_L(\psi_{\bm{\xi}} \theta_L \widehat{\Phi}^\chi_{\bm{\xi}})|
        \lesssim C^n \ell(B)^{2nd}\|\widehat{\Phi}^\chi_{\bm{\xi}}\|_{L^\infty W^{2d,1}}
        \lesssim C^n \ell(B)^{4nd}\|\chi\|_{L^\infty W^{2d,\infty}} \|\widehat{\Phi}(\cdot,\bm{\xi})\|_{L^\infty}.
        \qedhere
    \end{equation*}
\end{proof}

When $L=\infty$, we estimate the oscillatory integral $\mathcal{S}_\infty$ functional via integration by parts.

\begin{proposition}
    \label{prop:gauss-int-upper-bound}
    Let $d \ge 1$ and $n \ge 1$.
    If $\Phi \in L^\infty W^{2n,\infty}$ and suppose that $\supp \Phi \subset \mathbb{R}^n \times B$ for some bounded subset $B \subset (\mathbb{R}^{2d})^n$, then
    \begin{equation}
        \label{eq:gauss-int-upper-bound-pre}
        \sup_{\bm{\xi} \in \mathbb{R}^{nd}} |\mathcal{S}_\infty(\psi_{\bm{\xi}} \Phi)| \lesssim C^n \ell(B)^{2nd} \|\Phi\|_{L^\infty W^{2n,\infty}}
    \end{equation}
    for some universal constant $C > 0$.
    Consequently, if $\widehat{\Phi} \in L^\infty L^1$ and $\supp \Phi \subset \mathbb{R}^n \times B$, then
    \begin{equation}
        \label{eq:gauss-int-upper-bound}
        |\mathcal{S}_\infty(\Phi)|
        \lesssim C^n \ell(B)^{2nd} \|\widehat{\Phi}\|_{L^\infty L^1}.
    \end{equation}
\end{proposition}
\begin{proof}
    For all $\bm{s} \in \mathbb{R}^n \backslash \{0\}$ denote $\bm{s}^{-1} = (\bm{s}_j^{-1})_{1\le j \le n}$.
    For all $\bm{\xi} = (\bm{\zeta}_j,\bm{\eta}_j)_{1 \le j \le n} \in (\mathbb{R}^{2d})^n$ denote $\tilde{\bm{\xi}} = (\bm{\eta}_j,\bm{\zeta}_j)_{1 \le j \le n} \in (\mathbb{R}^{2d})^n$.
    Then $ \mathcal{S}_\infty(\psi_{\bm{\xi}}\Phi) = \mathcal{S}_\infty(\Phi_{\bm{\xi}})$, where
    \begin{equation*}
        \Phi_{\bm{\xi}}(\bm{s},\bm{z}) = e^{2\pi i \bm{s}^{-1} \cdot \varpi(\bm{\xi})} \Phi(\bm{s},\bm{z} - \bm{s}^{-1} \cdot \tilde{\bm{\xi}}/2).
    \end{equation*}
    Define the differential operator $\Lambda = \prod_{j = 1}^n \bigl(1 - \frac{\Delta_{\bm{z}_j}}{4\pi^2}\bigr)$, then we have the identity
    \begin{equation*}
        e^{-2\pi i \bm{s} \cdot \varpi(\bm{z})} \Lambda e^{2\pi i \bm{s} \cdot \varpi(\bm{z})} =\prod_{1\le j \le n}  \langle \bm{s}_j |\bm{z}_j| \rangle^2.
    \end{equation*}
    Integrate by part with respect to $\bm{z}$ using the identity above and perform the change of variables $\bm{s}_j \mapsto \bm{s}_j / |\bm{z}_j|$ for all $1 \le j \le n$.
    We obtain the estimate~\eqref{eq:gauss-int-upper-bound-pre}
    \begin{align*}
        |\mathcal{S}_\infty(\Phi_{\bm{\xi}})|
        & \le \iint \frac{|\Lambda \Phi_{\bm{\xi}}(\bm{s},\bm{z})|}{\prod_{j=1}^n  \langle \bm{s}_j |\bm{z}_j| \rangle^2} \diff \bm{s} \diff \bm{z}                                                     \le \|\Lambda \Phi_{\bm{\xi}}\|_{L^\infty L^\infty} \iint \frac{\bm{1}_B(\bm{z} - \bm{s}^{-1} \cdot \tilde{\bm{\xi}}/2)}{\prod_{j=1}^n |\bm{z}_j| \langle \bm{s}_j\rangle^2} \diff \bm{s} \diff \bm{z} \\
        & \le \pi^n \|\Lambda \Phi_{\bm{\xi}}\|_{L^\infty L^\infty} \int \frac{\bm{1}_{|\bm{z}_j| \le \ell(B)}}{\prod_{j=1}^n |\bm{z}_j| } \diff \bm{z} 
        \lesssim \pi^n \|\Phi\|_{L^\infty W^{2n,\infty}} \prod_{1\le j \le n} \int_{|\bm{z}_j|\le \ell(B)}  |\bm{z}_j|^{-1} \diff \bm{z}_j\\
        & \lesssim \pi^n \|\Phi\|_{L^\infty W^{2n,\infty}}    \Bigl( \int_0^{\ell(B)} r^{2d-2} \diff r\Bigr)^n
        \lesssim \pi^n \ell(B)^{2nd} \|\Phi\|_{L^\infty W^{2n,\infty}} .
    \end{align*}
    To prove~\eqref{eq:gauss-int-upper-bound}, it suffices to choose $\chi \in C_c^\infty((\mathbb{R}^{2d})^n)$ such that $\chi \bm{1}_B = \chi$ and $\diam \supp \chi \le \ell(B)$.
    Therefore $\chi \Phi = \Phi$ and we conclude using Lemma~\ref{lem:gauss-sum-fourier-decomp}.
\end{proof}

Next we prove the localization property (at $\bm{s} = 0$) of the oscillatory integral $\mathcal{S}_\infty(\theta_L \Phi)$ as $L \to \infty$.

\begin{lemma}
    \label{lem:gauss-int-trace}
    Let $n \ge 1$ and $d \ge 1$.
    If $\Phi \in W^{1,\infty} W^{2n,\infty}$ and suppose that $\supp \Phi \subset \mathbb{R}^n \times B$ for some bounded subset $B \subset (\mathbb{R}^{2d})^n$ and suppose that $\theta_\infty \Phi = 0$.
    Then there exists $\nu > 0$ such that, for all $L \ge 1$ and all $\varphi \in L^\infty(\mathbb{R}^n)$, we have
    \begin{equation}
        \label{eq:gauss-int-trace-pre}
        \sup_{\bm{\xi} \in (\mathbb{R}^{2d})^n}|\mathcal{S}_\infty(\varphi \psi_{\bm{\xi}} \theta_L \Phi) |
        \lesssim C^n L^{-\nu} \ell(B)^{2nd} \|\varphi\|_{L^\infty} \|\Phi\|_{W^{1,\infty} W^{2n,\infty}}
    \end{equation}
    for some universal constant $C > 0$.
    Consequently, if $\widehat{\Phi} \in W^{1,\infty} L^1$ and $\supp \Phi \subset \mathbb{R}^n \times B$, then
    \begin{equation}
        \label{eq:gauss-int-trace}
        |\mathcal{S}_\infty(\varphi \theta_L \Phi)| \lesssim C^n L^{-\nu} \ell(B)^{2nd} \|\varphi\|_{L^\infty} \|\widehat{\Phi}\|_{W^{1,\infty} L^1}.
    \end{equation}
\end{lemma}
\begin{proof}
    By~\eqref{eq:gauss-int-upper-bound-pre}, we may restrict ourselves to the case where $n = 1$ hen estimate~\eqref{eq:gauss-int-trace-pre}.
    Following the proof of Proposition~\ref{prop:gauss-int-upper-bound}, it suffices to estimate $\mathcal{S}_\infty(\varphi \theta_L \Phi_\xi)$.
    Let $\varepsilon \in (0,\alpha/2)$ and let $\Lambda$ be defined as in the proof Proposition~\ref{prop:gauss-int-upper-bound}.
    Using the finite increment estimate and the assumption $ \theta_\infty \Phi(\cdot) = 0$, we obtain the upper bound
    \begin{equation*}
        |\Lambda \Phi_\xi(s/\gamma,z)| \le
        \begin{cases}
            \|\Phi\|_{L^\infty W^{2,\infty}}, & |s| \ge L^{\varepsilon};\\
            \gamma^{-1} |s| \|\partial_s \Phi\|_{L^\infty W^{2,\infty}}, & |s| \le L^{\varepsilon},
        \end{cases}
    \end{equation*}
    which yields the following estimate as in the proof of Proposition~\ref{prop:gauss-int-upper-bound}:
    \begin{equation*}
        \begin{aligned}
            |\mathcal{S}_\infty(\varphi \theta_L \Phi_{\xi} )|
            & \lesssim \|\varphi\|_{L^\infty}  \iint \frac{|\Lambda \Phi_\xi(s/\gamma,z)| }{1+s^2|z|^2} \diff s \diff z\\
            & \lesssim  \|\varphi\|_{L^\infty} \|\Phi\|_{W^{1,\infty} W^{2,\infty}}
            \int_{|z|\le \ell(B)} \int \frac{\bm{1}_{|s|\ge L^{\varepsilon}} + \gamma^{-1}|s| \bm{1}_{|s|\le L^{\varepsilon}}}{1+s^2|z|^2} \diff s \diff z \\
            & \lesssim  \|\varphi\|_{L^\infty} \|\Phi\|_{W^{1,\infty} W^{2,\infty}}
            \int_0^{\ell(B)} \int \frac{1_{|s| \ge L^\varepsilon}+ \gamma^{-1}|s|\bm{1}_{|s|\le L^\varepsilon}}{1+s^2 r^2} r^{2d-1} \diff s \diff r.
        \end{aligned}
    \end{equation*}
    We split the integral into the sum of two terms and estimate them separately:
    \begin{align*}
        \int_0^{\ell(B)} & \int \frac{\bm{1}_{|s| \ge L^\varepsilon}}{1+s^2 r^2} r^{2d-1}\diff s \diff r
        \lesssim \int_0^{\ell(B)} \int \frac{\bm{1}_{|s| \ge L^\varepsilon r}}{1+s^2} r^{2d-2}\diff s \diff r      \\
        & \lesssim \int_0^{\ell(B)} \Bigl(\frac{\pi}{2} - \arctan(L^\varepsilon r) \Bigr) r^{2d-2} \diff r 
        \lesssim L^{-\varepsilon} \int_0^{\ell(B)} r^{2d-3} \diff r
        \lesssim L^{-\varepsilon} \ell(B)^{2d-2}, \\
        \int_0^{\ell(B)} &\int \frac{|s| \bm{1}_{|s| \le L^\varepsilon}}{1+s^2r^2} r^{2d-1} \diff s \diff r
        \lesssim \int_0^{\ell(B)} \int \frac{|s| \bm{1}_{|s| \le L^\varepsilon r}}{1+s^2} r^{2d-3} \diff s \diff r \\
         & \lesssim \int_0^{\ell(B)} \log(1+L^{2\varepsilon} r^2) r^{2d-3} \diff r 
         \lesssim L^{2\varepsilon} \int_0^{\ell(B)} r^{2d-1} \diff r
        \lesssim L^{2\varepsilon} \ell(B)^{2d}.
    \end{align*}
    Therefore, choosing $\nu = \min\{\varepsilon,\alpha-2\varepsilon\} > 0$, we conclude with the estimate
    \begin{equation*}
        |\mathcal{S}_\infty(\theta_L \varphi \Phi_{\xi})| \lesssim  \|\varphi\|_{L^\infty} \|\Phi\|_{W^{1,\infty} W^{2,\infty}} (L^{-\varepsilon} + \gamma^{-1} L^{2\varepsilon}) \ell(B)^{2d}.
    \end{equation*}
    The estimate~\eqref{eq:gauss-int-trace} follows from Lemma~\ref{lem:gauss-sum-fourier-decomp} similarly as for the proof of~\eqref{eq:gauss-int-upper-bound}.
\end{proof}

\subsection{Integral approximation}

We now prove the convergence of the Riemann sum functional $\mathcal{S}_L \comp \theta_L$ to the integral functional $\mathcal{S}_\infty \comp \theta_\infty$ as $L \to \infty$.
As Proposition~\ref{prop:gauss-sum-upper-bound} implies, the integral $\mathcal{S}_\infty(\theta_\infty \Phi)$ is well-defined if $\widehat{\Phi}(0,\cdot) \in L^1$.
To prove the convergence, we adapt the number theoretical ideas used in~\cite{Deng2021} to our situation where general scaling laws are considered.
In the following proposition, the weighted Sobolev space $L^1_1 = W^{0,1}_1$ will be used.
Generally, if $m,\ell \ge 0$, $p \in (1,\infty]$ and $N \in \mathbb{N}$, then the weighted Sobolev space $W^{m,p}_\ell(\mathbb{R}^N)$\index{Sets and spaces!Function spaces!Weighted Sobolev spaces $W^{m,p}_{\ell}$} is the completion of $C_c^\infty(\mathbb{R}^N)$ under the norm
\begin{equation*}
    \|\varphi\|_{W^{m,p}_\ell} = \|\langle x \rangle^\ell \varphi(x)\|_{W^{m,p}}.
\end{equation*}

\begin{proposition}
    \label{prop:int-approx}
    Let $d \ge 3$ and $n \ge 1$, and let $\Phi \in \mathscr{S}'(X^{n,d})$ be such $\supp \Phi \subset \mathbb{R}^n \times B$ where $B$ is a bounded subset of $(\mathbb{R}^{2d})^n$ and $\widehat{\Phi} \in W^{1,\infty} L^1_1$, then there exists $C > 0$ and $\nu > 0$ such that
    \begin{equation*}
        |\mathcal{S}_L(\theta_L\Phi) - \mathcal{S}_\infty(\theta_\infty\Phi)| \lesssim C^n \langle n \rangle^d \ell(B)^{4nd} L^{-\nu} \|\widehat{\Phi}\|_{W^{1,\infty} L^1_1}.
    \end{equation*}
\end{proposition}
\begin{proof}
    Let $\chi \in C_c^\infty(X^{n,d})$ be such that $\chi \Phi = \Phi$.
    Let $\varphi \in C_c^\infty(\mathbb{R}^n)$ be such that $\varphi(\bm{s})=1$ for $|\bm{s}| \le 1/2$ and $\varphi(\bm{s}) = 0$ for $|\bm{s}| \ge 1$.
    Let $\delta \in (0,2-\alpha)$.
    Let $\chi^0,\chi^\infty$ be given respectively by $\chi^0(\bm{s},\cdot) = \varphi(L^{\alpha+\delta-1}\bm{s}/\gamma) \chi(\bm{s}/\gamma,\cdot)$ and $\chi^\infty(\bm{s},\cdot) = (1-\varphi)(L^{\alpha+\delta-1}\bm{s}/\gamma) \chi(\bm{s}/\gamma,\cdot)$.
    Write
    \begin{align*}
        \mathcal{S}_L(\theta_L \Phi) & - \mathcal{S}_\infty(\theta_\infty\Phi)
        = \mathcal{S}_L(\chi^\infty \theta_L\Phi)\\
        & + [\mathcal{S}_L(\chi^0 \theta_L\Phi) - \mathcal{S}_\infty(\chi^0 \theta_L\Phi)] 
        + [\mathcal{S}_\infty(\chi^0 \theta_L\Phi) - \mathcal{S}_\infty(\theta_\infty\Phi)].
    \end{align*}
    We conclude by estimating these three terms on the right hand side separately.
    In fact, by the minor arc estimate (Lemma~\ref{lem:minor-arc-sum}), the major arc estimate (Lemma~\ref{lem:major-arc-sum}) and Lemma~\ref{lem:gauss-int-trace}, there exists $\nu > 0$ such that the following estimates hold:
    \begin{align*}
        |\mathcal{S}_L(\chi^\infty \theta_L\Phi)|
        & \lesssim C^n \langle n \rangle^d \ell(B)^{4nd} L^{-\nu} \|\widehat{\Phi}\|_{L^\infty L^1}, \\
        |\mathcal{S}_L(\chi^0 \theta_L\Phi) - \mathcal{S}_\infty(\chi^0 \theta_L\Phi)|
        & \lesssim C^n \ell(B)^{2nd} L^{-\nu} \|\widehat{\Phi}\|_{L^\infty L^1_1},\\
        |\mathcal{S}_\infty(\chi^0 \theta_L\Phi) - \mathcal{S}_\infty(\theta_\infty\Phi)|
        & \lesssim C^n \ell(B)^{2nd} L^{-\nu} \|\widehat{\Phi}\|_{W^{1,\infty} L^1}.
        \qedhere
    \end{align*}
\end{proof}

\begin{lemma}
    \label{lem:minor-arc-sum}
    Let $d \ge 3$, $n \ge 1$, and let $\Phi \in C^\infty(X^{n,d})$ be such that $\supp \Phi \subset \{1 \ge |\bm{s}| \ge L^{1-\alpha-\delta} \} \times B$ for some $\delta \in (0,2-\alpha)$ where $B$ is bounded, then there exists $\nu > 0$ such that for some universal constant $C > 0$ and sufficiently large $L > 0$, we have
    \begin{equation*}
        \sup_{\bm{\xi} \in (\mathbb{R}^{2d})^n} |\mathcal{S}_L(\psi_{\bm{\xi}} \theta_L \Phi)| \lesssim C^n \langle n\rangle^d \ell(B)^{4nd} L^{-\nu} \|\Phi\|_{L^\infty W^{2n,\infty}}.
    \end{equation*}
    Consequently, if $\supp \Phi \subset \{1 \ge |\bm{s}| \ge L^{1-\alpha-\delta} \} \times B$ and $\widehat{\Phi} \in L^{\infty} L^1$, then
    \begin{equation*}
        |\mathcal{S}(\theta_L \Phi)| \lesssim C^n \langle n\rangle^d \ell(B)^{4nd} L^{-\nu} \|\widehat{\Phi}\|_{L^\infty L^1}.
    \end{equation*}
\end{lemma}
\begin{proof}
    Note that if $|\bm{s}| \ge L^{1-\alpha-\delta}$ then there exists $j \in \{1,\ldots,n\}$ such that $|s_j| \gtrsim L^{1-\alpha-\delta} / n$.
    Therefore, by Proposition~\ref{prop:gauss-sum-upper-bound}, it suffices to prove the lemma when $n=1$, which will follow by showing that the integral $Q_m$ given by~\eqref{eq:int-by-part-gauss-sum} satisfies the estimate $|I| \lesssim L^{-\nu}$ for some $\nu > 0$.

    Observe that by the support condition of $\Phi$, if $(\tau,p,q) \in \supp \tilde{\Phi}$ where $\tilde{\Phi}$ is defined by~\eqref{eq:integrand-gauss-sum-int}, then $L^{\alpha-2} \ge |\tau| \ge L^{-1-\delta}/n$ and
    $|p|+|q| \lesssim \ell(B) L$.
    By the Dirichlet approximation lemma, for such $\tau$ and all $N \in \mathbb{N} \backslash \{0\}$, there exists $u,v \in \mathbb{N}$ with $0 \le u < v \le N $ such that either $u = 0$ or $\mathrm{gcd}(u,v) = 1$ and $|\tau - u/v| < 1/|nv|$.
    We borrow the estimate from~\cite[Lemma 3.18]{Bourgain1993},
    \begin{equation}
        \label{eq:gauss-sum-est-bourgain}
        \sup_{r \in \mathbb{R}} \sup_{h \in \mathbb{Z}} |G_h(\tau,r,N)| \le \frac{N}{\sqrt{v}(1+N|(\tau-u/v)|^{1/2})}.
    \end{equation}
    If $u = 0$, then we use $|\tau| \ge L^{-1-\delta} /n $, $v \ge 1$, and~\eqref{eq:gauss-sum-est-bourgain} to estimate
    \begin{equation*}
        |G_h(\tau,r,N)| \le |\tau|^{-1/2} \le \sqrt{n} L^{1/2+\delta/2}.
    \end{equation*}
    If $u \ne 0$, then $ N \ge v \ge 2$.
    By the triangular inequality, we have
    \begin{equation*}
        1/v \le u/v \le |\tau| + |\tau-u/v| < L^{\alpha-2} + 1/(Nv),
    \end{equation*}
    which implies $v > L^{2-\alpha}(1-1/N) \gtrsim L^{2-\alpha}$.
    In this case, if $N \lesssim \ell(B) L$, then
    \begin{equation*}
        |G_h(\tau,r,N)| \le N/\sqrt{v} \lesssim \ell(B) L / L^{1-\alpha/2} = \ell(B) L^{\alpha/2}.
    \end{equation*}
    Therefore, in both cases, if $N \lesssim \ell(B) L$, then
    \begin{equation*}
        |G_h(\tau,r,N)| \lesssim \langle n \rangle^{1/2} \ell(B) L^{\max\{1+\delta,\alpha\}/2} = \langle n \rangle^{1/2} \ell(B) L^{1-\varepsilon},
    \end{equation*}
    where $\varepsilon = \min\{2-\alpha,1-\delta\} / 2 > 0$.
    By Lemma~\ref{lem:gauss-sum-def-est} and H\"older's inequality, we have
    \begin{align*}
        |Q_m|
        & \lesssim (\langle n \rangle^{1/2} \ell(B) L^{1-\varepsilon})^{2d-2}
        \frac{L^{2}}{L^{2d}}
        \iint \|\Lambda \tilde{\Phi}(\cdot,p,q)\|_{L^\infty}\\
        & \qquad \qquad \int \prod_{1\le j \le 2} |G_h(\alpha_j^2 \tau, \alpha_j \tilde{\zeta}_j,p_j) G_h(-\beta_j^2 \tau, \beta_j \tilde{\eta}_j,q_j)| \diff \tau \diff p \diff q  \\
        & \lesssim \langle n \rangle^{d-1} \ell(B)^{2d-2} L^{2 + (2d-4)(1-\varepsilon) + 2-2d} (1 + \log\ell(B) + \log L) \|\Lambda \tilde{\Phi}\|_{L^\infty L^1} \\
        & \lesssim \langle n \rangle^d \ell(B)^{4d} L^{- 2(d-2)\varepsilon} (1 + \log\ell(B) + \log L) \|\Phi\|_{L^\infty W^{2,\infty}},
    \end{align*}
    which concludes the proof of the lemma.
\end{proof}

\begin{lemma}
    \label{lem:major-arc-sum}
    Let $d \ge 3$, $n \ge 1$, and let $\Phi \in C^\infty(X^{n,d})$ be such that $\supp \Phi \subset \{|\bm{s}| \le L^{1-\delta}\} \times B$ where $B$ is bounded.
    Let $N \in \mathbb{N} $ with  $N > \max\{d,(1-\delta)/(1-\delta/2)\}$ and set $\nu = (1-\delta/2) N - (1-\delta) > 0$.
    Then for all $\bm{\xi} \in (\mathbb{R}^{2d})^n$ and for all $L \gg \ell(B)^{\delta/2}$, we have the uniform estimate
    \begin{equation}
        \label{eq:gauss-sum-est-major-arc}
        |\mathcal{S}_L(\psi_{\bm{\xi}} \Phi) - \mathcal{S}_\infty(\psi_{\bm{\xi}} \Phi)| \lesssim C^n \ell(B)^{2nd} \|\Phi\|_{L^\infty W^{\max\{2n,N\},\infty}} (L^{-\nu} + \bm{1}_{|\bm{\xi}| \ge L/2}).
    \end{equation}
    Consequently, if $\supp \Phi \subset \{|\bm{s}| \le L^{1-\alpha-\delta}\} \times B$, then
    \begin{equation*}
        |\mathcal{S}_L(\theta_L\Phi) - \mathcal{S}_\infty(\theta_L\Phi)|
        \lesssim C^n \ell(B)^{2nd} L^{-\min\{\nu,1\}} \|\widehat{\Phi}\|_{L^\infty L^1_1}.
    \end{equation*}
\end{lemma}
\begin{proof}
    For $\bm{\xi} \in (\mathbb{R}^{2d})^n$, $\bm{q} \in (\mathbb{Z}^{2d})^n$, denote $\bm{\omega} = \bm{\xi}+L\bm{q}$.
    By the Poisson summation formula,
    \begin{equation}
        \label{eq:poisson-summation}
        \mathcal{S}_L(\psi_{\bm{\xi}} \Phi)
        = \sum_{\bm{q} \in (\mathbb{Z}^{2d})^n}
        \mathcal{S}_\infty(\psi_{\bm{\omega}} \Phi)
        = \sum_{\bm{q} \in (\mathbb{Z}^{2d})^n}
        \iint e^{2\pi i \varphi_{\bm{\omega}}(\bm{z})} \Phi(\bm{s},\bm{z}) \diff \bm{s} \diff \bm{z},
    \end{equation}
    where $\varphi_{\bm{\omega}}(\bm{z}) = \varpi(\bm{z}) + \bm{z} \cdot \bm{\omega}$.
    Note that the term with $\bm{q} = 0$ in the summation~\eqref{eq:poisson-summation} is equal to the integral $\mathcal{S}_\infty(\psi_{\bm{\xi}} \Phi)$.
    Therefore, it suffices to estimate the terms with $\bm{q} \ne 0$.

    By the translation invariance of the norm on the right hand side of~\eqref{eq:gauss-sum-est-major-arc}, we may assume that $\supp \Phi \subset \mathbb{R}^n \times \{\bm{z} \in (\mathbb{R}^{2d})^n : |\bm{z}| \le 2 \ell(B)\}$.
    Let $\chi \in C_c^\infty((\mathbb{R}^{2d})^n)$ be such that $1-\chi$ vanishes near the origin, and let $\chi_L(\bm{\omega}) = \chi(\bm{\omega}/L^{1-\delta/2})$.
    We shall estimate the two integrals $(1-\chi_L)(\bm{\omega}) \mathcal{S}_\infty(\psi_{\bm{\omega}} \Phi )$ and $\chi_L(\bm{\omega}) \mathcal{S}_\infty(\psi_{\bm{\omega}} \Phi )$ separately.
    Note that if $\bm{q} \ne 0$, $\bm{\omega} \in \supp \chi_L$, and $(\bm{s},\bm{z}) \in \supp \psi_{\bm{\xi}} \Phi$, then for $L \gg \ell(B)^{\delta/2}$, we have
    \begin{equation*}
        |\nabla_{\bm{z}} \varphi_{\bm{\omega}}|
    \gtrsim 2|\bm{\omega}| - |\bm{s}| |\bm{z}|
    \gtrsim |\bm{\omega}| + L^{1-\delta/2} - L^{1-\delta} \ell(B) \gtrsim |\bm{\omega}| + L^{1-\delta/2}
    \end{equation*}
    and $|\nabla_{\bm{z}}^2 \varphi_{\bm{\omega}}| \lesssim |\bm{s}| \lesssim L^{1-\delta} \le |\nabla_{\bm{z}} \varphi_{\bm{\omega}}|.$
    By the stationary phase method, for all $N \in \mathbb{N}$, we have
    \begin{equation*}
        \begin{aligned}
            |(1-\chi_L)(\bm{\omega}) \mathcal{S}_\infty(\psi_{\bm{\omega}} \Phi )|
            & \lesssim L^{1-\delta} (|\bm{\omega}| + L^{1-\delta/2})^{-N} \|\Phi\|_{L^\infty W^{N,1}} \\
            & \lesssim L^{1-\delta-(1-\delta/2)N}
            (L^{\delta/2} |\bm{q}+\bm{\xi}/L| + 1)^{-N} \|\Phi\|_{L^\infty W^{N,1}}.
        \end{aligned}
    \end{equation*}
    Let $N > \max\{d,(1-\delta)/(1-\delta/2)\}$, then $\nu = (1-\delta/2) N - (1-\delta) > 0$.
    Summing up in $\bm{q} \ne 0$, we have the following estimate which is uniform in $\bm{\xi} \in (\mathbb{R}^{2d})^n$:
    \begin{equation*}
        \begin{aligned}
            \sum_{\bm{q} \in (\mathbb{Z}^{2d})^n \backslash \{0\}}
            |(1 &-\chi_L)(\bm{\omega}) \mathcal{S}_L(\psi_{\bm{\omega}} \Phi )|
            \lesssim L^{-\nu} \Bigl( \sum_{\bm{q} \in (\mathbb{Z}^{2d})^n \backslash \{0\}} (L^{\delta/2} |\bm{q}+\bm{\xi}/L| + 1)^{-N}\Bigr) \|\Phi\|_{L^\infty W^{N,1}}  \\
            & \lesssim L^{-\nu} \Bigl( \sum_{\bm{q} \in (\mathbb{Z}^{2d})^n} (|\bm{q}| + 1)^{-N}\Bigr) \|\Phi\|_{L^\infty W^{N,1}}
            \lesssim L^{-\nu} \ell(B)^{2nd} \|\Phi\|_{L^\infty W^{N,\infty}}.
        \end{aligned}
    \end{equation*}
    On the other hand, by Proposition~\ref{prop:gauss-int-upper-bound}, we have
    \begin{equation*}
        |\chi_L(\bm{\omega}) \mathcal{S}_L(\psi_{\bm{\omega}} \Phi)|
        \lesssim C^n \ell(B)^{2nd} |\chi_L(\bm{\omega})|  \| \Phi\|_{L^\infty W^{2n,\infty}}
        \lesssim C^n \ell(B)^{2nd}  \bm{1}_{|\bm{q}+\bm{\xi}/L|\lesssim L^{-\delta/2}} \| \Phi\|_{L^\infty W^{2n,\infty}}.
    \end{equation*}
    If $\bm{q} \ne 0$ and $\bm{1}_{|\bm{q}+\bm{\xi}/L|\lesssim L^{-\delta/2}} \ne 0$, then $|\bm{\xi}|/L \ge |\bm{q}| - L^{-\delta/2} \ge 1/2$.
    Moreover, for all $\bm{\xi} \in (\mathbb{R}^{2d})^n$, the number of $\bm{q} \in (\mathbb{Z}^{2d})^n $ such that $|\bm{q}+\bm{\xi}/L|\lesssim L^{-\delta/2}$ is at most $\mathcal{O}(1)$.
    Consequently, 
    \begin{equation*}
        \sum_{\bm{q} \in (\mathbb{Z}^{2d})^n \backslash \{0\}} |\chi_L(\bm{\omega}) \mathcal{S}_L(\psi_{\bm{\omega}} \Phi)|
        \lesssim C^n \ell(B)^{2nd} \bm{1}_{|\bm{\xi}|\ge L/2} \| \Phi\|_{L^\infty W^{2n,\infty}}.
    \end{equation*}
    The last estimate follows by Lemma~\ref{lem:gauss-sum-fourier-decomp} and $\bm{1}_{|\bm{\xi}| \ge L/2} \le 2 |\bm{\xi}|/L$.
\end{proof}

\section{Lattice counting estimates}

\label{sec:lattice-count-est}

The purpose of this section is to obtain a lattice counting estimate (Proposition~\ref{prop:deco-count-cp}) for \emph{resonance functions} (Definition~\ref{def:reso-function}) on couples.
Resonance functions are phase factors that appear in the diagrammatic expansions (see \S\ref{sec:cauchy}) for solutions to WNLS. We will use their counting estimate to obtain upper bounds for couples and show that non-regular couples give negligible contributions to the effective dynamics of WNLS.

\subsection{Resonance functions}

We now define resonance functions on trees and couples.
Using the topological structure introduced in \S\ref{sec:topo}, we characterize linear relations among resonance functions on couples.
We shall henceforth adopt the following shorthand notations, which will be used throughout \S\ref{sec:lattice-count-est}, \S\ref{sec:diagram}, \S\ref{sec:hom}, and \S\ref{sec:inhom}:
If $\tau \in \mathscr{T}$, $\tau^* \in \mathscr{T}^*$, $\cp \in \mathscr{K}$ and $L \in (0,\infty)$,
then we denote $\mathscr{D}^\tau = \mathscr{D}^\tau(\mathbb{R}^d)$, $\mathscr{D}^{\tau,L} = \mathscr{D}^\tau(\mathbb{Z}^d_L)$, $\mathscr{D}^{\tau^*} = \mathscr{D}^{\tau^*}(\mathbb{R}^d)$, $\mathscr{D}^{\tau^*,L} = \mathscr{D}^{\tau^*}(\mathbb{Z}^d_L)$, $\mathscr{D}^\cp = \mathscr{D}^\cp(\mathbb{R}^d)$, $\mathscr{D}^{\cp,L} = \mathscr{D}^\cp(\mathbb{Z}^d_L)$.
In this convention, for all $ k \in \mathbb{R}^d $, we denote $\mathscr{D}^\tau_k = \mathscr{D}^\tau_k(\mathbb{R}^d)$, $\mathscr{D}^{\tau^*}_k = \mathscr{D}^{\tau^*}_k(\mathbb{R}^d)$ and $\mathscr{D}^\cp_k = \mathscr{D}^\cp_k(\mathbb{R}^d)$;
for all $k \in \mathbb{Z}^d_L$, we denote $\mathscr{D}^{\tau,L}_k = \mathscr{D}^{\tau}_k(\mathbb{Z}^d_L)$, $\mathscr{D}^{\tau^*,L}_k = \mathscr{D}^{\tau^*}_k(\mathbb{Z}^d_L)$ and $\mathscr{D}^{\cp,L}_k = \mathscr{D}^\cp_k(\mathbb{Z}^d_L)$.\index{Trees and couples!Decorations!Decorations valued in $\mathbb{R}^d$ or $\mathbb{Z}^d_L$}

\label{sec:reso-function}

\begin{definition}
    \label{def:reso-function}
    If $\tau \in \mathscr{T} \backslash \mathscr{T}_0$, then the resonance function on $\tau$ is the bilinear map $\Omega^\tau = (\Omega^\tau_{\mathfrak{b}})_{\mathfrak{b} \in \mathfrak{B}^\tau}: \mathscr{D}^\tau \times \mathscr{D}^\tau \to \mathbb{R}^{\mathfrak{B}^\tau}$ defined as follows: if $\mathfrak{b} \in \mathfrak{B}^\tau$ and $\bm{\zeta},\bm{\eta} \in \mathscr{D}^\tau$, then
    \begin{equation*}
        \Omega^\tau_{\mathfrak{b}}(\bm{\zeta},\bm{\eta})
        = \frac{1}{2} \Bigl( \iota_{\mathfrak{b}} \bm{\zeta}_{\mathfrak{b}} \cdot \bm{\eta}_{\mathfrak{b}} - \sum_{\mathfrak{n}^p = \mathfrak{b}} \iota_{\mathfrak{n}} \bm{\zeta}_{\mathfrak{n}} \cdot \bm{\eta}_{\mathfrak{n}}\Bigr).
        \index{Functions and random variables!Resonance functions!Resonance functions on trees and couples $\Omega^{\tau}$, $\Omega^\cp$}
    \end{equation*}
    If $\tau^* \in \mathscr{T}^*$, then the resonance function on $\tau^*$ is the bilinear map $\Omega^{\tau^*} = (\Omega^{\tau^*}_{\mathfrak{b}})_{\mathfrak{b} \in \mathfrak{B}^{\tau^*}} : \mathscr{D}^{\tau^*} \times \mathscr{D}^{\tau^*} \to \mathbb{R}^{\mathfrak{B}^{\tau^*}}$ such that 
    $\Omega^{\tau^*}_{\mathfrak{b}}(\bm{\zeta},\bm{\eta})
    = \Omega^{\tau^\pm}_{\mathfrak{b}}(\bm{\zeta}|_{\tau^\pm},\bm{\eta}|_{\tau^\pm})$ for all $\mathfrak{b} \in \tau^\pm$ and for all $\bm{\zeta},\bm{\eta} \in \mathscr{D}^{\tau^*}$.
    If $\cp = (\tau^*,\wp) \in \mathscr{K}$, then the resonance function on $\cp$ is the bilinear map $\Omega^{\cp} = (\Omega^{\cp}_{\mathfrak{b}})_{\mathfrak{b} \in \mathfrak{B}^\cp} : \mathscr{D}^\cp \times \mathscr{D}^\cp \to \mathbb{R}^{\mathfrak{B}^\cp}$ defined by setting $\Omega^{\cp} = \Omega^{\tau^*}|_{\mathscr{D}^\cp \times \mathscr{D}^\cp}$.
\end{definition}

We shall always use the convention $\Omega(\bm{\zeta}) = \Omega(\bm{\zeta},\bm{\zeta})$ if $\Omega$ is a bilinear map.

\begin{lemma}
    \label{lem:reso-form-linear-indep}
    Let $\cp \in \mathscr{K} \backslash \mathscr{K}_0$.
    Then for all map $\bm{\mu}: \mathfrak{B}^\cp \to \mathbb{C}$, we have $ \bm{\mu} \cdot \Omega^\cp = 0$ if and only if $\bm{\mu}$ is locally constant with respect to the topology $\mathfrak{I}_\cp$ given in Definition~\ref{def:cp-topo}.
\end{lemma}
\begin{proof}
    By Lemma~\ref{lem:deco-formula}, if $\bm{\zeta},\bm{\eta} \in \mathscr{D}^\cp$ and $\mathfrak{b} \in \mathfrak{B}^\cp$, then
    \begin{align*}
        \Omega^\cp_{\mathfrak{b}}(\bm{\zeta},\bm{\eta})
        & = \frac{1}{2} \Bigl( \iota_{\mathfrak{b}}\sum_{\mathfrak{p}_1,\mathfrak{p}_2 \in \wp} \Gamma^\cp_{\mathfrak{b}}(\mathfrak{p}_1,\mathfrak{p}_2) \bm{\zeta}^\flat_{\mathfrak{p}_1} \cdot \bm{\eta}^\flat_{\mathfrak{p}_2} - \sum_{\mathfrak{n}^p = \mathfrak{b}} \iota_{\mathfrak{n}} \sum_{\mathfrak{p}_1,\mathfrak{p}_2 \in \wp} \Gamma^\cp_{\mathfrak{n}}(\mathfrak{p}_1,\mathfrak{p}_2) \bm{\zeta}^\flat_{\mathfrak{p}_1} \cdot \bm{\eta}^\flat_{\mathfrak{p}_2}  \Bigr)\\
        & = \frac{1}{2} \sum_{\mathfrak{p}_1,\mathfrak{p}_2 \in \wp} \Bigl(\iota_{\mathfrak{b}} \Gamma^\cp_{\mathfrak{b}} - \sum_{\mathfrak{n}^p = \mathfrak{b}} \iota_{\mathfrak{n}} \Gamma^\cp_{\mathfrak{n}}\Bigr)(\mathfrak{p}_2,\mathfrak{p}_2) \bm{\zeta}^\flat_{\mathfrak{p}_1}  \cdot  \bm{\eta}^\flat_{\mathfrak{p}_2} ,
    \end{align*}
    where $(\Gamma^\cp_{\mathfrak{n}})_{\mathfrak{n} \in \cp}$ is defined in Lemma~\ref{lem:topo-linear-char}.
    Therefore $\bm{\mu} \cdot \Omega^\cp = 0$ if and only if
    \begin{equation*}
        \sum_{\mathfrak{b} \in \mathfrak{B}^\cp} \bm{\mu}_{\mathfrak{b}}
        \Bigl(\iota_{\mathfrak{b}} \Gamma^\cp_{\mathfrak{b}} - \sum_{\mathfrak{n}^p = \mathfrak{b}} \Gamma^\cp_{\mathfrak{n}}\Bigr)(\mathfrak{p}_2,\mathfrak{p}_2) = 0
    \end{equation*}
    for all $\mathfrak{p}_1,\mathfrak{p}_2 \in \wp$.
    We conclude by Corollary~\ref{cor:reso-matrix-linear-rel}.
\end{proof}

\begin{definition}
    \label{def:edge-momentum-var}
    If $\tau \in \mathscr{T} \backslash \mathscr{T}_0$, then the linear map $\Lambda^\tau: \mathscr{D}^\tau  \to (\mathbb{R}^{2d})^{\mathfrak{B}^\tau}$ is defined by letting $\Lambda^\tau \bm{\zeta} = \bm{z} = (\bm{x}_{\mathfrak{b}},\bm{y}_{\mathfrak{b}})_{\mathfrak{b} \in \mathfrak{B}^\tau}$ where $\iota_{\mathfrak{b}}\bm{x}_{\mathfrak{b}} = \bm{\zeta}_{\mathfrak{b}} - \bm{\zeta}_{\mathfrak{b}[1]}$ and $\iota_{\mathfrak{b}}\bm{y}_{\mathfrak{b}} = \bm{\zeta}_{\mathfrak{b}} - \bm{\zeta}_{\mathfrak{b}[3]}$.\index{Operations and transforms!Change of variables!Change of variables in decorations $\Lambda^\tau$, $\Lambda^\cp$}
    If $\cp = (\tau^*,\wp) \in \mathscr{K} \backslash \mathscr{K}_0$, then the linear maps $\Lambda^{\cp,\pm} : \mathscr{D}^\cp \to (\mathbb{R}^{2d})^{\mathfrak{B}^{\tau^\pm}}$ are defined by setting $\Lambda^{\cp,\pm} \bm{\zeta} = \Lambda^{\tau^\pm} (\bm{\zeta}|_{\tau^\pm})$.
    We also denote $\Lambda^\tau_\zeta = \Lambda^\tau|_{\mathscr{D}^\tau_\zeta}$ and $\Lambda^{\cp,\pm}_\zeta = \Lambda^{\cp,\pm}|_{\mathscr{D}^\cp_\zeta}$ for all $\zeta \in \mathbb{R}^d$.
\end{definition}

Note that if $\bm{\zeta} \in \mathscr{D}^\tau$ and $\Lambda^\tau \bm{\zeta} = (\bm{x}_{\mathfrak{b}},\bm{y}_{\mathfrak{b}})_{\mathfrak{b} \in \mathfrak{B}^\tau}$, then for all $\mathfrak{b} \in \mathfrak{B}^\tau$, we have
\begin{equation}
    \label{eq:Omega-change-var}
    \Omega^\tau_{\mathfrak{b}}(\bm{\zeta}) = \iota_{\mathfrak{b}} \bm{x}_{\mathfrak{b}} \cdot \bm{y}_{\mathfrak{b}}.
\end{equation}

\begin{lemma}
    \label{lem:tree-deco-change-var}
    Let $\cp \in \mathscr{K}^\reg \backslash \mathscr{K}_0$ and let $\vartheta^{\cp,\pm}: \mathscr{I}_\cp \to \mathfrak{B}^{\tau^\pm}$ be such that $\mathfrak{B}^{\mathfrak{f}} \cap \mathfrak{B}^{\tau^\pm} = \{\vartheta^{\cp,\pm}_{\mathfrak{f}}\}$ for all $\mathfrak{f} \in \mathscr{I}_\cp$.
    There exists a linear map $\Lambda^\cp : \mathscr{D}^\cp \to (\mathbb{R}^{2d})^{\mathscr{I}_\cp}$ which makes the diagram commute:
    \begin{center}
        \begin{tikzcd}
            0 \arrow[r] & \mathbb{R}^d \arrow[d,equal] \ar[r] & \mathscr{D}^\cp \arrow[r,"\Lambda^{\cp,\pm}"] \arrow[d,equal] &  (\mathbb{R}^{2d})^{\mathfrak{B}^{\tau^\pm}} \arrow[r] & 0 \\
            0 \arrow[r] & \mathbb{R}^d \arrow[r] & \mathscr{D}^\cp \arrow[r,"\Lambda^\cp"] & (\mathbb{R}^{2d})^{\mathscr{I}_\cp} \arrow[r] \arrow[u,"\vartheta^{\cp,\pm}"'] & 0
        \end{tikzcd}
    \end{center}
    where we identify the space of constant $\mathscr{D}$-decorations on $\cp$ with $\mathbb{R}^d$.
    Moreover the horizontal lines in the diagram are short exact sequences.
    In other words, we have $\Lambda^\cp = \Lambda^{\cp,\pm} \comp \vartheta^{\cp,\pm}$ and $\ker \Lambda^{\cp} = \Lambda^{\cp,\pm} = \mathbb{R}^d$.
    We shall also denote $\Lambda^\cp_\zeta = \Lambda^\cp|_{\mathscr{D}^\cp_\zeta}$ for all $\zeta \in \mathbb{R}^d$.
\end{lemma}
\begin{proof}
    It suffices to show that if $\tau \in \mathscr{T} \backslash \mathscr{T}_0$, then we have a short exact sequence
    \begin{center}
        \begin{tikzcd}
            0 \arrow[r] & \mathbb{R}^d \arrow[r] & \mathscr{D}^\tau \arrow[r,"\Lambda^{\tau}"] & (\mathbb{R}^{2d})^{\mathfrak{B}^\tau} \arrow[r] & 0
        \end{tikzcd}
    \end{center}
    The commutative diagram follows from Corollary~\ref{cor:conju-node-char-deco}.
    In fact, if $\zeta \in \mathbb{R}^d$ and $\bm{z} = (\bm{x}_{\mathfrak{b}},\bm{y}_{\mathfrak{b}})_{\mathfrak{b} \in \mathfrak{B}^\tau} \in (\mathbb{R}^{2d})^{\mathfrak{B}^\tau}$, then the unique $\bm{\zeta} \in \mathscr{D}^\tau_\zeta$ such that $\Lambda^\tau \bm{\zeta} = \bm{z}$ is given by
    \begin{equation}
        \label{eq:deco-change-var-formula}
        \bm{\zeta}_{\mathfrak{n}} = \zeta + \sum_{\mathfrak{n} \preceq \mathfrak{m} \ne \mathfrak{r}^\tau} \iota_{\mathfrak{n}} \bm{h}_{\mathfrak{n}},\quad
        \text{where}\ \bm{h}_{\mathfrak{n}} =
        \begin{cases}
            \bm{x}_{\mathfrak{n}^p},                               & \mathfrak{n} = \mathfrak{n}^p[1]; \\
            - \bm{x}_{\mathfrak{n}^p} - \bm{y}_{\mathfrak{n}^p} , & \mathfrak{n} = \mathfrak{n}^p[2]; \\
            \bm{y}_{\mathfrak{n}^p},                               & \mathfrak{n} =  \mathfrak{n}^p[3].
        \end{cases}
        \qedhere
    \end{equation}
\end{proof}

We now state, in the following Proposition~\ref{prop:deco-count-cp}, the counting estimate for decorations on couples.
This estimate \eqref{eq:deco-count-cp} follows directly from \cite[Proposition~8.6]{Deng2023} (by setting $n_{\mathrm{sub}} = \rho_{\mathrm{sub}} = 2n$, $\Delta_{\mathrm{sub}} = q_{\mathrm{sub}} = 0$) and holds true for $d \ge 3$ and $\alpha \in (0,2)$.
However, as this Proposition is an overkill for our purposes, and for the sake of completeness, we shall present in \S\ref{sec:cp-counting-est} a much simpler proof \eqref{eq:deco-count-cp}  that is sufficient for us here and covers a narrower range of parameters where $d \ge 3$ and $0 < \alpha < 2/(1+4/d^*)$ with $d^* = d$ if $d$ is even and $d^* = d+1$ if $d$ is odd.

\begin{proposition}
    \label{prop:deco-count-cp}
    Assume that $d \ge 3$ and $0 < \alpha < 2/(1+4/d^*)$.
    Let $\cp \in \mathscr{K}_{n}$ with $n \in \mathbb{N} \backslash \{0\}$.
    Let $B \subset \mathscr{D}^\cp$ and let $Q \subset \mathbb{R}^{\mathfrak{B}^\cp}$ be bounded sets.
    Assume furthermore that $B$ is bounded by $nR$ for some $R \ge 1$ and $Q$ is bounded by $1$.
    For $k \in \mathbb{Z}^d_L$, let $G^{\cp,L}_k(B,Q)$ be the set of all $\bm{k} \in \mathscr{D}^{\cp,L}_k \cap B$ such that $\gamma \Omega^\cp(\bm{k}) \in Q$, then for some $C > 0$, $\nu > 0$, and $\delta > 0$, the following estimate holds when $L^\delta > n $:
    \begin{equation}
        \label{eq:deco-count-cp}
        \sup_{k \in \mathbb{Z}^d_L} \card G^{\cp,L}_k(B,Q) \lesssim C^n (nR)^{4nd} L^{n(2d-\alpha)} L^{-\nu \ind (q)}.
    \end{equation}
\end{proposition}

\subsection{Geometric counting methods}
\label{sec:geo-counting}

We now prove lattice counting results required in the proof of Proposition~\ref{prop:deco-count-cp}.

\begin{definition}
    \label{def:Z-set}
    If $f : \mathbb{R}^m \to \mathbb{R}^n$ is a map between Euclidean spaces of finite dimensions and $\delta \ge 0$, $r \ge 0$, then we denote by $Z^\delta_r(f)$ the set of all $x \in \mathbb{R}^m$ such that $|f(x)| \le r L^{-\delta}$.
    We also denote $Z^\delta(f) = Z^\delta_1(f)$ for simplicity.
\end{definition}

\begin{lemma}
    \label{lem:lattice-count-coarea}
    Let $u \in W^{1,\infty}(\mathbb{R}^2,\mathbb{R}^2)$ be such that
    \begin{equation*}
        N(u) \coloneqq \sup_{z \in \mathbb{R}^2, Ju(z) \ne 0} \card u^{-1}(z) < \infty,
    \end{equation*}
    where $Ju$ is the Jacobian determinant of $u$.
    Let $\alpha > 0$, $\delta \in (0,1)$, let $B \subset \mathbb{R}^2$ be a bounded set, let $\tilde{B}$ be the $L^{-1}$-neighborhood of $B$, and assume that $L^{1-\delta} \ge \|\nabla Ju\|_{L^\infty(\tilde{B})}$, then
    \begin{equation}
        \label{eq:lattice-est-coarea}
        \card ( Z^\alpha(u) \backslash Z^\delta(Ju) \cap \mathbb{Z}^2_L \cap B) \lesssim L^{2+\delta-2\min\{\alpha,1\}} \langle \|\nabla u\|_{L^\infty(\tilde{B})} \rangle^2 N(u).
    \end{equation}
    Consequently, if $f,g$ are polynomials of at most second degree whose coefficients for second order terms are bounded by $n \in \mathbb{N}$ and whose coefficients for first order terms are bounded by $nR$ where $R \ge 1$, and assume that $B$ is also bounded by $nR$, then when $L^{1-\delta} \gg n^2 R$ we have
    \begin{equation*}
        \card ( Z^\alpha(f) \cap Z^\alpha(g) \backslash Z^\delta(\nabla f \times \nabla g) \cap \mathbb{Z}^2_L \cap B) \lesssim (nR)^4 L^{2+\delta-2\min\{\alpha,1\}}.
    \end{equation*}
\end{lemma}
\begin{proof}
    For all $z \in \mathbb{Z}^2_L \cap B$, let $Q_z \subset \mathbb{R}^2 \cap \tilde{B}$ be the closed square centered at $z$, with edges parallel to the axes and with edge length $L^{-1}$.
    By the finite increment theorem, if in addition $z \in Z^\alpha(u) \backslash Z^\delta(Ju) $ and $L^{1-\delta} \ge \|\nabla Ju\|_{L^\infty(\tilde{B})}$, then
    \begin{equation*}
        Q_z \subset M = Z^{\min\{\alpha,1\}}_{\langle \|\nabla u\|_{L^\infty(\tilde{B})} \rangle}(u) \backslash Z^{\delta}_{1/2}(Ju) \cap \tilde{B}.
    \end{equation*}
    Therefore, observing that $u(M) \subset \{z \in \mathbb{R}^2 : |z| \le \langle \|\nabla u\|_{L^\infty(\tilde{B})} \rangle L^{-\min\{\alpha,1\}}\}$, we conclude~\eqref{eq:lattice-est-coarea} using the area method and the coarea formula:
    \begin{align*}
        L^{-2}&\card\bigl( Z^\alpha(u) \backslash Z^\delta(Ju) \cap \mathbb{Z}^2_L \cap B\bigr)
        \le |M|
        \lesssim L^{\delta} \int \bm{1}_M(x) |\nabla Ju(x)| \diff x \\
        & = L^\delta \int_{\mathbb{R}^2} \card\{x \in M : u(x) = z\}\diff z
        \le L^\delta \int_{\mathbb{R}^2} \bm{1}_{u(M)}(z) \card u^{-1}(z) \diff z \\
        & \le L^\delta |u(M)| N(u)
        \lesssim L^{\delta-2\min\{\alpha,1\}} \langle \|\nabla u\|_{L^\infty} \rangle^2 N(u).\qedhere
    \end{align*}
\end{proof}

\begin{lemma}
    \label{lem:count-poly-quasi-zero}
    Let $\delta \ge \varepsilon > 0$.
    Let $f : \mathbb{R}^2 \to \mathbb{R}$ be a polynomial for at most second degree such that the coefficients of its second degree monomials are integers bounded by $n \in \mathbb{N}$.
    Suppose in addition that if $\Hess_f = 0$, then $|\nabla f(0)| \ge L^{-\varepsilon}$.
    Then for all $B \in \mathbb{R}^2$ bounded by $nR$ where $R \ge 1$, we have
    \begin{equation*}
        \card(Z^{2\delta}(f) \cap B \cap \mathbb{Z}^d_L) \lesssim (n R)^2
        L^{2-\min\{2\varepsilon,2\delta-\varepsilon,1\}}.
    \end{equation*}
\end{lemma}
\begin{proof}
    If $\varphi$ is an invertible linear transform of $\mathbb{R}^2$, then
    \begin{equation*}
        \card(Z^{2\delta}(f) \cap B \cap \mathbb{Z}^2_L)
        = \card(Z^\delta(f \comp \varphi) \cap B_\varphi \cap \Gamma_\varphi ),
    \end{equation*}
    where $B_\varphi = \varphi^{-1}(B)$ and $\Gamma_\varphi = \varphi^{-1}(\mathbb{Z}^d_L)$.
    Choose $\varphi$ such that $|J\varphi| \gtrsim n^{-2}$ if $\Hess_f \ne 0$ and $|J\varphi| \simeq 1$ if $\Hess_f = 0$, and which makes $g = f \comp \varphi$ fall into one of the following categories:
    \begin{enumerate}
        \item $g(x,y) = x^2 \pm y^2 + r$,
        \item $g(x,y) = px + y^2 + r$,
        \item $g(x,y) = px + r$;
    \end{enumerate}
    where $|p| \gtrsim L^{-\varepsilon}$ and $r \in \mathbb{R}$.
    Let $D_\varepsilon$ be the set of all $(x,y) \in \mathbb{R}^2$ such that $\max\{|x|,|y|\} \le 2 L^{-\varepsilon}$.
    Then we decompose
    \begin{equation}
        \label{eq:Z-set-decompose}
        Z^{2\delta}(g) \cap B_\varphi
        \subset \bigl( B_\varphi \cap D_\varepsilon \bigr) \cup
        \bigl( Z^{2\delta}(g) \cap B_\varphi \backslash D_\varepsilon \bigr).
    \end{equation}
    By the area method, we have the estimate
    \begin{equation*}
        \card (B_\varphi \cap D_\varepsilon  \cap \Gamma_\varphi)
        \lesssim n^2  L^{\max\{2-2\varepsilon,0\}} = n^2 L^{2-2\min\{\varepsilon,1\}}.
    \end{equation*}

    To estimate the second term on the right hand side of~\eqref{eq:Z-set-decompose}, we discuss the three categories.
    First assume that $g(x,y) = x^2 \pm y^2 + r$.
    Let $(x,y) \in Z^{2\delta}(g) \backslash D_\varepsilon$ and assume without losing generality that $|x| \ge 2 L^{-\varepsilon}$.
    Write $h(y) = \mp y^2 - r$.
    Then $h(y) = x^2 - g(x,y) \ge x^2 - L^{-2\delta} \ge 3 L^{-2\varepsilon}$.
    Therefore, 
    \begin{equation*}
        \begin{aligned}
            |x| & = \sqrt{h(y)} \sqrt{1 + g(x,y)/h(y)} = \sqrt{h(y)} \bigl(1 + \mathcal{O}(L^{-2\delta}/h(y))\bigr) \\
            & = \sqrt{h(y)} + \mathcal{O}(L^{-2\delta}/\sqrt{h(y)})
            = \sqrt{h(y)} + \mathcal{O}(L^{\varepsilon-2\delta}).
        \end{aligned}
    \end{equation*}
    A similar holds if $|y| \ge 2 L^{-\varepsilon}$.
    Therefore, by the area method, we have
    \begin{equation*}
        \begin{aligned}
            \card(Z_g^\delta \cap B_\varphi \backslash D_\epsilon \cap \Gamma_\varphi)
            & \lesssim \card\{(x,y) \in B_\varphi \cap \Gamma_\varphi: ||x| - h(y)| \lesssim L^{\varepsilon-2\delta}\} \\
            & \lesssim n^2 R L^{1+\max\{1-(2\delta-\varepsilon),0\}}
            \le (nR)^2 L^{2 - \min\{2\delta-\varepsilon,1\}}.
        \end{aligned}
    \end{equation*}
    Assume now that $g(x,y) = px + r$ or $g(x,y) = px + y^2 + r$ and also write $g(x,y) = px - h(y)$.
    If $|px-h(y)| \le L^{-2\delta}$, then by the hypothesis that $p \gtrsim L^{-\varepsilon}$, we have $|x-h(y)/p| \le L^{-2\delta}/|p| \lesssim L^{\varepsilon-2\delta}$.
    Therefore, by the area method, we obtain a similar estimate
    \begin{align*}
        \card(Z_g^\delta \cap B_\varphi \cap \Gamma_\varphi)
        & \le \card\{(x,y) \in B_\varphi \cap \Gamma_\varphi:|x-h(y)| \lesssim L^{-\delta}\}
        \lesssim (nR)^2 L^{2 - \min\{2\delta-\varepsilon,1\}}. \qedhere
    \end{align*}
\end{proof}

\subsection{Counting estimates for decorations}
\label{sec:cp-counting-est}

In this section we prove Proposition~\ref{prop:deco-count-cp}.

\emph{Step 1: Reduction to irreducible couples.}
By Corollary~\ref{cor:conju-node-char-deco}, we may restrict ourselves to the case where $\cp$ is irreducible.
In fact, let $\mathfrak{c} \in \mathfrak{C}^\cp_2$ and let $B_1 \subset \mathscr{D}^{\check{\cp}_{\mathfrak{c}}}$, $B_2 \subset \mathscr{D}^{\hat{\cp}_{\mathfrak{c}}}$, $Q_1 \subset \mathbb{R}^{\mathfrak{B}^{\check{\cp}_{\mathfrak{c}}}}$, and $Q_2 \subset \mathbb{R}^{\mathfrak{B}^{\hat{\cp}_{\mathfrak{c}}}}$ be bounded sets such that $\ell(Q_j) \lesssim 1$ and $B_j$ are bounded by $nR$ for $j = 1,2$, and moreover $B \subset B_1 \times B_2$ and $Q \subset Q_1 \times Q_2$.
Then
\begin{equation*}
    G^{\cp,L}_k \subset \{\bm{k} \in \mathscr{D}^\cp : \bm{k}|_{\check{\cp}_{\mathfrak{c}}} \in \mathscr{D}^{\check{\cp}_{\mathfrak{c}}}_k, \bm{k}|_{\hat{\cp}_{\mathfrak{c}}} \in \mathscr{D}^{\hat{\cp}_{\mathfrak{c}}}_{\bm{k}^\flat_{\mathfrak{c}}}\},
\end{equation*}
which yields the estimate
\begin{equation*}
    \card G^{\cp,L}_k(B,Q)
    \le \card G^{\check{\cp}_{\mathfrak{c}},L}_k(B_1,Q_1) \sup_{k' \in \mathbb{Z}^d_L} \card G^{\hat{\cp}_{\mathfrak{c}},L}_{k'}(B_2,Q_2).
\end{equation*}
Therefore, if~\eqref{eq:deco-count-cp} is obtained for all irreducible couples, then we can use mathematical induction to prove~\eqref{eq:deco-count-cp} for all couples.

\emph{Step 2: Change of variables.}
Assume that $\cp = (\tau^*,\wp)$ is irreducible and suppose that
\begin{equation*}
    Q \subset \prod_{\mathfrak{b} \in \mathfrak{B}^{\tau^+}} Q^+_{\mathfrak{b}} \times \prod_{\mathfrak{b} \in \mathfrak{B}^{\tau^-}} Q^-_{\mathfrak{b}},
    \quad
    \Lambda^{\cp,+} B \cup \Lambda^{\cp,-} B
    \subset \prod_{\mathfrak{b}\in \mathfrak{B}^{\tau^+}} B^+_{\mathfrak{b}} \times \prod_{\mathfrak{b}\in \mathfrak{B}^{\tau^-}} B^-_{\mathfrak{b}},
\end{equation*}
where $Q^\pm_{\mathfrak{b}} \subset \mathbb{R}$ and $B^\pm_{\mathfrak{b}} \subset \mathbb{R}^{2d}$ are bounded sets such that $\ell(Q^\pm_{\mathfrak{b}}) \lesssim 1$ and $B^\pm_{\mathfrak{b}}$ is bounded by $nR$.
For all $\mathfrak{b} \in \mathfrak{B}^{\cp} $, define
\begin{equation*}
    M^{\pm}_{\mathfrak{b}}
    = \{(\bm{x}^\pm_{\mathfrak{b}},\bm{y}^\pm_{\mathfrak{b}}) \in B^\pm_{\mathfrak{b}} \cap (\mathbb{Z}^d_L)^2: \gamma (\bm{x}^\pm_{\mathfrak{b}} \cdot \bm{y}^\pm_{\mathfrak{b}}) \in Q^\pm_{\mathfrak{b}}\}.
\end{equation*}
Let $G^\pm$ be the set of all $\bm{k} \in \mathscr{D}^\cp_k$ such that $\Lambda^{\cp,\pm} \bm{k} \in M^{\pm}_{\mathfrak{b}}$ for all $\mathfrak{b} \in \mathfrak{B}^{\tau^\pm}$.
By~\eqref{eq:Omega-change-var}, we have
\begin{equation*}
    G^{\cp,L}_k \subset G^+ \cap G^-.
\end{equation*}

\emph{Step 3: Regular mini couple.}
Fix any $\mathfrak{b} \in \mathfrak{B}^{\tau^\pm}$ and let $\varphi = \varphi(\omega,z) \in C_c^\infty(\mathbb{R}\times \mathbb{R}^{2d})$ be such that $\varphi \ge 0$.
For all $u \in \mathbb{R}$, let $\varphi^n_u(\omega,z) = \varphi(\omega-u,z/(nR))$ and suppose that $\varphi^n_u|_{Q^\pm_{\mathfrak{b}} \times B^\pm_{\mathfrak{b}}} \equiv 1$ for some $u \in \mathbb{R}$.
Let $\widehat{\varphi}^n_u(s,z) = \mathcal{F}_{\omega \to s} \varphi_u(\omega,z)$.
By Proposition~\ref{prop:gauss-sum-upper-bound}, for all $N > 0$ and sufficiently large $L$:
\begin{equation}
    \label{eq:mini-cp-card-est}
    \begin{aligned}
        \frac{\gamma}{L^{2d}} &\card(M^{\pm}_{\mathfrak{b}})
        \le \frac{\gamma}{L^{2d}} \sum_{z = (x,y) \in (\mathbb{Z}^{d}_L)^2} \varphi^n_u(\gamma x \cdot y, z)
        = \mathcal{S}_L(\theta_L \widehat{\varphi}^n_u)\\
        & \lesssim (nR)^{2d} (\|\widehat{\varphi}^n_u\|_{L^\infty W^{2d,1}} + L^\alpha \|\widehat{\varphi}^n_u\|_{L^1_{|\bm{s}|\ge L^\delta L^\infty}})
        \lesssim (nR)^{2d} ((nR)^{2d} + L^{\alpha-N} )
        \lesssim (nR)^{4d}.
    \end{aligned}
\end{equation}
Consequently, this also concludes the proof of~\eqref{eq:deco-count-cp} when $n = 1$:
\begin{equation}
    \label{eq:deco-count-mini-cp}
    \card G^{\cp,L}_k(B,Q) \le \card(G^+) \le \card(M^{\pm}_{\mathfrak{b}}) \lesssim (nR)^{4d} L^{2d-\alpha}.
\end{equation}

\emph{Step 4: General irreducible couple.}
Now assume that $n \ge 2$.
Fix $\mathfrak{b} \in \mathfrak{B}^{\tau^-}$.
By Lemma~\ref{lem:tree-deco-change-var}, the vectors $\bm{x}^-_{\mathfrak{b}}$ and $\bm{y}^-_{\mathfrak{b}}$ are linear combinations of the vectors $\{\bm{x}^+_{\mathfrak{n}},\bm{y}^+_{\mathfrak{n}} : \mathfrak{n} \in \mathfrak{B}^{\tau^+}\}$ with integer coefficients.
We claim that there exists $\mathfrak{n} \in \mathfrak{B}^{\tau^+}$ such that
\begin{equation}
    \label{eq:omega-l-r-linear-decomp}
    \bm{x}^-_{\mathfrak{b}} \cdot \bm{y}^-_{\mathfrak{b}}
    = a |\bm{x}^+_{\mathfrak{n}}|^2 + b |\bm{y}^+_{\mathfrak{n}}|^2 + c \bm{x}^+_{\mathfrak{n}} \cdot \bm{y}^+_{\mathfrak{n}} + p(\bm{z}') \cdot \bm{x}^+_{\mathfrak{n}} + q(\bm{z}') \cdot \bm{y}^+_{\mathfrak{n}} + u(\bm{z}') \cdot v(\bm{z}'),
\end{equation}
where $a,b,c \in \mathbb{Z}$ and $p,q,u,v$ are linear combinations with integer coefficients of the components of the vector $\bm{z}' = (\bm{x}^+_{\mathfrak{m}},\bm{y}^+_{\mathfrak{m}})_{\mathfrak{m} \ne \mathfrak{n}}$, and moreover we have either $(a,b) \ne (0,0)$ or $(p,q) \ne (0,0)$.
In fact, if this not true, then for all $\mathfrak{n} \in \mathfrak{B}^{\tau^+}$, quadratic terms such as $|\bm{x}^+_{\mathfrak{n}}|^2$, $|\bm{y}^+_{\mathfrak{n}}|^2$, $\bm{x}^+_{\mathfrak{n}} \cdot \bm{x}^+_{\mathfrak{m}}$, $\bm{y}^+_{\mathfrak{n}} \cdot \bm{y}^+_{\mathfrak{m}}$, $\bm{x}^+_{\mathfrak{n}} \cdot \bm{y}^+_{\mathfrak{m}}$, $\bm{y}^+_{\mathfrak{n}} \cdot \bm{x}^+_{\mathfrak{m}}$ where $\mathfrak{m} \ne \mathfrak{n}$ do not appear in $\bm{x}^-_{\mathfrak{b}} \cdot \bm{y}^-_{\mathfrak{b}}$, which leaves $\bm{x}^-_{\mathfrak{b}} \cdot \bm{y}^-_{\mathfrak{b}}$ as a linear combination of $\{\bm{x}^+_{\mathfrak{n}} \cdot \bm{y}^+_{\mathfrak{n}} : \mathfrak{n} \in \mathfrak{B}^{\tau^+}\}$ and contradicts Lemma~\ref{lem:reso-form-linear-indep}.
We shall now fix such $\mathfrak{n}$ as well.

Now let us omit the $+$ sign for simplicity and thus write $\bm{x}^+_{\mathfrak{m}} = \bm{x}_{\mathfrak{m}}$, $\bm{y}^+_{\mathfrak{m}} = \bm{y}_{\mathfrak{m}}$ for all $\mathfrak{m} \in \mathfrak{B}^{\tau^+}$.
Let $\delta \in (0,1/2)$, which will be determined later.
Write $p = (p_j)_{1\le j \le d}$, $q = (q_j)_{1 \le j \le d}$, $u = (u_j)_{1\le j \le d}$, and $v = (v_j)_{1\le j \le d}$, which are given by~\eqref{eq:omega-l-r-linear-decomp}.
Write $\bm{z}_{\mathfrak{n}} = (\bm{x}_{\mathfrak{n}},\bm{y}_{\mathfrak{n}})$. For all $j \in \{1,\ldots,d\}$, write $\bm{z}_{\mathfrak{n}}^j = (\bm{x}_{\mathfrak{n}}^j,\bm{y}_{\mathfrak{n}}^j)$.
Write the right hand side of~\eqref{eq:omega-l-r-linear-decomp} as
\begin{equation*}
    a |\bm{x}_{\mathfrak{n}}|^2 + b |\bm{y}_{\mathfrak{n}}|^2 + c \bm{x}_{\mathfrak{n}} \cdot \bm{y}_{\mathfrak{n}} + p(\bm{z}') \cdot \bm{x}_{\mathfrak{n}} + q(\bm{z}') \cdot \bm{y}_{\mathfrak{n}} + u(\bm{z}') \cdot v(\bm{z}')
    = \sum_{1 \le j \le d} f^j(\bm{z}',\bm{z}^j),
\end{equation*}
where the polynomials $f^j$ are given by
\begin{equation*}
    f^j(\bm{z}',\bm{z}^j)
    = a (\bm{x}_n^j)^2 + b(\bm{y}_n^j)^2 + c \bm{x}_n^j \bm{y}_n^j + p_j(\bm{z}') \bm{x}_n^j + q_j(\bm{z}') \bm{y}_n^j + u_j(\bm{z}') v_j(\bm{z}').
\end{equation*}
Let $g^j = \nabla_{\bm{z}_{\mathfrak{n}}^j} f^j \times \nabla_{\bm{z}_{\mathfrak{n}}^j} (\bm{x}_{\mathfrak{n}}^j \bm{y}_{\mathfrak{n}}^j)$, so that $g^j(\bm{z}',\cdot)$ is, for every fixed $\bm{z}'$, the Jacobian determinant of the map $(\bm{x}_{\mathfrak{n}}^j,\bm{y}_{\mathfrak{n}}^j) \mapsto (f^j(\bm{z}',\bm{z}^j),\bm{x}_{\mathfrak{n}}^j \bm{y}_{\mathfrak{n}}^j)$.
Then decompose
\begin{equation}
    \label{eq:G-card-decomp}
    G \coloneq \Lambda^{q,+} G^{\cp,L}_k(B,Q)
    = \Bigl(\bigcup_{1 \le j \le d} X^j \Bigr) \cup \Bigl( \bigcap_{1 \le j \le d} N^j \Bigr),
\end{equation}
where $X^j, N^j$ are subsets of $G$ defined by
\begin{equation*}
    X^j = \{\bm{z} \in G: |g^j(\bm{z}',\bm{z}_{\mathfrak{n}})| \ge L^{-2\delta}\}, \quad
    N^j = \{\bm{z} \in G: |g^j(\bm{z}',\bm{z}_{\mathfrak{n}})| \le L^{-2\delta}\}.
\end{equation*}

To estimate $\card(\cup_{j=1}^d X^j) \le \sum_{j=1}^d \card(X^j)$, it suffices to separately estimate $\card(X^j)$ for all $j$.
For any fixed $j$, we first count $\bm{z}'=(\bm{x}_{\mathfrak{m}},\bm{y}_{\mathfrak{m}})_{\mathfrak{m} \ne \mathfrak{n}}$ using~\eqref{eq:mini-cp-card-est}.
Next we count $\bm{z}_{\mathfrak{n}}$ by freely assigning values to $(\bm{z}_{\mathfrak{n}}^{j'})_{j'\ne j}$ and applying Lemma~\ref{lem:lattice-count-coarea} to count $\bm{z}_{\mathfrak{n}}^{j}$.
This yields the estimate
\begin{equation*}
    \card(X^j)
    \lesssim
    C^n (nR)^{4nd} L^{(n-1)(2d-\alpha)} L^{2d+2\delta-2\min\{\alpha,1\}}.
\end{equation*}

To estimate $\card(\cap_{j=1}^d N^j)$, we shall consider separately the non-degenerate case where $(a,b) \ne (0,0)$ and the degenerate case where $(a,b) = (0,0)$ (so that $(p,q) \ne (0,0)$).
In the degenerate case, we simply use~\eqref{eq:mini-cp-card-est} to count $(\bm{z}_{\mathfrak{m}})_{\mathfrak{m} \ne \mathfrak{n}}$ and use Lemma~\ref{lem:count-poly-quasi-zero} to count $\bm{z}_{\mathfrak{n}}^j$ for all $j$, yielding the estimate
\begin{equation*}
    \card\Bigl(\bigcap_{1 \le j \le d} N^j\Bigr)
    \lesssim C^n (nR)^{4nd} L^{(n-1)(2d-\alpha)} L^{(2-\delta)d}
    = C^n (nR)^{4nd} L^{n(2d-\alpha)} L^{\alpha-d\delta}.
\end{equation*}
In the degenerate case, we further decompose $N^j = D^j \cup (N^j \backslash D^j)$, where $D^j$ the set of all $\bm{z} \in N^j$ such that $|p_j(\bm{z}')| + |q_j(\bm{z}')| \le L^{-\delta}$.
Then write
\begin{equation*}
    \bigcap_{1 \le j \le d} N^j
    = \bigcup_{J \in 2^{\{1,\ldots,d\}}} Z^J, \quad
    \text{where}\ 
    Z^J = \Bigl(\bigcap_{j \in J} D^j\Bigr) \cap \Bigl(\bigcap_{j \notin J} N^j \backslash D^j\Bigr).
\end{equation*}
Fix any $J \in 2^{\{1,\ldots,d\}}$ and let $\ell = \card(J)$.
To estimate $\card(Z^J)$, we first count $(\bm{z}_{\mathfrak{b}})_{\mathfrak{b} \in \mathfrak{B}^{\tau^+} \backslash \{\mathfrak{m},\mathfrak{n}\}}$ using~\eqref{eq:mini-cp-card-est}, then count $(\bm{z}_{\mathfrak{b}})_{\mathfrak{b} \in \{\mathfrak{m},\mathfrak{n}\}}$ by writing $(\bm{x}_{\mathfrak{b}},\bm{y}_{\mathfrak{\mathfrak{b}}}) = (\bm{x}'_{\mathfrak{b}},\bm{x}''_{\mathfrak{b}},\bm{y}'_{\mathfrak{b}},\bm{y}''_{\mathfrak{b}})$, where $\bm{x}'_{\mathfrak{b}}$ and $\bm{y}'_{\mathfrak{b}}$ are respectively the coordinates of $\bm{x}_{\mathfrak{b}}$ and $\bm{x}_{\mathfrak{b}}$ in the directions of $J$, while $\bm{x}''_{\mathfrak{b}}$ and $\bm{y}''_{\mathfrak{b}}$ are respectively the coordinates of $\bm{x}_{\mathfrak{b}}$ and $\bm{x}_{\mathfrak{b}}$ in the remaining directions.
To count $(\bm{x}_{\mathbf{m}},\bm{y}_{\mathbf{m}})$, we first freely choose the $2(d-\ell)$ coordinates of $(\bm{x}''_{\mathbf{m}},\bm{y}''_{\mathbf{m}})$, then count $(\bm{x}'_{\mathbf{m}},\bm{y}'_{\mathbf{m}})$ by observing that if $j \in J$, then the point $(\bm{x}_{\mathfrak{m}}^j,\bm{y}_{\mathfrak{m}}^j)$ lies within a rectangle of scale $L^{-\delta} \times nR$.
Therefore, together with~\eqref{eq:mini-cp-card-est}, we see that the number of possible choices for $(\bm{x}_{\mathbf{m}},\bm{y}_{\mathbf{m}})$ is bounded by
\begin{equation*}
    (nR)^{2d}\min\{L^{2(d-\ell)} (L^{2-\delta})^\ell,L^{2d-\alpha}\} \lesssim (nR)^{2d} L^{\min\{2d-\ell\delta,2d-\alpha\}}.
\end{equation*}
To count $(\bm{x}_{\mathfrak{n}},\bm{y}_{\mathfrak{n}})$, we first freely choose the $2\ell$ coordinates of $(\bm{x}'_{\mathfrak{n}},\bm{y}'_{\mathfrak{n}})$ and then use Lemma~\ref{lem:count-poly-quasi-zero} to count $(\bm{x}''_{\mathfrak{n}},\bm{y}''_{\mathfrak{n}})$.
Together with~\eqref{eq:mini-cp-card-est}, this shows that the number of possible choices is bounded by
\begin{equation*}
    (nR)^{2d} \min\{L^{2\ell} (L^{2-\delta})^{d-\ell},L^{2d-\alpha}\}
    \lesssim (nR)^{2d} L^{\min\{2d-(d-\ell)\delta,2d-\alpha\}}.
\end{equation*}
Denote $\mu(\ell) = \max\{\ell \delta-\alpha,0\} + \max\{(d-\ell) \delta - \alpha,0\}$ for $\ell \in \{0,1,\ldots, d\}$.
Then, we have
\begin{equation*}
    \begin{aligned}
        \card(Z^J)
        & \lesssim C^n (nR)^{4nd} L^{(n-2)(2d-\alpha)} L^{\min\{2d-\ell\delta,2d-\alpha\}} L^{\min\{2d-(d-\ell)\delta,2d-\alpha\}} \\
        & \lesssim C^n (nR)^{4nd} L^{n(2d-\alpha)}  L^{-\mu(\ell)}
        \lesssim C^n (nR)^{4nd} L^{n(2d-\alpha)}  L^{-\min \mu}.
    \end{aligned}
\end{equation*}
This yields, in the degenerate case, the estimate
\begin{equation*}
    \card\Bigl(\bigcap_{1 \le j \le d} N^j\Bigr)
    \le \sum_{J \in 2^{\{1,\ldots,d\}}} \card(Z^J)
    \lesssim C^n (nR)^{4nd} L^{n(2d-\alpha)}  L^{-\min\mu}.
\end{equation*}

Combining the estimates above, for all $d \ge 3$ and $\delta \in (0,1/2)$, we have
\begin{equation*}
    \card(G)
    \lesssim
    C^n (nR)^{4nd} L^{n(2d-\alpha)} (L^{\alpha+2\delta-2\min\{\alpha,1\}} + L^{\alpha - d\delta} + L^{-\min\mu})
\end{equation*}
In order to obtain the smallness of the estimate, we need to prove the existence of $\delta$ such that $\alpha+2\delta-2\min\{\alpha,1\} < 0$, $\alpha-d\delta<0$, and $-\min\mu < 0$.
Our idea is to fix $d$ and $\alpha$, and obtain the range of $\delta$ by solving these three inequalities.
Note that, in order for such a $\delta$ to exist, this range must have a nonempty intersection with the interval $(0,1/2)$.

First, the inequality $\alpha - d\delta < 0$ implies $\delta > \alpha / d$.
Next, we separately treat the cases where $\alpha \in (0,1)$ and $\alpha \in [1,2)$.
If $\alpha \in (0,1)$, then the inequality $\alpha+2\delta-2\min\{\alpha,1\} < 0$ gives $\delta > \alpha/2$.
Let us further divide our discussion into the subcases where $\delta \ge \alpha$ and $\alpha/2 < \delta \le \alpha$
\begin{enumerate}
    \item Assume that $\delta \ge \alpha$. Then $\min \mu = d\delta - 2\alpha$. The inequality $-\min \mu < 0$ gives $\delta > 2\alpha/d$. Therefore, the range of $\delta$ given by all three inequalities is $\delta \ge \alpha$.
    In order for $[\alpha,+\infty) \cap (0,1/2) = \emptyset$, one needs $\alpha \in (0,1/2)$.
    \item Assume that $\alpha /2 < \delta \le \alpha$. Then $\min \mu = 2\delta - \alpha$ if $d = 3$ and $\min \mu = d\delta - 2\alpha$ if $d \ge 4$.
    \begin{enumerate}
        \item If $d = 3$, then $ -\min \mu < 0$ gives $\delta > \alpha/2$. The range given by all three inequalities is $\alpha / 2 < \delta \le \alpha$.
        In order for $[\alpha/2,\alpha) \cap (0,1/2) \ne \emptyset$, one needs $\alpha/2 < 1/2$, i.e, $\alpha \in (0,1)$.
        \item If $d \ge 4$, then $-\min \mu < 0$ gives $\delta > 2\alpha/d$. The range given by all three inequalities is still $\alpha/2 < \delta \le \alpha$.
        Therefore, one still needs $\alpha \in (0,1)$ for a nonempty intersection.
    \end{enumerate}
\end{enumerate}
If $\alpha \in [1,2)$, then $\alpha+2\delta-2\min\{\alpha,1\} < 0$ implies $\delta < 1-\alpha/2$.
Note that, by the triangular inequality
\begin{equation*}
    \mu(\ell) 
    = d\delta/2 - \alpha + (|\ell\delta-\alpha| + |(d-\ell)\delta-\alpha|)/2 \ge d\delta/2 - \alpha + |d-2\ell|\delta/2
    \ge d^*\delta/2 -\alpha.
\end{equation*}
The inequality $-\min \mu \le \alpha - d^*\delta/2 < 0$ then implies $\delta > 2\alpha/d^*$.
The range given by all three inequalities is then $2\alpha/d^*<\delta<1-\alpha/2$.
To ensure a nonempty intersection, we need $2\alpha/d^* <\delta<\min\{1-\alpha/2,1/2\} $, which gives $1 \le \alpha < \min \{2/(1+4/d^*),d^*/4\}$.
This range for $\alpha$ is nonempty only when $d \ge 5$, in which case we always have $2/(1+4/d^*) < d^*/4$.

Observing that $d^* = 4$ (and thus $2/(1+4/d^*)=1)$ when $d \in \{3,4\}$, the above analysis can be summarized as follows: if $d \ge 3$ and $0 < \alpha < 2/(1+4/d^*)$, then there exists $\delta \in (0,1/2)$ such that $\alpha+2\delta-2\min\{\alpha,1\} < 0$, $\alpha-d\delta<0$, and $-\min\mu < 0$.
This concludes the proof of Proposition~\ref{prop:deco-count-cp}.

\section{Diagrammatic expansions}

\label{sec:diagram}

The method of diagrammatic expansion was first used by Dyson~\cite{Dyson1949} to study asymptotics of scattering operators in the context of QFT and bares the name of ``perturbation method''.
In this section, we give diagrammatic expansions to solutions of WNLS and their energy spectra.

\subsection{Time ordering}

\label{sec:time-ord}

The diagrammatic expansions of solution to WNLS is obtained by iteratively applying Duhamel's principle, and relies on tree structures to encode this process so that ordering of tree nodes give that of the time variables in the iterated Duhamel integrations.
The purpose of this section is to define time orderings on trees and unions of trees and obtain a decay estimate of time ordered sets (Proposition~\ref{prop:time-ord-fourier-est}).
We have seen this decay estimate in~\cite[Proposition~2.3]{Deng2021b} and~\cite[Lemma~10.2]{Deng2021}.
However, we shall give a completely new, extremely short, and self-contained proof which exploits the time ordering structures.
Even though our result hold true for both positive and negative times, we shall only write the proof when the time is positive as the situation is similar in the negative case.

\begin{definition}
    \label{def:graph}
    Let $\mathscr{G}$ be the set of all finite graphs formed by disjoint unions of trees. If $G \in \mathscr{G}$, then $G$ canonically inherits a strict partial ordering $\prec$ from the trees that form $G$.
    We denote respectively by $\mathfrak{R}(G)$, $\mathfrak{L}(G)$ and $\mathfrak{B}(G)$ the sets of all roots, all leaves and all branching nodes in $G$.
    If $\mathfrak{l} \in \mathfrak{L}(G)$, then we define the subgraph $G^{\backslash\mathfrak{l}} \coloneqq G \backslash \{\mathfrak{l}\}$ and the quotient graph $G^{\sim\mathfrak{l}} \coloneqq G / \{\mathfrak{l} \sim \mathfrak{l}^p\}$.
    Clearly they both belong to $\mathscr{G}$ by canonically inheriting the strict partial ordering $\prec$ from $G$.
\end{definition}

\begin{definition}
    \label{def:time-ord}
    If $G \in \mathscr{G}$, then $\mathscr{O}(G)$ is the set of all $\bm{t} \in \mathbb{R}^G$ such that $\mathfrak{n} \prec \mathfrak{m}$ implies $ \bm{t}_{\mathfrak{n}} < \bm{t}_{\mathfrak{m}}$ for all $\mathfrak{n},\mathfrak{m} \in G$.
    For all $t > 0$, we let $I_t = (0,t)$ and denote $\mathscr{O}_t(G) = \mathscr{O}(G) \cap I_t^{G}$.
    We also denote by $\Theta_t[G]$ the inverse Fourier transform of $\bm{1}_{\mathscr{O}_t(G)}$.\index{Operations and transforms!Phase space transforms!Fourier transforms of Graphs $\Theta_t[G]$}

    If $\tau \in \mathscr{T}$ and $\cp = (\tau^*,\wp) \in \mathscr{K}$, then we denote $\mathscr{O}^\tau = \mathscr{O}(\mathfrak{B}^\tau)$, $\mathscr{O}^\cp = \mathscr{O}(\mathfrak{B}^\cp) = \mathscr{O}(\mathfrak{B}^{\tau_+}) \times \mathscr{O}(\mathfrak{B}^{\tau_-})$, and $\mathscr{O}^\tau_t = \mathscr{O}_t(\mathfrak{B}^\tau)$, $\mathscr{O}^\cp_t = \mathscr{O}_t(\mathfrak{B}^\cp) = \mathscr{O}_t(\mathfrak{B}^{\tau_+}) \times \mathscr{O}_t(\mathfrak{B}^{\tau_-})$.
    We also denote $\Theta^\tau_t = \Theta_t[\mathfrak{B}^\tau]$ and $\Theta^\cp_t = \Theta_t[\mathfrak{B}^\cp]$.
    Clearly $ \Theta^\cp_t = \Theta_t[\mathfrak{B}^{\tau^+}] \otimes \Theta_t[\mathfrak{B}^{\tau_-}] = \Theta^{\tau^+}_t \otimes \Theta^{\tau^-}_t $.
    We also use the convention that $\Theta^\tau_t = 1$ and $\Theta^\cp_t = 1$ when $\tau \in \mathscr{T}_0$ and $\cp \in \mathscr{K}_0$.\index{Operations and transforms!Phase space transforms!Fourier transforms of trees and couples $\Theta^\tau_t$, $\Theta^\cp_t$}
\end{definition}

\begin{lemma}
    \label{lem:time-ord-ind-diff}
    Let $G \in \mathscr{G}$ and $\mathfrak{l} \in \mathfrak{L}(G) \backslash \mathfrak{R}(G)$. Then
    \begin{equation}
        \label{eq:time-ord-ind-diff-id}
        \partial_{\bm{t}_{\mathfrak{l}}} \bm{1}_{\mathscr{O}(G)} = - \varrho_{\mathfrak{l}}^* \bm{1}_{\mathscr{O}(G^{\sim\mathfrak{l}})},
    \end{equation}
    where $\varrho_{\mathfrak{l}}^*$ is the pullback of the map $\varrho_{\mathfrak{l}} : \mathbb{R}^G \to \mathbb{R}^{G^{\sim\mathfrak{l}}}$ defined as the composition of the restriction $\mathbb{R}^G \to \Gamma^G_{\mathfrak{l}} \coloneqq \{\bm{t} \in \mathbb{R}^G : \bm{t}_{\mathfrak{l}} = \bm{t}_{\mathfrak{l}^p}\}$ and the identification $\Gamma^G_{\mathfrak{l}}  \to \mathbb{R}^{G^{\sim\mathfrak{l}}}$ canonically induced from the quotient map $G \to G^{\sim\mathfrak{l}}$.
    Consequently, if $t > 0$, then
    \begin{equation}
        \label{eq:time-ord-ind-diff-id-boundary}
        \partial_{\bm{t}_{\mathfrak{l}}} \bm{1}_{\mathscr{O}_t(G)}
        = - \varrho_{\mathfrak{l}}^* \bm{1}_{\mathscr{O}_t(G^{\sim\mathfrak{l}})}
        + \bm{1}_{\mathscr{O}_t(G^{\backslash\mathfrak{l}})} \otimes \bm{\delta}^{\mathfrak{l}}_{0},
    \end{equation}
    where $\bm{\delta}^{\mathfrak{l}}$ denotes the Dirac math for the variable $\bm{t}_{\mathfrak{l}}$.
\end{lemma}
\begin{proof}
    For all $\bm{t} \in \mathbb{R}^G$ denote by $\bm{t}^{\backslash\mathfrak{l}}$ its projection in $\mathbb{R}^{G^{\backslash\mathfrak{l}}}$, then $\bm{1}_{\mathscr{O}(G)}(\bm{t}) = \bm{1}_{\bm{t}_{\mathfrak{l}}<\bm{t}_{\mathfrak{l}^p}} \bm{1}_{\mathscr{O}(G^{\backslash\mathfrak{l}})}(\bm{t}^{\backslash\mathfrak{l}})$.
    Therefore $\partial_{\bm{t}_{\mathfrak{l}}} \bm{1}_{\mathscr{O}(G)}(\bm{t}) = -\bm{\delta}(\bm{t}_{\mathfrak{l}}-\bm{t}_{\mathfrak{l}^p}) \bm{1}_{\mathscr{O}(G^{\backslash\mathfrak{l}})}(\bm{t}^{\backslash\mathfrak{l}})$ which is exactly~\eqref{eq:time-ord-ind-diff-id}.
    To prove~\eqref{eq:time-ord-ind-diff-id-boundary}, we denote $\bm{1}^{\otimes G}_{I_t}(\bm{t}) = \prod_{\mathfrak{n} \in G} \bm{1}_{I_t}(\bm{t}_{\mathfrak{n}})$ and write
    \begin{equation*}
        \partial_{\bm{t}_{\mathfrak{l}}} \bm{1}_{\mathscr{O}_t(G)}
        = \partial_{\bm{t}_{\mathfrak{l}}} \bigl( \bm{1}_{\mathscr{O}(G)} \bm{1}^{\otimes G}_{I_t}\bigr)
        =  \bm{1}^{\otimes G}_{I_t} \partial_{\bm{t}_{\mathfrak{l}}} \bm{1}_{\mathscr{O}(G)}  + \bm{1}_{\mathscr{O}(G)} \partial_{\bm{t}_{\mathfrak{l}}} \bm{1}^{\otimes G}_{I_t}.
    \end{equation*}
    Using~\eqref{eq:time-ord-ind-diff-id}, we calculate the first term as follows:
    \begin{equation*}
        \bm{1}^{\otimes G}_{I_t} \partial_{\bm{t}_{\mathfrak{l}}} \bm{1}_{\mathscr{O}(G)}
        = - \bm{1}^{\otimes G}_{I_t} \varrho_{\mathfrak{l}}^* \bm{1}_{\mathscr{O}(G^{\sim\mathfrak{l}})} 
        = -\varrho_{\mathfrak{l}}^* \bm{1}_{\mathscr{O}_t(G^{\sim\mathfrak{l}})}.
    \end{equation*}
    As for the second term, note that $\partial_{\bm{t}_{\mathfrak{l}}} \bm{1}^{\otimes G}_{I_t} = - \bm{1}^{\otimes G^{\backslash\mathfrak{l}}}_{I_t} \otimes (\bm{\delta}^{\mathfrak{l}}_t - \bm{\delta}^{\mathfrak{l}}_0)$.
    Also note that since $\mathfrak{l} \not \in \mathfrak{R}(G)$, we have $\supp \bm{1}_{\mathscr{O}(G)} (\bm{1}^{G^{\backslash\mathfrak{l}}}_{I_t} \otimes \bm{\delta}^{\mathfrak{l}}_t) = \emptyset$.
    Therefore,
    \begin{align*}
        \bm{1}_{\mathscr{O}(G)} \partial_{\bm{t}_{\mathfrak{l}}} \bm{1}^{\otimes G}_{I_t}
        & = \bm{1}_{\mathscr{O}(G)} \bigl(\bm{1}^{\otimes G^{\backslash\mathfrak{l}}}_{I_t} \otimes \bm{\delta}^{\mathfrak{l}}_0\bigr)
        = \bm{1}_{\mathscr{O}_t(G^{\backslash\mathfrak{l}})}\otimes \bm{\delta}^{\mathfrak{l}}_{0}.
        \qedhere
    \end{align*}
\end{proof}

\begin{proposition}
    \label{prop:time-ord-fourier-est}
    Let $G \in \mathscr{G}$ with $\card(G) = n \ge 1$ and denote $\mathscr{Q}(G) = \{0,1\}^{G \backslash \mathfrak{R}(G)} $.
    Every $ \mu \in \mathscr{Q}(G) $ defines a linear map on $\mathbb{R}^G$ by setting $\mu(\bm{\omega})_{\mathfrak{g}} = \bm{\omega}_{\mathfrak{g}} + \sum_{\mathfrak{n} \prec \mathfrak{g}} \mu_{\mathfrak{n}} \bm{\omega}_{\mathfrak{n}}$ for all $\mathfrak{g} \in G$.
    Then for some universal constant $C >0$, for all $t > 0$ and all $\bm{\omega} \in \mathbb{R}^G$, we have 
    \begin{equation}
        \label{eq:time-ord-fourier-est}
        |\Theta_t[G](\bm{\omega})| \lesssim C^n t^n \sup_{\mu \in \mathscr{Q}(G)} \prod_{\mathfrak{g} \in G} \frac{1}{\langle t \mu(\bm{\omega})_{\mathfrak{g}}  \rangle}.
    \end{equation}
\end{proposition}
\begin{proof}
    Clearly it suffices to prove~\eqref{eq:time-ord-fourier-est} when $\card \mathfrak{R}(G) = 1$, that is, when $G$ is a tree.
    We will prove~\eqref{eq:time-ord-fourier-est} using mathematical induction on $\card(G)$.
    Observe that $\Theta_t[G](\bm{\omega}) = t^n \Theta_1[G](\bm{\omega})$, therefore it suffices to prove~\eqref{eq:time-ord-fourier-est}, when $t = 1$.

    If $\card(G) = 1$, then $\bm{1}_{\mathscr{O}_1(G)} = \bm{1}_{I_1}$ and $\Theta_1[G](\bm{\omega}) = \frac{e^{2\pi i \bm{\omega}}-1}{2\pi i \bm{\omega}}$, which yields~\eqref{eq:time-ord-fourier-est}.
    Now assume that $\card(G) \ge 2$.
    Let $\mathfrak{r}$ be the root of $G$.
    For all $\mathfrak{g} \in G$, let $G_{\mathfrak{g}} = \{\mathfrak{n} \in G: \mathfrak{n} \preceq \mathfrak{g}\}$ be the subtree rooted at $\mathfrak{g}$ and let $m_{\mathfrak{g}} = \card(G_{\mathfrak{g}})$.
    Note that $n - 1 = \sum_{\mathfrak{g}^p = \mathfrak{r}} m_{\mathfrak{g}}$ and we have
    \begin{equation}
        \label{eq:time-ord-fourier-induction}
        \Theta_t[G](\bm{\omega}) = \int_0^t e^{2\pi i s \bm{\omega}_{\mathfrak{r}}} \prod_{\mathfrak{g}^p = \mathfrak{r}} \Theta_s[G_{\mathfrak{g}}](\bm{\omega}^{\mathfrak{g}}) \diff s,
    \end{equation}
    where $\bm{\omega}^{\mathfrak{g}}$ is the projection of $\bm{\omega}$ in $\mathbb{R}^{G^{\mathfrak{g}}}$.
    By~\eqref{eq:time-ord-fourier-induction} and the induction hypothesis, we have
    \begin{align*}
        |\Theta_1[G](\bm{\omega})|
        & \le \int_0^1 \prod_{\mathfrak{g}^p = \mathfrak{r}}  |\Theta_s[G_{\mathfrak{g}}](\bm{\omega}^{\mathfrak{g}})| \diff s
        \lesssim \int_0^1 \prod_{\mathfrak{g}^p = \mathfrak{r}}  (Cs)^{m_{\mathfrak{g}}} \sup_{\mu^{\mathfrak{g}} \in \mathscr{Q}(G_\mathfrak{g})} \prod_{\mathfrak{n} \in G_{\mathfrak{g}}} \frac{1}{\langle s \mu^\mathfrak{g}(\bm{\omega}^{\mathfrak{g}})_{\mathfrak{n}}\rangle} \diff s \\
        & \lesssim C^{n-1} \sup_{\mu \in \mathscr{Q}(G)} \sup_{s \in [0,1]} \prod_{\mathfrak{g} \in G \backslash \{\mathfrak{r}\}} \frac{s}{\langle s \mu(\bm{\omega})_{\mathfrak{g}}\rangle}
        \lesssim C^{n-1} \sup_{\mu \in \mathscr{Q}(G)}  \prod_{\mathfrak{g} \in G\backslash\{\mathfrak{r}\}} \frac{1}{\langle \mu(\bm{\omega})_{\mathfrak{g}} \rangle}. 
    \end{align*}
    Next, let $\varphi_{\bm{\omega}}(\bm{t}) = e^{2\pi i \bm{t} \cdot \bm{\omega}}$ for all $\bm{t},\bm{\omega} \in \mathbb{R}^G$.
    Choose any $\mathfrak{l} \in \mathfrak{L}(G) \backslash \mathfrak{R}(G)$.
    Let $\bm{\omega}^{\backslash\mathfrak{l}}$ be the projection of $\bm{\omega}$ in $\mathbb{R}^{G^{\backslash\mathfrak{l}}}$ and let $\bm{\omega}^{\sim\mathfrak{l}} \in \mathbb{R}^{G^{\backslash\mathfrak{l}}}$ be such that $\bm{\omega}^{\sim \mathfrak{l}}_{\mathfrak{l}^p} = \bm{\omega}_{\mathfrak{l}} + \bm{\omega}_{\mathfrak{l}^p}$ and $\bm{\omega}^{\sim \mathfrak{l}}_{\mathfrak{g}} = \bm{\omega}_{\mathfrak{g}} $ for all $\mathfrak{g} \in G \backslash \{\mathfrak{l},\mathfrak{l}^p\}$.
    As $\bm{1}_{\mathscr{O}_1(G)}$ is compactly supported, we integrate by part using~\eqref{eq:time-ord-ind-diff-id-boundary} to obtain
    \begin{align*}
        | \bm{\omega}_{\mathfrak{l}} &\Theta_1[G](\bm{\omega}) |
        = |\langle \bm{1}_{\mathscr{O}_1(G)}, \bm{\omega}_{\mathfrak{l}}\varphi_{\bm{\omega}}  \rangle_{\mathbb{R}^G}|
        \lesssim |\langle \partial_{\bm{t}_{\mathfrak{l}}}  \bm{1}_{\mathscr{O}_1(G)}, \varphi_{\bm{\omega}} \rangle_{\mathbb{R}^G}| \\
        &\lesssim |\langle \varrho_{\mathfrak{l}}^* \bm{1}_{\mathscr{O}_1(G^{\sim\mathfrak{l}})}, \varphi_{\bm{\omega}} \rangle_{\mathbb{R}^G}| + |\langle \bm{1}_{\mathscr{O}_1(G^{\backslash \mathfrak{l}})} \otimes \bm{\delta}^{\mathfrak{l}}_{0}, \varphi_{\bm{\omega}} \rangle_{\mathbb{R}^G}|
        \lesssim |\Theta[G^{\backslash\mathfrak{l}}]_1(\bm{\omega}^{\sim\mathfrak{l}})| + | \Theta[G^{\backslash\mathfrak{l}}]_1(\bm{\omega}^{\backslash\mathfrak{l}})| \\
        & \lesssim C^{n-1} \sup_{\mu \in \mathscr{Q}(G^{\backslash\mathfrak{l}})} \Bigl( \prod_{\mathfrak{g} \in G^{\backslash\mathfrak{l}}} \frac{1}{\langle \mu(\bm{\omega}^{\sim\mathfrak{l}})_{\mathfrak{g}}\rangle} + \prod_{\mathfrak{g} \in G^{\backslash\mathfrak{l}}} \frac{1}{\langle \mu(\bm{\omega}^{\backslash \mathfrak{l}})_{\mathfrak{g}} \rangle} \Bigr)
        \lesssim C^{n-1} \sup_{\mu \in \mathscr{Q}(G)}\prod_{\mathfrak{g} \in G^{\backslash\mathfrak{l}}} \frac{1}{\langle \mu(\bm{\omega})_{\mathfrak{g}} \rangle} .
    \end{align*}
    We conclude by combining the two estimates:
    \begin{equation*}
        \langle \bm{\omega}_{\mathfrak{l}} \rangle |\Theta_1[G](\bm{\omega})| \lesssim C^{n-1} \sup_{\mu \in \mathscr{Q}(G)} \Bigl( \prod_{\mathfrak{g} \in G\backslash\{\mathfrak{r}\}} \frac{1}{\langle \mu(\bm{\omega})_{\mathfrak{g}} \rangle} + \prod_{\mathfrak{g} \in G^{\backslash\mathfrak{l}}} \frac{1}{\langle \mu(\bm{\omega})_{\mathfrak{g}} \rangle}\Bigr).\qedhere
    \end{equation*}
\end{proof}

\subsection{Cauchy problem}

\label{sec:cauchy}

In this section we prove Theorem~\ref{thm:cauchy} with the exception of the phase randomization condition~\eqref{eq:cdt-phase-randomization-wp}, which will be established in \S\ref{sec:energy-spectrum}.
Due to technical nuances, we will study the periodic case $\epsilon = 0$ and the non-periodic case $\epsilon \ne 0$ separately.

\subsubsection{Periodic setting}
\label{sec:cauchy-periodic}

Recall that WNLS is equivalent to the system WNLS$'$ for $(u^n)_{n \ge 0}$.
Particularly $u^0$ satisfies NLS with null initial data, and we therefore have $u^0 = 0$.
If $n_1+n_2+n_3=n$ and $n \ge 1$, then $\max\{n_1,n_2,n_3\} \le n$ with most one number among $n_1,n_2,n_3$ is equal to $n$.
If $n \ge 1$, then $u^n$ satisfies a linear Schr\"odinger equation.
If in addition $n$ is even, then at least one element among $n_1,n_2,n_3$ is even.
We use induction to deduce that $u^{2n} = 0$ for all $n \in \mathbb{N}$. Next write
\begin{equation}
    \label{eq:cauchy-diagram-hom}
    u^{2n+1}(t,x) = \frac{1}{L^d} \sum_{k \in \mathbb{Z}^d_L} J^n_{t,k} e^{2\pi i (k \cdot x - |k|^2 t/2)}.
    \index{Functions and random variables!Wavepacket functions!Amplitude/Fourier coefficient (periodic setting) $J^n_{t,k}$}
\end{equation}
Then WNLS$'$ is equivalent to the following system for $(J^n_{t,k})_{n \ge 0, k \in \mathbb{Z}^d_L}$:
\begin{equation}
    \label{eq:WNLS-J-hom}
    i \frac{\diff}{\diff t} J^n_{t,k} = -\frac{\lambda}{L^d} \sum_{n_1+n_2+n_3=n-1} \sum_{\bm{k} \in \mathscr{D}^L_k} e^{2\pi i t \Omega(\bm{k})} J^{n_1}_{t,k_1} \odot \overline{J^{n_2}_{t,k_2}} \odot J^{n_3}_{t,k_3}.
\end{equation}
Let $\psi \in \mathscr{S}(\mathbb{R}^d)$ be the trace of $\phi$ at $x = 0$, i.e., we have $\psi(k) = \phi(0,k)$ for all $k \in \mathbb{R}^d$.
Then, using the initial condition
$J^n_{t,k} = \bm{1}_{n=1} \psi(k) \mathfrak{g}_k,$
we give explicit formulas for $J^n_{t,k}$ when $t > 0$ in the following.
The formulas when $t < 0$ can be derived similarly.

For all $\tau \in \mathscr{T}$, set $\Psi^\tau_t = \Theta^\tau_t \comp \Omega^\tau$, and set
$\psi^\tau(\bm{k}) = \prod_{\mathfrak{l} \in \mathfrak{L}^\tau} \psi^{\iota_{\mathfrak{l}}} (\bm{k}_{\mathfrak{l}})$ for all $\bm{k} \in \mathscr{D}^\tau$,
using the convention that $f^{+1} = f$ and $f^{-1} = \overline{f}$.
Next, define the discrete stochastic process $\mathfrak{g}^\tau:\mathscr{D}^{\tau,L} \to H^{\wick{2n+1}}$, where $n$ is the order of $\tau$, by setting
$\mathfrak{g}^\tau_{\bm{k}} = \bigodot_{\mathfrak{l} \in \mathfrak{L}^\tau}
\mathfrak{g}_{\bm{k}_{\mathfrak{l}}}^{\iota_{\mathfrak{l}}}.$
Finally, for all $n \in \mathbb{N}$, all $k \in \mathbb{Z}^d_L$, and all $t > 0$, define
\begin{equation*}
    J^{n,\pm}_{t,k} = \sum_{\tau \in \mathscr{T}^\pm_n} J^\tau_{t,k},
    \quad
    J^\tau_{t,k}
    = \varsigma_\tau \Bigl( \frac{\lambda}{L^d} \Bigr)^{n} \sum_{\bm{k} \in \mathscr{D}^{\tau,L}_k} \Psi^{\tau}_t(\bm{k}) \psi^\tau(\bm{k}) \mathfrak{g}^\tau_{\bm{k}}.
\end{equation*}
Clearly $J^{n,\pm}_{t,k}$ and $J^{\tau}_{t,k}$ are almost surely well-defined since $\psi \in \mathscr{S}(\mathbb{R}^d)$.
Note that $J^{n,\pm}_{t,k} = \overline{J^{n,\mp}_{t,k}}$.
We claim that $J^n_{t,k} = J^{n,+}_{t,k}$, which follows by observing that the $t$-derivative of~\eqref{eq:J-n-hom-induction} in the following lemma yields the equation~\eqref{eq:WNLS-J-hom}.

\begin{lemma}
    \label{lem:J-hom-induction}
    If $\tau \in \mathscr{T} \backslash \mathscr{T}_0$ with $\mathfrak{r}$ being its root, then for all $t > 0$ and all $k \in \mathbb{Z}^d_L$, we have
    \begin{equation}
        \label{eq:J-tree-hom-induction}
        J^\tau_{t,k} = \frac{i \iota_\tau \lambda}{L^d} \sum_{\bm{k} \in \mathscr{D}^L_k} \int_0^t e^{2\pi i s \Omega(\bm{k})} \bigodot_{1\le j \le 3} J^{\tau_{\mathfrak{r}[j]}}_{s,k_j} \diff s,
    \end{equation}
    where $\mathscr{D}^L_k = \mathscr{D}_k \cap \mathbb{Z}^{3d}_L$, the factor $\Omega(\bm{k})$ is defined by~\eqref{eq:def-reso-factor}, and $\tau_{\mathfrak{r}[j]}$ is the subtree of $\tau$ rooted at $\mathfrak{r}[j]$.
    Consequently, for all $n \ge 1$, we have the induction formula
    \begin{equation}
        \label{eq:J-n-hom-induction}
        J^{n,\pm}_{t,k} = \frac{\pm i\lambda}{L^d} 
        \sum_{n_1+n_2+n_3=n-1} \sum_{\bm{k} \in \mathscr{D}^L_k} \int_0^t e^{2\pi i s \Omega(\bm{k})} J^{n_1,\pm}_{s,k_1}  \odot J^{n_2,\mp}_{s,k_2} \odot J^{n_3,\pm}_{s,k_3} \diff s.\qedhere
    \end{equation}
\end{lemma}
\begin{proof}
    We shall only prove~\eqref{eq:J-tree-hom-induction} as~\eqref{eq:J-n-hom-induction} follows directly.
    For $1 \le j \le 3$, write $\tau_j = \tau_{\mathfrak{r}[j]}$ and let $n_j$ be the order of $\tau_j$.
    For all $\overline{\bm{k}} \in \mathscr{D}^\tau$, write $\bm{k} = (\overline{\bm{k}}_{\mathfrak{r}[1]},\overline{\bm{k}}_{\mathfrak{r}[2]},\overline{\bm{k}}_{\mathfrak{r}[3]}) \in \mathbb{R}^{3d}$ and $\bm{k}^j = \overline{\bm{k}}|_{\tau_j}$.
    Note that $\overline{\bm{k}} \in \mathscr{D}^{\tau,L}_k$ if and only if $\bm{k} \in \mathscr{D}^L_k$.
    Therefore~\eqref{eq:J-tree-hom-induction} follows from $n-1 = n_1 + n_2 + n_3$,
    \begin{equation*}
        \varsigma_\tau = i \iota_\tau \prod_{1\le j \le 3} \varsigma_{\tau_j}, \quad
        \psi^\tau(\overline{\bm{k}}) = \prod_{1 \le j \le 3} \psi^{\tau_j}(\bm{k}^j), \quad
        \mathfrak{g}^\tau_{\overline{\bm{k}}} = \bigodot_{1 \le j \le  3} \mathfrak{g}^{\tau_j}_{\bm{k}^j},
    \end{equation*}
    and the following induction formula which is a result of~\eqref{eq:time-ord-fourier-induction}:
    \begin{equation*}
        \Psi^\tau_t(\overline{\bm{k}}) = \int_0^t e^{2\pi i s\Omega(\bm{k}) } \prod_{1 \le j \le 3} \Psi^{\tau_j}_s(\bm{k}^j) \diff s.\qedhere
    \end{equation*}
\end{proof}

\subsubsection{Non-periodic setting}
\label{sec:cauchy-non-periodic}

As shown in \S\ref{sec:cauchy-periodic}, we have $u^{2n} \equiv 0$.
Then we write
\begin{equation}
    \label{eq:u-wp-expansion}
    u^{2n+1}(t,x) = \frac{1}{L^d} \sum_{k \in \mathbb{Z}^d_L} (e^{-\epsilon t k \cdot \nabla + \epsilon^2 i t \Delta / 4\pi} \mathcal{J}^n_{t,k})(\epsilon x) e^{2\pi i (k \cdot x - |k|^2 t/2)}.
    \index{Functions and random variables!Wavepacket functions!Amplitude (non-periodic setting) $\mathcal{J}^n_{t,k}$}
\end{equation}
To solve WNLS, it suffices for the Fourier transforms $(\widehat{\mathcal{J}}^n_{t,k})_{n \ge 0, k \in \mathbb{Z}^d_L}$ to solve system:
\begin{equation}
    \label{eq:WNLS-J-inhom}
    i \partial_t \widehat{\mathcal{J}}^n_{t,k}(\xi) + \frac{\lambda}{L^d} \sum_{n_1+n_2+n_3=n-1} \sum_{\bm{k} \in \mathscr{D}^L_k} \int_{\mathscr{D}_\xi} e^{2\pi i t \Omega(\bm{k}+ \epsilon\bm{\xi})} \widehat{\mathcal{J}}^{n_1}_{t,k_1}(\xi_1) \odot \overline{\widehat{\mathcal{J}}^{n_2}_{t,k_2}(\xi_2)} \odot \widehat{\mathcal{J}}^{n_3}_{t,k_3}(\xi_3) \diff \bm{\xi},
\end{equation}
with initial data $\mathcal{J}^n_{0,k}(x) = \bm{1}_{n=1}\widehat{\phi}(\xi,k)$.
We give explicit formulas for $\mathcal{J}^n_{t,k}$ when $t > 0$ in the following. The formulas when $t < 0$ can be derived similarly.

For all $\tau \in \mathscr{T}$, $\varphi \in \mathscr{S} (\mathscr{D}^\tau)$, and $\xi \in \mathbb{R}^d$, define
\begin{equation}
    \label{eq:tree-int-def}
    \int_{\mathscr{D}^\tau_\xi} \varphi(\bm{\xi}) \diff\bm{\xi}
    = \int_{(\mathbb{R}^d)^{\mathfrak{L}^\tau}} \bm{\delta}_\xi\Bigl(\iota_\tau\sum_{\mathfrak{l} \in \mathfrak{L}^\tau} \iota_{\mathfrak{l}} \bm{f}_{\mathfrak{l}}\Bigr) \varphi(\breve{\bm{f}}) \diff \bm{f},
\end{equation}
where $\breve{\bm{f}}$ is the unique element in $\mathscr{D}^\tau_\xi$ such that $\breve{\bm{f}}|_{\mathfrak{L}^\tau} = \bm{f}$.
Define $\widehat{\phi}^\tau \in \mathscr{S}(\mathscr{D}^\tau \times \mathscr{D}^\tau)$ by setting $\widehat{\phi}^\tau(\bm{k},\bm{\xi}) = \prod_{\mathfrak{l} \in \mathfrak{L}^\tau} \widehat{\phi}^{\iota_{\mathfrak{l}}}(\bm{k}_{\mathfrak{l}},\bm{\xi}_{\mathfrak{l}})$.
For all $n \in \mathbb{N}$, all $t > 0$ and all $k \in \mathbb{Z}^d_L$, define
\begin{equation*}
    \mathcal{J}^{n,\pm}_{t,k} = \sum_{\tau \in \mathscr{T}^\pm_n} \mathcal{J}^{\tau}_{t,k},
    \quad
    \widehat{\mathcal{J}}^{\tau,L}_{t,k}(\xi)=
    \varsigma_\tau \Bigl(\frac{\lambda}{L^d}\Bigr)^{n}
    \sum_{\bm{k} \in \mathscr{D}^{\tau,L}_{k}}
    \Bigl( \int_{\mathscr{D}^\tau_\xi}
    \Psi^\tau_t(\bm{k}+\epsilon\bm{\xi}) \widehat{\phi}^\tau (\bm{k},\bm{\xi}) \diff \bm{\xi} \Bigr)
    \mathfrak{g}^\tau_{\bm{k}},
\end{equation*}
where $\widehat{\mathcal{J}}^{\tau,L}_{t,k}(\xi) = \mathcal{F}_{x\to\eta}\{ \mathcal{J}^{\tau,L}_{t,k}(x) \}$.
Clearly $\mathcal{J}^{n,\pm}_{t,k}(x)$ and $\mathcal{J}^{\tau}_{t,k}(x)$ are almost surely well-defined since $\phi \in \mathscr{S}(\mathbb{R}^d \times \mathbb{R}^d)$.
Note that $\mathcal{J}^{n,\pm}_{t,k} = \overline{\mathcal{J}^{n,\mp}_{t,k}}$.
We claim that $\mathcal{J}^n_{t,k} = \mathcal{J}^{n,+}_{t,k}$, which follows by observing that the $t$-derivative of~\eqref{eq:J-n-inhomo-induction} in the following lemma yields the equation~\eqref{eq:WNLS-J-inhom}.

\begin{lemma}
    \label{lem:lem:J-inhomo-induction}
    If $\tau \in \mathscr{T} \backslash \mathscr{T}_0$ with $\mathfrak{r}$ being its root, then for all $t > 0$ and all $k \in \mathbb{Z}^d_L$, we have
    \begin{equation}
        \label{eq:J-tree-inhomo-induction}
        \widehat{\mathcal{J}}^\tau_{t,k}(\xi) = \frac{i \iota_\tau \lambda}{L^d} \sum_{k \in \mathscr{D}^L_k} \int_{\mathscr{D}_\xi} \diff \bm{\xi} \int_0^t e^{2\pi i s\Omega(\bm{k}+\epsilon\bm{\xi})} \bigodot_{1\le j \le 3} \widehat{\mathcal{J}}^{\tau_{\mathfrak{r}[j]}}_{t,k}(\xi_j) \diff s.
    \end{equation}
    Consequently, for all $n \ge 1$, we have the induction formula
    \begin{equation}
        \label{eq:J-n-inhomo-induction}
        \widehat{\mathcal{J}}^{n,\pm}_{t,k}(\xi) = \frac{\pm i\lambda}{L^d} 
        \sum_{n_1+n_2+n_3=n-1} \sum_{\bm{k} \in \mathscr{D}^L_k} \int_{\mathscr{D}_\xi} \diff \bm{\xi} \int_0^t e^{2\pi i s \Omega(\bm{k}+\epsilon\bm{\xi})} \widehat{\mathcal{J}}^{n_1,\pm}_{s,k_1}(\xi_1)  \odot \widehat{\mathcal{J}}^{n_2,\mp}_{s,k_2}(\xi_2) \odot \widehat{\mathcal{J}}^{n_3,\pm}_{s,k_3}(\xi_3) \diff s.
    \end{equation}
\end{lemma}
\begin{proof}
    The proof is similar to, and will use notations from, that of Lemma~\ref{lem:J-hom-induction}.
    In fact~\eqref{eq:J-tree-inhomo-induction} follows as in Lemma~\ref{lem:J-hom-induction} with an addition ingredient: for all $\overline{\bm{\xi}} \in \mathscr{D}^\tau$, write $ \bm{\xi} = (\xi_1,\xi_2,\xi_3) = (\overline{\bm{\xi}}_{\mathfrak{r}[1]},\overline{\bm{\xi}}_{\mathfrak{r}[2]},\overline{\bm{\xi}}_{\mathfrak{r}[3]})$ and $\bm{\xi}^j = \overline{\bm{\xi}}|_{\tau_j}$, then for all $\varphi \in \mathscr{S}(\mathscr{D}^\tau)$, we have    
    \begin{align*}
        \int_{\mathscr{D}_\xi}
        \Bigl(\prod_{1\le j \le 3} \int_{\mathscr{D}^{\tau_j}_{\xi_j}} \diff \bm{\xi}^j \Bigr) \tilde{\varphi}(\bm{\xi})
        \diff \bm{\xi}
        & = \int_{\mathscr{D}_\xi}
        \Bigl\{ \prod_{1 \le j \le 3} \int \bm{\delta}_{k_j}\Bigl(\iota_{\tau_j}\sum_{\mathfrak{l} \in \mathfrak{L}^{\tau_j}} \iota_{\mathfrak{l}} \bm{f}^j_{\mathfrak{l}}\Bigr) \diff \bm{f}^j \Bigr\} \tilde{\varphi}(\breve{\bm{f}}) \diff \bm{\xi}
        = \int_{\mathscr{D}^\tau_\xi}
        \varphi(\overline{\bm{\xi}}) \diff \overline{\bm{\xi}},
    \end{align*}
    where $\tilde{\varphi} \in \mathscr{S}(\mathscr{D}^{\tau_1}\times\mathscr{D}^{\tau_2}\times\mathscr{D}^{\tau_3})$ is such that $\tilde{\varphi}(\overline{\bm{\xi}}|_{\tau_1},\overline{\bm{\xi}}|_{\tau_2},\overline{\bm{\xi}}|_{\tau_3}) = \varphi(\overline{\bm{\xi}})$ and $\bm{f} \in (\mathbb{R}^d)^{\mathfrak{L}^\tau}$ is such that $\bm{f}|_{\mathfrak{L}^{\tau_j}} = \bm{f}^j$ for $1 \le j \le 3$.
    The formula~\eqref{eq:J-n-inhomo-induction} follows from~\eqref{eq:J-tree-inhomo-induction}.
\end{proof}

\subsection{Energy spectrum}
\label{sec:energy-spectrum}

Now we use Wick's probability theorem to give diagrammatic expansions for the energy spectrum $\mathbb{E} \mathcal{W} [\Pi_n u]$.
In our Wick renormalized setting, Wick's theorem has the form as stated in the following lemma, which follows directly from~\cite[Theorem~3.9, p.~26]{Janson1997}.

\begin{lemma}
    \label{lem:wick}
    Let $\tau^\pm \in \mathscr{T}^\pm$ and $\bm{k}^\pm \in \mathscr{D}^{\tau^\pm}$.
    If $\tau^+$ and $\tau^-$ have different orders, then $\mathbb{E}\bigl( \mathfrak{g}^{\tau^+}_{\bm{k}^+} \mathfrak{g}^{\tau^-}_{\bm{k}^-} \bigr) = 0$.
    Otherwise, let $\bm{k} \in \mathscr{D}^{\tau^*,L}$, where $\tau^* = (\tau^+,\tau^-) \in \mathscr{T}^*$, be such that $\bm{k}^\pm = \bm{k}|_{\tau^\pm}$, then
    \begin{equation}
        \label{eq:wick-identity}
        \mathbb{E}\bigl( \mathfrak{g}^{\tau^+}_{\bm{k}^+} \mathfrak{g}^{\tau^-}_{\bm{k}^-} \bigr)
        = \sum_{\wp \in \mathfrak{P}^{\tau^*}} \prod_{\mathfrak{p} = \{\mathfrak{l}^+,\mathfrak{l}^-\} \in \wp} \bm{1}_{\bm{k}_{\mathfrak{l}^+} = \bm{k}_{\mathfrak{l}^-}}
        = \sum_{\cp = (\tau^*,\wp) \in \mathscr{K}} \bm{1}_{\bm{k} \in \mathscr{D}^\cp}.
    \end{equation}
\end{lemma}

As a consequence, we have $\mathbb{E} \mathcal{W} [\Pi_n u] = \sum_{0 \le j \le n} \mathbb{E} \mathcal{W} [u^j].$

\subsubsection{Periodic setting}

\label{sec:energy-spec-periodic}

For all $\cp = (\tau^*,\wp) \in \mathscr{K}$, define $\psi^\cp \in \mathscr{S}(\mathscr{D}^\cp)$ by \begin{equation*}
    \psi^\cp(\bm{k}) = \psi^{\tau^+}(\bm{k}|_{\tau^+}) \psi^{\tau^-}(\bm{k}|_{\tau^-})
    = \prod_{\mathfrak{p} \in \wp} |\psi(\bm{k}^\flat_{\mathfrak{p}})|^2,
\end{equation*}
and set $\Psi^\cp_t = \Theta^\cp \comp \Omega^\cp$.
For all $n \in \mathbb{N}$, all $k \in \mathbb{Z}^d_L$, and all $t > 0$, we define
\begin{equation}
    \label{eq:def-N-hom}
    N^n_{t,k} = \sum_{\cp \in \mathscr{K}_n} N^q_{t,k},
    \quad
    N^\cp_{t,k} = \varsigma_\cp \Bigl(\frac{\lambda}{L^d}\Bigr)^{2n} \sum_{\bm{k} \in \mathscr{D}^{\cp,L}_k} \Psi^\cp_t(\bm{k}) \psi^\cp(\bm{k}).
    \index{Functions and random variables!Energy spectra!Energy spectra (periodic setting) $N^n_{t,k}$, $N^\cp_{t,k}$}
\end{equation}

\begin{remark}
    \label{rmk:def-N-hom-smooth-k}
    We can extend $k \mapsto N^\cp_{t,k}$ to a function on $\mathbb{R}^d$.
    In fact, for all $k \in \mathbb{R}^d$, we define
    \begin{equation*}
        N^\cp_{t,k} = \varsigma_\cp \Bigl(\frac{\lambda}{L^d}\Bigr)^{2n} \sum_{\bm{k} \in \mathscr{D}^{\cp,L}_0} \Psi^\cp_t(\bm{k}) \psi^\cp(\bm{k}+k).
    \end{equation*}
    This coincides with the definition~\eqref{eq:def-N-hom} when $k \in \mathbb{Z}^d_L$ because $\Psi^\cp_t(\cdot + k) = \Psi^\cp_t(\cdot)$.
\end{remark}

\begin{proposition}
    \label{prop:J-N-id}
    For all $n,n' \in \mathbb{N}$, all $k,k' \in \mathbb{Z}^d_L$, and all $t > 0$, we have
    \begin{equation}
        \label{eq:J-N-id}
        \mathbb{E}\bigl( J^n_{t,k} \overline{J^{n'}_{t,k'}} \bigr) = \bm{1}_{(n,k) = (n',k')} N^{n}_{t,k}.
    \end{equation}
    Particularly, the phase randomization condition~\eqref{eq:cdt-phase-randomization-wp} holds.
\end{proposition}
\begin{proof}
    If $n \ne n'$, then $\mathbb{E}\bigl( J^n_{t,k} \overline{J^{n'}_{t,k'}} \bigr) = 0$ by Lemma~\ref{lem:wick}.
    If $n=n'$ and let $\tau^* = (\tau^+,\tau^-) \in \mathscr{T}^*$.
    Then Lemma~\ref{lem:wick} implies that
    \begin{align*}
        \mathbb{E}\bigl( J^n_{t,k} \overline{J^{n}_{t,k'}} \bigr)
        & = \varsigma_\cp \Bigl(\frac{\lambda}{L^d}\Bigr)^{2n} \sum_{\bm{k}^+ \in \mathscr{D}^{\tau^+}_k} \sum_{\bm{k}^- \in \mathscr{D}^{\tau^+}_{k'}} \Psi^{\tau^+}_t(\bm{k}^+) \Psi^{\tau^-}_t(\bm{k}^-) \psi^{\tau^+}(\bm{k}^+) \psi^{\tau^-}(\bm{k}^-) \mathbb{E} \bigl( \mathfrak{g}^{\tau^+}_{\bm{k}^+} \mathfrak{g}^{\tau^-}_{\mathfrak{k}^-} \bigr) \\
        & = \varsigma_\cp \Bigl(\frac{\lambda}{L^d}\Bigr)^{2n} \sum_{\bm{k}\in \mathscr{D}^{\tau^*,L}_{k,k'}} \sum_{\cp = (\tau^*,\wp) \in \mathscr{K}}
        \Psi^{\tau^+}_t(\bm{k}|_{\tau^+}) \Psi^{\tau^-}_t(\bm{k}|_{\tau^-}) \psi^{\tau^+}(\bm{k}|_{\tau^+}) \psi^{\tau^-}(\bm{k}|_{\tau^-}) \bm{1}_{\bm{k} \in \mathscr{D}^\cp}.
    \end{align*}
    By Lemma~\ref{lem:deco-formula}, if $k \ne k'$, then $\mathscr{D}^{\tau^*,L}_{k,k'} \cap \mathscr{D}^\cp = \emptyset$ and therefore $\mathbb{E}\bigl( J^n_{t,k} \overline{J^{n}_{t,k'}} \bigr) = 0$.
    If $k = k'$, then~\eqref{eq:J-N-id} follows by observing that for all $\bm{k} \in \mathscr{D}^{\cp,L}$, we have
    \begin{equation*}
        \Psi^\cp_t(\bm{k}) = \Psi^{\tau^+}_t(\bm{k}|_{\tau^+}) \Psi^{\tau^-}_t(\bm{k}|_{\tau^-}), \quad
        \psi^\cp(\bm{k}) = \psi^{\tau^+}(\bm{k}|_{\tau^+}) \psi^{\tau^-}(\bm{k}|_{\tau^-}).\qedhere
    \end{equation*}
\end{proof}

\begin{proposition}
    \label{prop:E-spec-N-id}
    Let the solution $u$ to WNLS be given as in \S\ref{sec:cauchy-periodic}, then for all $n \in \mathbb{N}$, we have
    \begin{equation}
        \label{eq:E-spec-N-id}
        \mathbb{E} \mathcal{W}[u^{2n+1}(t)](x,\xi)
        = \frac{1}{L^d} \sum_{k \in \mathbb{Z}^d_L}  N^n_{t,k} \bm{\delta}_k(\xi).
    \end{equation}
\end{proposition}
\begin{proof}
    By Proposition~\ref{prop:J-N-id}, we have
    \begin{equation*}
        \mathbb{E} \mathcal{W}[u^{2n+1}(t)]
        = \frac{1}{L^d} \sum_{k \in \mathbb{Z}^d_L}  \mathbb{E}\mathcal{W}[J^{n}_{t,k} e^{2\pi i k\cdot x - |k|^2t /2}]
        = \frac{1}{L^d} \sum_{k \in \mathbb{Z}^d_L}  N^n_{t,k} \mathcal{W}[e^{2\pi i k\cdot x}].
    \end{equation*}
    The identity~\eqref{eq:E-spec-N-id} follows from $\mathcal{W}[e^{2\pi i k\cdot x}](x,\zeta)
    = \bm{\delta}_k(\zeta)$.
\end{proof}

\subsubsection{Non-periodic setting}

\label{sec:energy-spec-non-periodic}

Recall that, the Wigner transform of $f,g \in L^2(\mathbb{R}^d)$ is defined by
\begin{equation*}
    \mathcal{W}[f,g] (x,\xi)
    = \int e^{-2\pi i y \cdot \xi} f(x+y/2) \overline{g(x-y/2)} \diff y
    = \int e^{2\pi i x \cdot \eta} \widehat{f}(\xi+\eta/2) \overline{\widehat{g}(\xi-\eta/2)} \diff \eta,
\end{equation*}
and we also write $\mathcal{W}[f] = \mathcal{W}[f,f]$ for simplicity.
Let us denote
\begin{equation*}
    \mathcal{V}[\phi](k,x,\zeta) = \mathcal{W}[\phi(\cdot,k)](x,\zeta), \quad
    \widehat{\mathcal{V}}[\phi](k,\eta,\zeta) = \widehat{\mathcal{W}}[\phi(\cdot,k)](\eta,\xi).
\end{equation*}
For all $\cp = (\tau^*,\wp) \in \mathscr{K}$, define $\widehat{\mathcal{V}}^\cp[\phi] \in \mathscr{D}(\mathscr{D}^\cp \times \mathscr{C}^\cp \times \mathscr{D}^\cp)$ by 
\begin{equation*}
    \widehat{\mathcal{V}}^\cp[\phi](\bm{k},\bm{\eta},\bm{\zeta})
    = \prod_{\mathfrak{p} \in \wp} \widehat{\mathcal{V}}(\bm{k}^\flat_{\mathfrak{p}},\bm{\eta}^\flat_{\mathfrak{p}},\bm{\zeta}^\flat_{\mathfrak{p}})
    = \prod_{\mathfrak{p} \in \wp} \widehat{\phi}(\bm{k}^\flat_{\mathfrak{p}},\bm{\zeta}^\flat_{\mathfrak{p}}+\bm{\eta}^\flat_{\mathfrak{p}}/2) \overline{\widehat{\phi}(\bm{k}^\flat_{\mathfrak{p}},\bm{\zeta}^\flat_{\mathfrak{p}}-\bm{\eta}^\flat_{\mathfrak{p}}/2)}.
\end{equation*}
For all $\zeta,\eta \in \mathbb{R}^d$, and all $\varphi \in \mathscr{S}(\mathscr{D}^\cp)$, $\psi \in \mathscr{S}(\mathscr{C}^\cp)$, define
    \begin{align*}
        \int_{\mathscr{D}^\cp_\zeta} \varphi(\bm{\zeta}) \diff \bm{\zeta}
        = \int_{(\mathbb{R}^d)^{\wp}} \bm{\delta}_\zeta\Bigl(\sum_{\mathfrak{p} \in \wp} \iota_{\mathfrak{p}} \bm{f}_{\mathfrak{p}} \Bigr)\varphi(\breve{\bm{f}}) \diff \bm{f}, \quad
        \int_{\mathscr{C}^\cp_\eta} \psi(\bm{\eta}) \diff \bm{\eta}
        = \int_{(\mathbb{R}^d)^{\wp}} \bm{\delta}_\eta\Bigl(\sum_{\mathfrak{p} \in \wp} \bm{g}_{\mathfrak{p}} \Bigr) \varphi(\tilde{\bm{g}}) \diff \bm{g},
    \end{align*}
where $\iota_{\mathfrak{p}} \coloneqq \iota_{\mathfrak{l}}$ with $\tau^+ \cap \mathfrak{p} = \{\mathfrak{l}\}$, and $\breve{\bm{f}}$ and $\tilde{\bm{g}}$ are respectively unique elements in $\mathscr{D}^\cp$ and $\mathscr{C}^\cp$ such that $\breve{\bm{f}}^\flat = \bm{f}$ and $ \tilde{\bm{g}}^\flat = \bm{g}$.
For all $n \ge 0$, $k \in \mathbb{R}^d$ (following Remark~\ref{rmk:def-N-hom-smooth-k}) and $t > 0$, define
\begin{equation*}
    \mathcal{E}^n_{t,k}(x,\zeta) = \sum_{\cp \in \mathscr{K}_n} \mathcal{E}^\cp_{t,k}(x,\zeta),
    \quad
    \mathcal{N}^n_{t,k}(x) = \sum_{\cp \in \mathscr{K}_n} \mathcal{N}^\cp_{t,k}(x)
    \index{Functions and random variables!Energy spectra!Energy spectra (non-homogeneous setting) $\mathcal{N}^n_{t,k}$, $\mathcal{N}^\cp_{t,k}$, $\mathcal{E}^n_{t,k}$, $\mathcal{E}^\cp_{t,k}$}
\end{equation*}
as follows:
the Fourier transform $ \widehat{\mathcal{E}}^\cp_{t,k}(\eta,\zeta) = \mathcal{F}_{x \to \eta} \{ \mathcal{E}^\cp_{t,k}(x,\zeta)\}$ is given by
\begin{equation*}
    \widehat{\mathcal{E}}^\cp_{t,k}(\eta,\zeta)
    = \varsigma_\cp \Bigl(\frac{\lambda}{L^d}\Bigr)^{2n} \sum_{k \in \mathscr{D}^{\cp,L}_k} \int_{\mathscr{C}^\cp_\eta} \int_{\mathscr{D}^\cp_\zeta} \Psi^{\tau^*}_t\bigl(\bm{k}+\epsilon(\bm{\zeta}+\iota\bm{\eta}/2)\bigr) \widehat{\mathcal{V}}^\cp[\phi](\bm{k},\bm{\zeta},\bm{\eta}) \diff \bm{\eta} \diff \bm{\zeta},
\end{equation*}
where $\Psi^{\tau^*}_t = \Theta^\cp_t \comp \Omega^{\tau^*}$, and $\mathcal{N}^\cp_{t,k}(x)$ is given by the integration
\begin{equation*}
    \mathcal{N}^\cp_{t,k}(x) = \int \mathcal{E}^\cp_{t,k}(x,\zeta) \diff \zeta.
\end{equation*}
As in Remark~\ref{rmk:def-N-hom-smooth-k}, we shall also extend the definitions of $\mathcal{E}^\cp_{t,k}$ and $\mathcal{N}^\cp_{t,k}$, and thus the definitions of $\mathcal{E}^n_{t,k}$ and $\mathcal{N}^n_{t,k}$, to all $k \in \mathbb{R}^d$.

\begin{proposition}
    \label{prop:J-E-id}
    For all $n,n' \in \mathbb{N}$, all $k,k' \in \mathbb{Z}^d_L$, and all $t > 0$, we have
    \begin{equation}
        \label{eq:J-E-id}
        \mathbb{E}\mathcal{W}[\mathcal{J}^n_{t,k}, \mathcal{J}^{n'}_{t,k'}] = \bm{1}_{(n,k) = (n',k')} \mathcal{E}^{n}_{t,k}.
    \end{equation}
    Particularly, the phase randomization condition~\eqref{eq:cdt-phase-randomization-wp} holds.
\end{proposition}
\begin{proof}
    Similarly as for Proposition~\ref{prop:J-N-id}, it suffices to establish~\eqref{eq:J-E-id} when $(n,k) = (n',k')$:
    \begin{align*}
        \mathbb{E}\widehat{\mathcal{W}}&[\mathcal{J}^{n}_{t,k} , \mathcal{J}^{n}_{t,k} ](\eta,\zeta)
        = \sum_{\tau^* \in \mathscr{T}^*_n} \mathbb{E}\bigl( \mathcal{J}^{\tau^+}_{t,k}(\zeta+\eta/2) \mathcal{J}^{\tau^-}_{t,k}(\zeta-\eta/2) \bigr) \\
        & = \sum_{\cp = (\tau^*,\wp) \in \mathscr{K}_n}
        \varsigma_{\cp} \Bigl(\frac{\lambda}{L^d}\Bigr)^{2n}
        \sum_{\bm{k} \in \mathscr{D}^{\cp}_k}
        \prod_{\sigma \in \{+,-\}} \int_{\mathscr{D}^{\tau^\sigma}_{\zeta+\sigma\eta/2}} 
        \Psi^{\tau^\sigma}_{t} (\bm{k}|_{\tau^\sigma} + \epsilon \bm{\xi}^\sigma)
        \widehat{\phi}^{\tau^\sigma}(\bm{k}|_{\tau^\sigma}+\bm{\xi}^\sigma)
        \diff \bm{\xi}^\sigma.
    \end{align*}
    To conclude, we use the Following Lemma~\ref{lem:cp-deco-wigner}.
    In fact, the Jacobian of the map $\mathscr{C}^\cp \times \mathscr{D}^\cp \ni (\bm{\eta},\bm{\zeta}) \mapsto \bm{\xi} \in \mathscr{D}^{\tau^*}$ is equal to one, and we have
    \begin{align*}
        \prod_{\sigma \in \{+,-\}} \Psi^{\tau^\sigma}_{t} (\bm{k}|_{\tau^\sigma} + \epsilon \bm{\xi}|_{\tau^\sigma})
        & =\Psi^\cp_t\bigl(\bm{k}+\epsilon(\bm{\zeta} + \iota\bm{\eta}/2)\bigr), \\
        \prod_{\sigma \in \{+,-\}} \widehat{\phi}^{\tau^\sigma}(\bm{k}|_{\tau^\sigma}+\bm{\xi}|_{\tau^\sigma})
        & =  \widehat{\mathcal{V}}^\cp[\phi](\bm{k},\bm{\eta},\bm{\zeta}).\qedhere
    \end{align*}
\end{proof}

\begin{lemma}
    \label{lem:cp-deco-wigner}
    For all $\cp = (\tau^*,\wp) \in \mathscr{K}$, we have the linear space isomorphism
    \begin{equation}
        \label{eq:Wigner-change-var}
        \mathscr{C}^\cp \times \mathscr{D}^\cp
        \simeq \mathscr{D}^{\tau^*}\quad
        (\bm{\eta},\bm{\zeta})
        \mapsto \bm{\xi} \mapsto \bm{\zeta} + \iota \bm{\eta}/2.
    \end{equation}
    Moreover~\eqref{eq:Wigner-change-var} induces for all $\eta,\zeta \in \mathbb{R}^d$ a bijection
    $\mathscr{C}^\cp_\eta \times \mathscr{D}^\cp_\zeta \to \mathscr{D}^{\tau^*}_{\zeta+\eta/2,\zeta-\eta/2}.$
\end{lemma}
\begin{proof}
    Clearly~\eqref{eq:Wigner-change-var} is a linear space homomorphism.
    To show that it is also an isomorphism, it suffices to construct its inverse.
    Let $\bm{\xi} \in \mathscr{D}^{\tau^*}$ and let $\bm{f},\bm{g} : \wp \to \mathbb{R}^d$ be defined by
    \begin{equation*}
        \bm{f}_{\mathfrak{p}} = (\bm{\xi}_{\mathfrak{l}^+} + \bm{\xi}_{\mathfrak{l}^-}) / 2, \quad
        \bm{g}_{\mathfrak{p}} = \bm{\xi}_{\mathfrak{l}^+} - \bm{\xi}_{\mathfrak{l}^-}, \quad
        \forall \mathfrak{p} = \{\mathfrak{l}^+,\mathfrak{l}^-\} \in \wp.
    \end{equation*}
    Let $\bm{\eta} \in \mathscr{C}^\cp$ and $\bm{\zeta} \in \mathscr{D}^\cp$ be such that $\bm{\eta}^\flat = \bm{g}$ and $\bm{\zeta}^\flat = \bm{f}$.
    Then the relation $\bm{\xi} = \bm{\zeta} + \iota \bm{\eta}/2$ hold on $\mathfrak{L}^\cp$ and therefore, by Lemma~\ref{lem:deco-formula}, also holds on $\cp$.
    If $\bm{\zeta}_{\mathfrak{r}^{\tau^\pm}} = \zeta$ and $\bm{\eta}_{\mathfrak{r}^{\tau^\pm}} = \eta $, then
    \begin{equation*}
        \bm{\xi}_{\mathfrak{r}^{\tau^\pm}} = \bm{\zeta}_{\mathfrak{r}^{\tau^\pm}} \pm \bm{\eta}_{\mathfrak{r}^{\tau^\pm}}/2 = \zeta \pm \eta/2.
    \end{equation*}
    Conversely, if $ \bm{\xi}_{\mathfrak{r}^{\tau^\pm}} = \zeta \pm \eta/2 $, then
    \begin{equation*}
        \bm{\zeta}_{\mathfrak{r}^{\tau^\pm}} = (\bm{\xi}_{\mathfrak{r}^{\tau^+}} + \bm{\xi}_{\mathfrak{r}^{\tau^-}}) / 2 = \zeta, \quad
        \bm{\eta}_{\mathfrak{r}^{\tau^\pm}} = \bm{\xi}_{\mathfrak{r}^{\tau^+}} - \bm{\xi}_{\mathfrak{r}^{\tau^-}} = \eta.\qedhere
    \end{equation*}
\end{proof}

\begin{proposition}
    \label{prop:E-spec-E-id}
    Let the solution $u$ to WNLS be given as in \S\ref{sec:cauchy-non-periodic}, then for all $n \in \mathbb{N}$, we have
    \begin{equation}
        \label{eq:E-spec-E-id}
        \mathbb{E}\mathcal{W}\bigl[u^{2n+1}(t)\bigr](x/\epsilon,\xi)
        = \frac{1}{\epsilon^d L^d} \sum_{k \in \mathbb{Z}^d_L} \mathcal{E}^n_{t,k}\Bigl(x - \epsilon t \xi, \frac{\xi-k}{\epsilon}\Bigr).
    \end{equation}
    If in addition $\supp \widehat{\phi} \subset \{|\zeta| \le R\} \times \mathbb{R}^d$ for some $R > 0$, then for all $t > 0$ and $x \in \mathbb{R}^d$, we have
    \begin{equation}
        \label{eq:E-spec-supp-cdt}
        \supp \mathbb{E} \mathcal{W}\bigl[u^{2n+1}(t)\bigr](x,\cdot) \subset \{\xi \in \mathbb{R}^d : \mathrm{dist}(\xi,\mathbb{Z}^d_L) \le (2n+1)R\epsilon\}.
    \end{equation}
    Consequently, when $(2n+1) R\epsilon < 1/(2L)$, for all $k \in \mathbb{Z}^d_L$, we have
    \begin{equation}
        \label{eq:E-spec-xi-int}
        L^d \int_{|\xi-k|\le (2n+1)R\epsilon} \mathbb{E}\bigl[u^{2n+1}(t)\bigr](x/\epsilon,\xi) \diff \xi
        = \int \mathcal{E}^n_{t,k}(x-\epsilon t(k+\epsilon \zeta),\zeta) \diff \zeta.
    \end{equation}
\end{proposition}
\begin{proof}
    Note that $A^{2n+1}_{k}(t,\cdot) = e^{-\epsilon t k \cdot \nabla + \epsilon^2 i t \Delta / 4\pi} \mathcal{J}^n_{t,k}$.
    A direct computation shows that
    \begin{equation*}
        \mathcal{F}_{x\to\xi}\bigl\{ A^{2n+1}_k(t,\epsilon x)e^{2\pi i k \cdot x - |k|^2t/2} \bigr\}
        = \frac{e^{-\pi i t|\xi|^2}}{\epsilon^d} \widehat{\mathcal{J}}^n_{t,k}\Bigl(\frac{\xi-k}{\epsilon}\Bigr).
    \end{equation*}
    Therefore, by Proposition~\ref{prop:J-E-id}, we have
    \begin{align*}
        \mathbb{E}\widehat{\mathcal{W}}&\bigl[u^{2n+1}(t)\bigr](\eta,\xi)
        = \frac{1}{L^d} \sum_{k \in \mathbb{Z}^d_L} \mathbb{E}\widehat{\mathcal{W}}\bigl[A^{2n+1}_k(t,\epsilon x)e^{2\pi i k \cdot x - |k|^2t/2}\bigr](\eta,\xi) \\
        & = \frac{1}{\epsilon^{2d} L^d} \sum_{k \in \mathbb{Z}^d_L} e^{-2\pi i t \xi \cdot \eta} \mathbb{E}\widehat{\mathcal{W}}\bigl[\mathcal{J}^n_{t,k}\bigr]\Bigl(\frac{\eta}{\epsilon},\frac{\xi-k}{\epsilon}\Bigr)
        = \frac{1}{\epsilon^{2d} L^d} \sum_{k \in \mathbb{Z}^d_L} e^{-2\pi i t \xi \cdot \eta} \widehat{\mathcal{E}}^n_{t,k}\Bigl(\frac{\eta}{\epsilon},\frac{\xi-k}{\epsilon}\Bigr).
    \end{align*}
    We conclude~\eqref{eq:E-spec-E-id} by taking the inverse Fourier transform of both sides.
    The support condition~\eqref{eq:E-spec-supp-cdt} follows directly from Lemma~\ref{lem:deco-formula} and readily implies the integral identity~\eqref{eq:E-spec-xi-int}.
\end{proof}

\section{Effective dynamics: homogeneous setting}

\label{sec:hom}

In this section we prove Theorem~\ref{thm:hom-inhom} in the homogeneous setting.
Recall that this corresponds to the periodic case with the scaling laws $\lambda^{-2} = L^\alpha$ where $0<\alpha<2/(1+4/d^*)$ and $\beta = \infty$.

\subsection{General couple}

We first give an upper bound of $N^\cp_{t,k}$, defined by~\eqref{eq:def-N-hom}, for all $\cp \in \mathscr{K}$, which shows that couples which are not regular have negligible contributions to the effective dynamics.

\begin{lemma}
    \label{lem:I-op-est}
    Let $\cp \in \mathscr{K}_n$, and let $V \in C^\infty(\mathbb{R}^{\mathfrak{B}^\cp} \times \mathscr{D}^\cp)$ be such that $\supp V \subset \mathbb{R}^{\mathfrak{B}^\cp} \times B$ where $B$ is bounded by $nR$ for some $R \ge 1$. For all $t > 0$ and $k \in \mathbb{R}^d$ set
    \begin{equation*}
        \mathcal{I}^\cp_{t,k} V = \Bigl(\frac{\lambda}{L^d}\Bigr)^{2n} \sum_{k \in \mathscr{D}^{\cp,L}_k} \int_{\mathscr{O}^\cp_t} e^{2\pi i \bm{t} \cdot \Omega^\cp(\bm{k})} V(\lambda^2 \bm{t},\bm{k}) \diff \bm{t}.
        \index{Operations and transforms!Oscillatory sums and integrals!Energy spectrum operator on couples $\mathcal{I}^\cp_{t,k}$}
    \end{equation*}
    For all $m \in \mathbb{N}^d$, and uniformly for all $(t,k) \in [0,1] \times \mathbb{R}^d$, we have the estimate
    \begin{equation}
        \label{eq:reso-int-upper-bound}
        |\mathcal{I}^\cp_{t/\lambda^2,k} V| \lesssim C^n t^{2n} (nR)^{4nd} \ln(1+nR^2L^\alpha)^{2n} L^{-\nu \ind(\cp)} \|V\|_{L^\infty(\mathscr{O}^\cp_t,L^\infty)}.
    \end{equation}
\end{lemma}
\begin{proof}
    Because the norm on the right hand side of~\eqref{eq:reso-int-upper-bound} is invariant under translation in $\bm{k}$, we may therefore assume that $k=0$ without losing generality.
    Let $\tilde{\mathcal{F}}$ denote the inverse temporal Fourier transform, i.e., with the variable $\bm{t}$, and denote $\tilde{V}_t(\bm{\omega},\bm{k}) = \tilde{\mathcal{F}} (\bm{1}_{\mathscr{O}^\cp_t} V)(\bm{\omega},\bm{k})$.
    Performing the change of variable $\bm{t} \to \bm{t}/\lambda^2$, observing that
    \begin{equation}
        \label{eq:time-ord-scaling-id}
        \bm{1}_{\mathscr{O}^\cp_{t /\lambda^2}}(\bm{t}) = \bm{1}_{\mathscr{O}^\cp_{t}}(\lambda^2\bm{t}),
    \end{equation}
    and setting $\bm{\varrho} = \lambda^{-2} \Omega^\cp(\bm{k}) -\bm{\omega}$,  we then obtain
    \begin{align*}
        \mathcal{I}^\cp_{t/\lambda^2,k} V
        & = \frac{1}{(\lambda L^d)^{2n}} \sum_{\bm{k} \in \mathscr{D}^{\cp,L}_0} \int_{\mathscr{O}^\cp_t} e^{2\pi i \lambda^{-2} \bm{t} \cdot \Omega^\cp(\bm{k})} \Bigl( \int e^{-2\pi i \bm{t} \cdot \bm{\omega}} \tilde{V}_t(\bm{\omega},\bm{k}) \diff \bm{\omega} \Bigr) \diff \bm{t} \\
        & = \frac{1}{(\lambda L^d)^{2n}} \sum_{\bm{k} \in \mathscr{D}^{\cp,L}_0} \int \Theta^\cp_t(\bm{\varrho}) \tilde{V}_t(\bm{\omega},\bm{k}) \diff \bm{\omega}.
    \end{align*}
    Because $t/\langle t \omega \rangle \lesssim 1/\langle \omega \rangle$ when $0 \le t \le 1$, by Proposition~\ref{prop:time-ord-fourier-est}, we have
    \begin{equation*}
        |\mathcal{I}^\cp_{t/\lambda^2,k} V|
        \lesssim \frac{C^n}{(\lambda L^d)^{2n}} \sup_{\mu \in \mathscr{O}(\mathfrak{B}^\cp)} \int \sum_{\bm{k} \in \mathscr{D}^{\cp,L}_0} \Bigl( \prod_{\mathfrak{b} \in \mathfrak{B}^\cp} \frac{1}{\langle \mu(\bm{\varrho})_{\mathfrak{b}} \rangle} \Bigr) |\tilde{V}_t(\bm{\omega},\bm{k})| \diff \bm{\omega}.
    \end{equation*}
    For all $\bm{z} \in \mathbb{Z}^{\mathfrak{B}^\cp}$, let $Q_{\bm{z}}$ be the cube centered at $\bm{z}$ with edge length equal to one and with faces parallel to coordinate planes.
    For all $\mu \in \mathscr{O}(\mathfrak{B}^\cp)$, all $\bm{\omega} \in \mathbb{R}^{\mathfrak{B}^\cp}$ and all $\bm{z} \in \mathbb{Z}^{\mathfrak{B}^\cp}$ let
    \begin{equation*}
        G^\mu_{\bm{\omega},\bm{z}} = \{\bm{k} \in \mathscr{D}^{\cp,L}_0 \cap B: \mu(\bm{\varrho}) \in Q_{\bm{z}}\}.
    \end{equation*}
    For any $\bm{\omega}$, we have the decomposition:
    \begin{equation*}
        \mathscr{D}^{\cp,L}_0 \subset \bigcup_{\bm{z} \in \mathbb{R}^{\mathfrak{B}^\cp}} (\mathscr{D}^{\cp,L}_0 \cap G^\mu_{\bm{\omega},\bm{z}}).
    \end{equation*}
    Uniformly in $\bm{\omega}$ and $\bm{z}$, we have the estimate
    \begin{equation*}
        \sum_{k \in \mathscr{D}^{\cp,L}_0 \cap G^\mu_{\bm{\omega},\bm{z}}} \Bigl( \prod_{\mathfrak{b} \in \mathfrak{B}^\cp} \frac{1}{\langle \mu(\bm{\varrho})_{\mathfrak{b}} \rangle} \Bigr) |\tilde{V}_t(\bm{\omega},\bm{k})|
        \lesssim C^n \card( G_{\bm{\omega},\bm{z}} )
        \Bigl( \prod_{\mathfrak{b} \in \mathfrak{B}^\cp} \frac{1}{\langle \bm{z}\rangle_{\mathfrak{b}}} \Bigr) \|\tilde{V}_t(\bm{\omega},\cdot)\|_{L^\infty}.
    \end{equation*}
    Note that if $\bm{k} \in B$, then $|\Omega^\cp(\bm{k})| \lesssim nR^2$, so the set $U_{\bm{\omega}} = \{\bm{z} \in \mathbb{Z}^d : G_{\bm{\omega},\bm{z}} \ne \emptyset\}$ has a diameter $\lesssim  n R^2 \lambda^{-2} = n R^2 L^\alpha$ for all $\bm{\omega}$.
    We therefore have
    \begin{align*}
        \sum_{\bm{k} \in \mathscr{D}^{\cp,L}_0} \Bigl( \prod_{\mathfrak{b} \in \mathfrak{B}^\cp} \frac{1}{\langle \mu(\bm{\varrho})_{\mathfrak{b}} \rangle} \Bigr) |\tilde{V}_t(\bm{\omega},\bm{k})|
        & \lesssim C^n \sum_{\bm{z} \in U_{\bm{\omega}}} \card(G_{\bm{\omega},\bm{z}}) \Bigl( \prod_{\mathfrak{b} \in \mathfrak{B}^\cp} \frac{1}{\langle \bm{z}\rangle_{\mathfrak{b}}} \Bigr) \|\tilde{V}_t(\bm{\omega},\cdot)\|_{L^\infty} \\
        & \lesssim C^n \ln(1+nR^2L^\alpha)^{2n} \sup_{\bm{z} \in \mathbb{Z}^{\mathfrak{B}^\cp}} \card(G_{\bm{\omega},\bm{z}}) \|\tilde{V}_t(\bm{\omega},\cdot)\|_{L^\infty}
    \end{align*}
    To conclude, we first use Proposition~\ref{prop:deco-count-cp} to bound
    \begin{equation*}
        \card(G_{\bm{\omega},\bm{z}}) 
        \lesssim C^n (nR)^{4nd} (\lambda^2 L^{2d})^n L^{-\nu \ind(\cp)}, 
    \end{equation*}
    then we integrate in $\bm{\omega}$ and use the estimate
    \begin{equation*}
        \|\tilde{V}_t\|_{L^1 L^\infty}
        = \|\Theta^\cp_t * \tilde{V}_t\|_{L^1 L^\infty}
        \lesssim \|\Theta^\cp_t\|_{L^1} \|\tilde{V}_t\|_{L^\infty L^\infty}
        \lesssim \|\bm{1}_{\mathcal{O}_t} V\|_{L^1 L^\infty}
        \lesssim t^{2n} \|V\|_{L^\infty(\mathscr{O}^\cp_t,L^\infty)}.\qedhere
    \end{equation*}
\end{proof}

\begin{proposition}
    \label{prop:E-spec-general-upper-bound-hom}
    For all $\cp \in \mathscr{K}$, all $t > 0$ and all $k \in \mathbb{Z}^d_L$, we have
    \begin{equation}
        \label{eq:N-I-id-hom}
        N^\cp_{t,k} = \varsigma_\cp \mathcal{I}^\cp_{t,k} \psi^\cp.\qedhere
    \end{equation}
    Consequently, supposing that $\widehat{\psi} \in C_c^\infty(\mathbb{R}^d)$ with $\supp \psi$ bounded by $R \ge 1$, if $\cp \in \mathscr{K}^\reg_n$ with $n \ge 0$, then for all $m \in \mathbb{N}^d$ and uniformly for all $(t,k) \in [0,t] \times \mathbb{R}^d$, we have
    \begin{equation*}
        |k^m N^{\cp}_{t/\lambda^2,k}| \lesssim
        C^n t^{2n} (nR)^{4nd} \ln(1+nR^2L^\alpha)^{2n} L^{-\nu \ind(\cp)} n^{|m|} \||\psi|^2\|_{L^\infty_{|m|/2}}^{4n+2}.
    \end{equation*}
    Particularly, if $\cp \notin \mathscr{K}^\reg$, then
    \begin{equation}
        \label{eq:cp-non-reg-est-hom}
        \lim_{L \to \infty} \sup_{(t,k) \in [0,1] \times \mathbb{R}^d} k^m N^\cp_{t/\lambda^2,k} = 0.
    \end{equation}
\end{proposition}
\begin{proof}
    The identity follows from~\eqref{eq:def-N-hom}.
    Write $\cp = (\tau^*,\wp)$.
    For all $\bm{k} \in \mathscr{D}^\cp$, denote $\bm{k}_\cp = \bm{k}_{\mathfrak{r}^{\tau^\pm}}$ for simplicity.
    Note that $k^m N^{\cp}_{t,k} = \varsigma_\cp \mathcal{I}^\cp_{t,k} (\bm{k}_\cp^m \psi^\cp)$ and $\supp \psi^\cp$ is bounded by $(2n+1) R$.
    Therefore, by Lemma~\ref{lem:I-op-est}, we have
    \begin{equation*}
        |k^m N^{\cp}_{t/\lambda^2,k}| \lesssim 
        C^n t^{2n} (nR)^{4nd} \ln(1+nR^2L^\alpha)^{2n} L^{-\nu \ind(\cp)} \|\bm{k}_\cp^m \psi^\cp\|_{L^\infty}.
    \end{equation*}
    We conclude with the estimate
    \begin{equation*}
        \|\bm{k}_\cp^m \psi^\cp\|_{L^\infty} 
        \lesssim n^{|m|} \||\psi|^2\|_{L^\infty_{|m|}}^{2n+1}
        \lesssim n^{|m|} \||\psi|^2\|_{L^\infty_{|m|/2}}^{4n+2}
        \qedhere
    \end{equation*}
\end{proof}

\subsection{Regular couples}

\label{sec:reg-cp-hom}

Now we give finer estimates and asymptotic analysis for $N^\cp_{t,k}$ when $\cp \in \mathscr{K}^\reg$, using the structure of regular couples and results from \S\ref{sec:oscillatory}.

\subsubsection{Algebraic structures}

Let $\cp \in \mathscr{K}^\reg$ and let $r \in \mathscr{I}_\cp$. Note that $r \in \mathscr{K}^\reg_1$, so we can write $\mathfrak{B}^r = \{\mathfrak{b}^+,\mathfrak{b}^-\}$ with $\iota_{\mathfrak{b}^\pm} = \pm 1$.
By Lemma~\ref{lem:deco-formula} and Definition~\ref{def:reso-function}, we ave $\Omega^\cp_{\mathfrak{b}^+} + \Omega^\cp_{\mathfrak{b}^-} = 0$.
Therefore, we define the bilinear map $\mho^\cp = (\mho^\cp_r)_{r \in \mathscr{I}_\cp}: \mathscr{D}^\cp \times \mathscr{D}^\cp \to \mathbb{R}^{\mathscr{I}_\cp}$ by setting
$\mho^\cp_r = \pm \Omega^\cp_{\mathfrak{b}^\pm},$
\index{Functions and random variables!Resonance functions!Resonance factors on regular couples $\mho^\cp_r$}
where $\mathfrak{b}^\pm \in \mathfrak{B}^r \cap \iota^{-1}(\pm 1).$
Next, define the linear isomorphism $\kappa_\cp : \mathbb{R}^{\mathscr{I}_\cp} \times \mathbb{R}^{\mathscr{I}_\cp} \to \mathbb{R}^{\mathfrak{B}^\cp}$ by setting
\begin{equation*}
    \kappa_\cp(\bm{\mu},\bm{s})_{\mathfrak{b}} = \bm{\mu}_{r} + \iota_{\mathfrak{\mathfrak{b}}} \bm{s}_{r}/2
\end{equation*}
if $\mathfrak{b} \in \mathfrak{B}^r$ and $r \in \mathscr{I}_\cp$.
Let $\kappa_\cp^*$ denote the pullback induced by $\kappa_\cp$.
For all $t > 0 $, define the linear integral operator $\Phi^\cp_t:L^\infty(\mathbb{R}^{\mathfrak{B}^\cp} \times (\mathbb{R}^{2d})^{\mathfrak{B}^\cp}) \to L^\infty(\mathbb{R}^{\mathscr{I}_\cp} \times (\mathbb{R}^{2d})^{\mathfrak{B}^\cp})$ by
\begin{equation*}
    \Phi^{\cp}_{t}U(\bm{s},\bm{z})
    = \int_{\mathbb{R}^{\mathscr{I}_\cp}} \kappa_\cp^* ( \bm{1}_{\mathscr{O}^\cp_{t}} U)(\bm{\mu},\bm{s},\bm{z}) \diff \bm{\mu}.
\end{equation*}
Define $\varpi^\cp = (\varpi^\cp_{r})_{r \in \mathscr{I}_\cp}$ by setting $\varpi^\cp_{r}(\bm{z}) = \bm{x}_{r} \cdot \bm{y}_{r}$ if $\bm{z} = (\bm{x}_{r},\bm{y}_{r})_{r \in \mathscr{I}_\cp}$.
Suppose that $\cp \in \mathscr{K}^\reg_n$ with $n \in \mathbb{N}$.
As in \S\ref{sec:oscillatory}, for all $L \in (0,\infty]$, define the linear functional $\mathcal{S}^\cp_L$\index{Operations and transforms!Oscillatory sums and integrals!Oscillatory sum/integral operator on couples $\mathcal{S}^\cp_L$} on $C_c^\infty(\mathbb{R}^{\mathscr{I}_\cp} \times (\mathbb{R}^{2d})^{\mathscr{I}_\cp})$ by
\begin{equation*}
    \mathcal{S}^\cp_L(\Phi)
    = \begin{cases}
        \displaystyle
        \frac{1}{L^{2nd}} \sum_{\bm{z} \in (\mathbb{Z}^{2d}_L)^{\mathscr{I}_\cp}} \int_{\mathbb{R}^{\mathscr{I}_\cp}} e^{2\pi i \bm{s} \cdot \varpi^\cp(\bm{z})} \Phi(\bm{s},\bm{z}) \diff \bm{s}, & L \in (0,\infty);\\
        \displaystyle
        \int_{(\mathbb{R}^{2d})^{\mathscr{I}_\cp}}\int_{\mathbb{R}^{\mathscr{I}_\cp}} e^{2\pi i \bm{s} \cdot \varpi^\cp(\bm{z})} \Phi(\bm{s},\bm{z}) \diff \bm{s} \diff \bm{z}, & L = \infty.
    \end{cases}
\end{equation*}

\begin{lemma}
    \label{lem:I-S-op-id}
    If $\cp \in \mathscr{K}^\reg$, then for all $t > 0$, all $k \in \mathbb{R}^d$ and all $V \in C^\infty(\mathbb{R}^{\mathfrak{B}^\cp}\times \mathscr{D}^\cp)$ such that $V(\bm{t},\cdot) \in C_c^\infty(\mathscr{D}^\tau)$ for all $\bm{t} \in \mathbb{R}^{\mathfrak{B}^\cp}$, we have
    \begin{equation}
        \label{eq:I-S-op-rel-id}
        \mathcal{I}^\cp_{t,k} V = \mathcal{S}^\cp_L\bigl(\theta_L \Phi^\cp_t \bigl(V|_{\mathscr{D}^\cp_k} \comp (\Lambda^\cp_k)^{-1}\bigr)\bigr).
    \end{equation}
    Particularly, observing that $\iota_\cp = 1$ when $\cp \in \mathscr{K}^\reg$, we have
    \begin{equation}
        \label{eq:N-S-op-id-hom}
        N^\cp_{t,k} = \mathcal{S}^\cp_L\bigl(\theta_L \Phi^\cp_t \bigl(\psi^\cp|_{\mathscr{D}^\cp_k} \comp (\Lambda^\cp_k)^{-1}\bigr)\bigr).
    \end{equation}
\end{lemma}
\begin{proof}
    Suppose that $\cp \in \mathscr{K}^\reg_n$ with $n \in \mathbb{N}$.
    By Lemma~\ref{lem:tree-deco-change-var}, the identities~\eqref{eq:Omega-change-var} and~\eqref{eq:time-ord-scaling-id}, we have the following direct computation which establishes~\eqref{eq:I-S-op-rel-id} using the scaling law~$\lambda^{-2} = L^\alpha$:
    \begin{align*}
        \mathcal{I}^\cp_{t/\lambda^2,k} V
        & = \Bigl(\frac{\lambda}{L^d}\Bigr)^{2n} \sum_{\bm{k} \in \mathscr{D}^\cp_k}
        \iint e^{2\pi i \bm{s} \cdot \mho^\cp(\bm{k})}  \kappa_\cp^* (\bm{1}_{\mathscr{O}^\cp_t} V)(\lambda^2 \bm{\mu},\lambda^2\bm{s},\bm{k}) \diff \bm{\mu} \diff \bm{s} \\
        & = \frac{1}{L^{2nd}} \sum_{\bm{z} \in (\mathbb{Z}^{2d}_L)^{\mathscr{I}_\cp}}
        \iint e^{2\pi i \bm{s} \cdot \varpi^\cp(\bm{z})} \kappa_\cp^* (\bm{1}_{\mathscr{O}^\cp_t} V)(\bm{\mu},\lambda^2\bm{s},(\Lambda^\cp_k)^{-1}\bm{z}) \diff \bm{\mu} \diff \bm{s} \\
        & = \frac{1}{L^{2nd}} \sum_{\bm{z} \in (\mathbb{Z}^{2d}_L)^{\mathscr{I}_\cp}} \int e^{2\pi i \bm{s} \cdot \varpi^\cp(\bm{z})} \Phi^\cp_t V(\lambda^2\bm{s},(\Lambda^\cp_k)^{-1}\bm{z}) \diff \bm{s}.\qedhere
    \end{align*}
\end{proof}

Because of the identity~\eqref{eq:N-S-op-id-hom}, we can use Proposition~\ref{prop:int-approx} to study the limit $\lim_{L \to \infty} N^\cp_{t/\lambda^2,k}$ when $\cp \in \mathscr{K}^\reg$.
For all $n \ge 0$, all $t > 0$ and all $k \in \mathbb{R}^d$, define
    \begin{equation}
        \label{eq:def-N-infty-hom}
        N^{n,\infty}_{t,k} = \sum_{\cp \in \mathscr{K}^\reg_n} N^{\cp,\infty}_{t,k}, 
        \quad \text{where}\ 
        N^{\cp,\infty}_{t,k}
        =  \mathcal{S}^\cp_\infty\bigl(\theta_\infty \Phi^\cp_t \bigl(\psi^\cp|_{\mathscr{D}^\cp_k} \comp (\Lambda^\cp_k)^{-1}\bigr)\bigr).
    \end{equation}

\begin{lemma}
    \label{lem:N-infty-expression}
    For all $\cp \in \mathscr{K}^\reg$, we have the identity:
    \begin{equation*}
        N^{\cp,\infty}_{t,k} = \Theta_t[\mathscr{I}_\cp](0) \int_{\mathscr{D}^\cp_k} \bm{\delta}(\mho^\cp(\bm{k})) \psi^\cp(\bm{k}) \diff \bm{k}.
    \end{equation*}
    Particularly, if $\cp \in \mathscr{K}_0$, then $N^{\cp,\infty}_{t,k} = |\psi(k)|^2$.
\end{lemma}
\begin{proof}
    Note that $\kappa_\cp^* \bm{1}_{\mathscr{O}^\cp_{t}}(\bm{\mu},0) = \bm{1}_{\mathscr{O}_t(\mathscr{I}_\cp)}(\bm{\mu})$.
    We therefore have:
    \begin{align*}
        N^{\cp,\infty}_{t,k}
         & =  \int_{(\mathbb{R}^{2d})^{\mathscr{I}_\cp}} \Bigl( \int_{\mathbb{R}^{\mathscr{I}_\cp}} e^{2\pi i \bm{s} \cdot \varpi^\cp(\bm{z})} \diff \bm{s} \Bigr) \Bigl( \int_{\mathbb{R}^{\mathscr{I}_\cp}} \kappa_\cp^* \bm{1}_{\mathscr{O}^\cp_{t}}(\bm{\mu},0) \diff \bm{\mu} \Bigr) \psi^\cp((\Lambda^\cp_k)^{-1}\bm{z})\diff \bm{z} \\
         & =  \int_{\mathscr{O}_t(\mathscr{I}_\cp)} \diff \bm{\mu} \int_{(\mathbb{R}^{2d})^{\mathscr{I}_\cp}} \bm{\delta}(\varpi^\cp(\bm{z})) \psi^\cp((\Lambda^\cp_k)^{-1}\bm{z})\diff \bm{z}.
    \end{align*}
    We conclude using Definition~\ref{def:time-ord} and the fact that the Jacobian of $ \Lambda^\cp_k \comp (\pi^\cp|_\wp)^*|_{\pi^\cp_*\mathscr{D}^\cp_k|_{\wp}}$ is equal to one, which follows from the definition of $\Lambda^\cp_k$ in Lemma~\ref{lem:tree-deco-change-var}.
\end{proof}

\begin{lemma}
    \label{lem:N-infty-hom-induction}
    If $\cp = \otimes^\sigma (q_1,q_2,q_3) \in \mathscr{K}^\reg$ where $\sigma \in \{+,-\}$ and $q_1,q_2,q_3 \in \mathscr{K}^\reg$, then for all $t > 0$ and all $k \in \mathbb{R}^d$, we have the identity
    \begin{equation}
        \label{eq:N-infty-hom-cp-induction}
        N^{\cp,\infty}_{t,k} = \int_0^t \int_{\mathscr{D}_k} \bm{\delta}(\Omega(\bm{k})) \prod_{1 \le j \le 3} N^{\cp_j,\infty}_{s,\bm{k}_j} \diff \bm{k} \diff s.
    \end{equation}
    Consequently, for all $n \ge 1$,  we have the induction formula
    \begin{equation}
        \label{eq:N-infty-hom-n-induction}
        N^{n,\infty}_{t,k} = 2\sum_{n_1+n_2+n_3=n-1} \int_0^t \int_{\mathscr{D}_k} \bm{\delta}(\Omega(\bm{k})) \prod_{1 \le j \le 3} N^{n_j,\infty}_{s,\bm{k}_j} \diff \bm{k} \diff s.
    \end{equation}
\end{lemma}
\begin{proof}
    By Lemma~\ref{lem:N-infty-expression}, we have
    \begin{equation*}
        N^{\cp,\infty}_{t,k}
        = \Theta_t[\mathscr{I}_\cp](0)
        \int_{\mathscr{D}_k} \bm{\delta}(\Omega(\bm{k})) \Bigl( \prod_{1 \le j \le 3} \int_{\mathscr{D}^{\cp_j}_{k_j}} \bm{\delta}(\mho^{\cp_j}(\bm{k}^j)) \diff \bm{k}^j \Bigr) \psi^\cp(\overline{\bm{k}}) \diff \bm{k},
    \end{equation*}
    where $\overline{\bm{k}}$ is the unique element in $\mathscr{D}^\cp$ such that $\overline{\bm{k}}|_{\cp^j} = \bm{k}^j$ for $1 \le j \le 3$.
    We conclude~\eqref{eq:N-infty-hom-cp-induction} with~\eqref{eq:time-ord-fourier-induction} and the identity $\psi^\cp(\overline{\bm{k}}) = \prod_{1 \le j \le 3} \psi^{\cp_j}(\bm{k}^j).$
    The identity~\eqref{eq:N-infty-hom-n-induction} follows from~\eqref{eq:N-infty-hom-cp-induction} and Lemma~\ref{lem:reg-cp-str}.
    In fact, the factor $2$ in~\eqref{eq:N-infty-hom-n-induction} is due to the dichotomy of $\sigma \in \{+,-\}$.
\end{proof}

\subsubsection{Asymptotic estimates}

Now we establish the limit
$\lim_{L \to \infty} N^\cp_{t/\lambda^2,k} = N^{\cp,\infty}_{t,k}.$

\begin{lemma}
    \label{lem:trans-time-int-est}
    If $\cp \in \mathscr{K}^\reg_n$ with $n \ge 1$, then for all $t > 0$, all $U \in C_c^\infty(\mathbb{R}^{\mathfrak{B}^\cp} \times (\mathbb{R}^{2d})^{\mathfrak{B}^\cp})$ any any weighted Sobolev norm $B$ on $(\mathbb{R}^{2d})^{\mathfrak{B}^\cp}$, we have
    \begin{equation*}
        \|\mathcal{F} \Phi^\cp_t U\|_{W^{1,\infty} B} \lesssim (n 2^{n-1} t^{n-1} + t^n) \|\mathcal{F} U\|_{W^{1,\infty} B}.
    \end{equation*}
\end{lemma}
\begin{proof}
    For all $s \in \mathbb{R}$, set $I(t,s) = (-s/2,t-s/2) \cap (s/2,t+s/2)$.
    For all $\bm{s} \in \mathbb{R}^{\mathfrak{B}^\cp}$, we have
    \begin{equation*}
        \supp \kappa^* \bm{1}_{\mathscr{O}^\cp_t}(\cdot,\bm{s}) \subset \prod_{r \in \mathscr{I}_\cp} I(t,\bm{s}_{r}).
    \end{equation*}
    Therefore, for all $\bm{s},\bm{s}' \in \mathbb{R}^{\mathfrak{B}^\cp}$, we have
    \begin{align*}
        \|\kappa^* \bm{1}_{\mathscr{O}^\cp_t}(\cdot,\bm{s}) & - \kappa^* \bm{1}_{\mathscr{O}^\cp_t}(\cdot,\bm{s}')\|_{L^1}
        \lesssim \Bigl| \Bigl\{ \prod_{r \in \mathscr{I}_\cp} I(t,\bm{s}_{r}) \Bigr\} \Delta \Bigl\{ \prod_{r \in \mathscr{I}_\cp} I(t,\bm{s}'_{r})  \Bigr\} \Bigr| \\
        & \lesssim \sum_{r \in \mathscr{I}_\cp} |I_r(t,\bm{s}_{r}) \Delta I_r(t,\bm{s}_{r}')| \prod_{h \in \mathscr{I}_\cp \backslash \{r\}} |I(t,\bm{s}_{r}) \cup I(t,\bm{s}_{r}')|
        \lesssim 2^{n-1} t^{n-1} \sum_{r \in \mathscr{I}_\cp}|\bm{s}_{r} - \bm{s}_{r}'|,
    \end{align*}
    where $\Delta$ denotes the symmetric difference.
    Therefore,
    \begin{equation*}
        \|\mathcal{F} \Phi^\cp_t U\|_{W^{1,\infty} B}
        \lesssim \|\kappa_\cp^* \bm{1}_{\mathscr{O}^\cp_t}\|_{L^1 W^{1,\infty}} \|\mathcal{F} U\|_{W^{1,\infty} B}
        \lesssim  (n 2^{n-1} t^{n-1} + t^n) \|\mathcal{F} U\|_{W^{1,\infty} B}. \qedhere
    \end{equation*}
\end{proof}

\begin{lemma}
    \label{lem:leaf-data-fourier-est}
    Let $\cp \in \mathscr{K}^\reg_n$ with $n \ge 1$.
    Let $m \in \mathbb{N}^d$ and let $\ell \in \mathbb{N}$.
    Then uniformly for all $k \in \mathbb{R}^d$,
    \begin{equation*}
        \bigl\|\mathcal{F}\bigl((\bm{k}_\cp^\gamma\psi^\cp)|_{\mathscr{D}^\cp_k} \comp (\Lambda^\cp_k)^{-1}\bigr)\bigr\|_{L^1_\ell} \lesssim C^n n^{2\ell+|m|} \|\widehat{\psi}\|_{W^{|m|,1}_{\ell}}^{4n+2}.
    \end{equation*}
\end{lemma}
\begin{proof}
    Let $X = (\mathbb{R}^{2d})^{\mathfrak{B}^\cp}$ and let $Y = \mathbb{R}^d \times X$.
    By Lemma~\ref{lem:tree-deco-change-var}, we have a canonical isomorphism $ \overline{\Lambda}^\cp : \mathscr{D}^\tau \to Y $, $\bm{k} \mapsto (k,\Lambda^\cp\bm{k})$, where $k = \bm{k}_{\mathfrak{r}^{\tau^\pm}}$ if $\cp = (\tau^*,\wp)$.
    Let $\overline{\Lambda}_\cp = ((\overline{\Lambda}^\cp)^\dagger)^{-1} = ((\overline{\Lambda}^\cp)^{-1})^{\dagger}$, where $\dagger$ gives the adjoint operator. By Lemma~\ref{lem:tree-deco-change-var}, we have the norm estimate
    $\|\overline{\Lambda}_\cp\| \lesssim \|(\overline{\Lambda}^\cp)^{-1}\| \lesssim n \|\overline{\Lambda}^\cp\| \lesssim n.$
    Let $f = (\bm{k}_\cp^m \psi^\cp) \comp (\overline{\Lambda}^\cp)^{-1}$, i.e., $f(k,\bm{z}) = (\bm{k}_\cp^m \psi^\cp)|_{\mathscr{D}^\cp_k} \bigl( (\Lambda^\cp_k)^{-1} \bm{z} \bigr)$.
    By Lemma~\ref{lem:fourier-isom-L1} below, \eqref{eq:deco-change-var-formula}, and Young's convolution inequality, we have
    \begin{align*}
        \sup_{k \in \mathbb{R}^d} \|\mathcal{F}_Xf(k,\cdot)\|_{L^1_\ell(X)}
        & \lesssim \|\overline{\mathcal{F}}_Y f\|_{L^1_\ell(Y)}
        \lesssim n^\ell \| \mathcal{F}_\cp (\bm{k}_\cp^m\psi^\cp)\|_{L^1_\ell(\mathscr{D}^\cp)}
        \lesssim n^{2\ell} \|\mathcal{F}_{\wp} \pi^\cp_*(\bm{k}_\cp^m\psi^\cp)|_{(\mathbb{R}^d)^{\wp}}\|_{L^1_\ell} \\
        & \lesssim n^{2\ell+|m|} \|\mathcal{F}|\psi|^2\|_{W^{|m|,1}_\ell}^{2n+1}
        \lesssim C^n n^{2\ell+|m|} \|\widehat{\psi}\|_{W^{|m|,1}_{\ell}}^{4n+2}. \qedhere
    \end{align*}
\end{proof}

\begin{lemma}
    \label{lem:fourier-isom-L1}
    Let $X$ and $Y$ be subspaces of finite dimensional Euclidean spaces.
    Let $\mathcal{F}_X$ and $\mathcal{F}_Y$ be Fourier transforms on $X$ and $Y$ respectively.
    If $\Lambda : Y \to X$ is a linear space isomorphism and $\Lambda^\dagger : X \to Y $ is its adjoint operator, then for all $f \in C_c^\infty(Y)$ and $\varphi \in L^\infty_{\loc}(Y)$, we have
    \begin{equation*}
        \|(\varphi \comp \Lambda^\dagger) \mathcal{F}_X(f\comp \Lambda^{-1})\|_{L^1_X}
        = \|\varphi \mathcal{F}_Y f\|_{L^1_Y}.
    \end{equation*}
\end{lemma}
\begin{proof}
    Let $\langle \cdot,\cdot\rangle$ denote the inner product in respective spaces.
    We have
    \begin{align*}
        \mathcal{F}_X(f \comp \Lambda^{-1})(\zeta)
        & = \int_X e^{-2\pi i \langle \zeta, x\rangle_X} f(\Lambda^{-1} x) \diff x
        = |\det(\Lambda)| \int_Y e^{-2\pi i \langle \zeta, \Lambda y\rangle_X} f(y) \diff y \\
        & = |\det(\Lambda)| \int_Y e^{-2\pi i \langle \Lambda^\dagger \zeta, y\rangle_Y} f(y) \diff y
        = |\det(\Lambda)| (\mathcal{F}_Y f)(\Lambda^\dagger \zeta).
    \end{align*}
    Therefore, using the change of variable $\eta = \Lambda^\dagger \zeta$ and $|\det(\Lambda^\dagger)| = |\det(\Lambda)|$, we have
    \begin{align*}
        \|(\varphi \comp \Lambda^\dagger) \mathcal{F}_X(f &\comp \Lambda^{-1})\|_{L^1_X}
        = |\det(\Lambda^\dagger)|^{-1} |\det(\Lambda)|
        \int_Y |\varphi(\eta)\mathcal{F}_Y f(\eta)| \diff \eta
        = \|\varphi \mathcal{F}_Y f\|_{L^1_Y}.
        \qedhere
    \end{align*}
\end{proof}

\begin{proposition}
    \label{prop:E-spec-reg-asym-hom}
    Let $d \ge 3$, $n \ge 1$.
    Let $\supp \psi$ be bounded by $R \ge 1$, let $\cp \in \mathscr{K}^\reg_n$, and let $m \in \mathbb{N}^d$.
    There exists $C >0$ and $\nu > 0$ such that uniformly for all $(t,k) \in [0,1] \times \mathbb{R}^d$, we have
    \begin{equation*}
        |k^m(N^\cp_{t/\lambda^2,k} - N^{\cp,\infty}_{t,k})|
        \lesssim C^n n^{d+2} (nR)^{4nd} (n 2^{n-1} t^{n-1} + t^n) L^{-\nu} \|\widehat{\psi}\|_{L^1_1}^{4n+2}.
    \end{equation*}
    Consequently, there holds the limit:
    \begin{equation*}
        \lim_{L \to \infty} \sup_{(t,k) \in [0,1] \times \mathbb{R}^d} |k^m(N^\cp_{t/\lambda^2,k} - N^{\cp,\infty}_{t,k})| = 0.
    \end{equation*}
\end{proposition}
\begin{proof}
    By~\eqref{eq:N-S-op-id-hom}, \eqref{eq:def-N-infty-hom}, Proposition~\ref{prop:int-approx}, and the fact that $\supp \psi^\cp$ is bounded by $nR$, we have
    \begin{equation*}
        |k^m(N^\cp_{t/\lambda^2,k} - N^{\cp,\infty}_{t,k})|
        \lesssim C^n n^d (nR)^{4nd} L^{-\nu} \bigl\|\mathcal{F} \Phi^\cp_t \bigl( (\bm{k}_\cp^m \psi^\cp)|_{\mathscr{D}^\cp_k} \comp (\Lambda^\cp_k)^{-1} \bigr) \bigr\|_{W^{1,\infty} L^1_1}.
    \end{equation*}
    We conclude with Lemmas~\ref{lem:trans-time-int-est} and~\ref{lem:leaf-data-fourier-est}.
\end{proof}

\begin{proposition}
    \label{prop:N-infty-hom-est}
    Let $\cp \in \mathscr{K}^\reg_n$ with $n \ge 1$, then for all $m,\ell \in \mathbb{N}^d$ and uniformly for all $(t,k) \in [0,1] \times \mathbb{R}^d$, we have the estimate for some universal constant $C > 0$,
    \begin{equation*}
        |k^m \partial_k^\ell N^{\cp,\infty}_{t,k}| \lesssim C^n t^n n^{|m|} \|\psi\|^{4n+2}_{W^{|\ell|+d,\infty}_{|m|+d/2}}.
    \end{equation*}
\end{proposition}
\begin{proof}
    Let $(\chi_z)_{z \in \mathbb{Z}^d}$ be a partition of unity of $\mathbb{R}^d$ such that $\chi_z(\cdot) = \chi(\cdot+z)$ for some $\chi \in C_c^\infty(\mathbb{R}^d)$ and for all $z \in \mathbb{Z}^d$.
    For all $\gamma \in (\mathbb{Z}^d)^{\wp}$ and $\bm{k} \in \mathscr{D}^\cp$, set $\chi_\gamma(\bm{k}) = \prod_{\mathfrak{p} \in \wp} \chi_{\gamma_{\mathfrak{p}}}(\bm{k}^\flat_{\mathfrak{p}})$.
    Then $\sum_{\gamma \in (\mathbb{Z}^d)^\wp} \chi_\gamma = 1$.
    Let $\varrho_k$ be the translation operator on $C^\infty(\mathscr{D}^\cp)$ defined by $\varrho_k f(\cdot) = f(\cdot + k)$.
    Denote $\mathcal{D}_k^\ell = \varrho_{k}^{-1} \partial_k^\ell \varrho_k$.
    By Lemma~\ref{lem:N-infty-expression}, we have the decomposition
    \begin{equation*}
        \partial_k^\ell N^{\cp,\infty}_{t,k}
        = \sum_{\gamma \in (\mathbb{Z}^d)^{\wp}} I_\gamma(t,k),
        \quad\text{where}\ 
        I_\gamma(t,k) = \Theta_t[\mathscr{I}_\cp](0) \int_{\mathscr{D}^\cp_k} \bm{\delta}(\mho^\cp(\bm{k})) (\chi_\gamma \mathcal{D}_k^\ell \psi^\cp)(\bm{k}) \diff \bm{k}.
    \end{equation*}
    For each term in the summation, we use Proposition~\ref{prop:gauss-int-upper-bound} and~\eqref{eq:def-N-infty-hom} for the estimate
    \begin{equation*}
        |k^m I_\gamma(t,k)|
        \lesssim C^n t^n \bigl\|\mathcal{F}_{\wp} \bigl\{   \pi^\cp_*(\chi_\gamma \bm{k}_\cp^m \mathcal{D}_k^\ell \psi^\cp)|_{(\mathbb{R}^d)^{\wp}} \bigr\}\bigr\|_{L^1}
        \lesssim \tilde{C}^n t^n n^{|m|} \prod_{\mathfrak{p} \in \wp} \|\mathcal{F}(\chi_{\gamma_{\mathfrak{p}}}|\psi|^2)\|_{W_\ell^{|m|,1}}.
    \end{equation*}
    By Cauchy's inequality and Plancherel's theorem, we have
    \begin{align*}
        \|\mathcal{F}(\chi_{\gamma_{\mathfrak{p}}}|\psi|^2)\|_{W^{|m|,1}_\ell}
        \lesssim \|\mathcal{F}(\chi_{\gamma_{\mathfrak{p}}}|\psi|^2)\|_{W^{|m|,2}_{|\ell|+d}}
        \lesssim \|\chi_{\gamma_{\mathfrak{p}}}|\psi|^2\|_{W^{|\ell|+d,2}_{|m|}}
        \lesssim \|\chi_{\gamma_{\mathfrak{p}}}\|_{W^{|\ell|+d,2}_{-d}} \|\psi\|_{W^{d,\infty}_{|m|+d/2}}^2.
    \end{align*}
    Summing up in $\gamma$, we conclude with the estimate:
    \begin{equation*}
        |k^m \partial_k^\ell N^{q,\infty}_{t,k}|
        \lesssim \Bigl(\sum_{z \in \mathbb{Z}^d} \|\chi_{z}\|_{W^{|\ell|+d,2}_{-d}} \Bigr)^{2n+1} \|\psi\|_{W^{|\ell|+d,\infty}_{|m|+d/2}}^{4n+2}
        \lesssim \Bigl( \sum_{z \in \mathbb{Z}^d} \langle \bm{z} \rangle^{-d} \Bigr)^{2n+1} \|\psi\|_{W^{|\ell|+d,\infty}_{|m|+d/2}}^{4n+2}.\qedhere
    \end{equation*}
\end{proof}

\subsection{Wave kinetic equation}

\label{sec:WK-hom}

Now we derive the wave kinetic equation WK and prove the convergence results in Theorem~\ref{thm:hom-inhom} when $\beta=\infty$.
By Proposition~\ref{prop:N-infty-hom-est} and Lemma~\ref{lem:tree-regcp-counting}, if $t \in [0,T]$ and $T$ is sufficiently small, then for all $k \in \mathbb{R}^d$, we have a absolutely convergent sum
\begin{equation*}
    N_k(t) = \sum_{n \ge 0} N^{n,\infty}_{t,k}.
\end{equation*}
We will prove that $(N_k)_{k \in \mathbb{R}^d}$ solves the wave kinetic equation WK with initial data
    $N_k(0) = |\psi(k)|^2,$
by establishing the following Duhamel's formula~\eqref{eq:N-duhamel-hom}.
Once it is proved, the estimate~\eqref{eq:A-W-convergence} follows as a consequence of Proposition~\ref{prop:N-infty-hom-est} and~\eqref{eq:cauchy-diagram-hom}.

\begin{proposition}
    \label{prop:N-duhamel-hom}
    If $T$ is sufficiently small, then for all $ (t,k) \in [0,T] \times \mathbb{R}^d$, we have
    \begin{equation}
        \label{eq:N-duhamel-hom}
        N_k(t) = |\psi(k)|^2 + 2 \int_0^t \int_{\mathscr{D}_k} \bm{\delta}(\Omega(\bm{k})) \prod_{1\le j \le 3} N_{k_j}(s) \diff \bm{k} \diff s.
    \end{equation}
\end{proposition}
\begin{proof}
    The formula \eqref{eq:N-infty-hom-n-induction} implies that
    \begin{equation*}
        \sum_{0 \le \ell \le n} N^{\ell,\infty}_{t,k} = |\psi(k)|^2 + 2 \sum_{1 \le \ell \le n} \sum_{n_1+n_2+n_3=\ell-1}\int_0^t \int_{\mathscr{D}_k} \bm{\delta}(\Omega(\bm{k})) \prod_{1 \le j \le 3} N^{n_j,\infty}_{s,k_j} \diff \bm{k} \diff s.
    \end{equation*}
    By Proposition~\ref{prop:N-infty-hom-est}, we have the absolute convergence of the series, which allows us to pass to the limit $n \to \infty$ and obtain the formula~\eqref{eq:N-duhamel-hom}.
\end{proof}

    Next, let $\chi \in \mathscr{C}_c^\infty(\mathbb{R}^d)$.
    By Proposition~\ref{prop:E-spec-N-id}, for all $n \in \mathbb{N}$, all $t \in [0,T]$, and all $x \in \mathbb{R}^d$, write
    \begin{equation}
        \label{eq:E-spec-decomp-hom}
        \int \chi(\xi) \mathbb{E}\mathcal{W}\Bigl[u^{2n+1}\Bigl(\frac{t}{\lambda^2}\Bigr)\Bigr](x,\xi) \diff \xi
        = \frac{1}{L^d} \sum_{k \in \mathbb{Z}^d_L} \chi(k) N^n_{t/\lambda^2,k}
        = \frac{1}{L^d} \sum_{k \in \mathbb{Z}^d_L} \chi(k) N^{n,\infty}_{t,k} + \mathcal{R}_L(t)
    \end{equation}
    where the remainder $\mathcal{R}_L(t)$ has the exact expression
    \begin{equation*}
        \mathcal{R}_L(t) = \frac{1}{L^d} \sum_{k \in \mathbb{Z}^d_L} \chi(k) \Bigl( \sum_{\cp \in \mathscr{K}_n \backslash \mathscr{K}^\reg_n} N^\cp_{t/\lambda^2,k} + \frac{1}{L^d} \sum_{\cp \in \mathscr{K}^\reg_n}(N^\cp_{t/\lambda^2,k} - N^{\cp,\infty}_{t,k}) \Bigr).
    \end{equation*}
    By Propositions~\ref{prop:E-spec-general-upper-bound-hom} and~\ref{prop:E-spec-reg-asym-hom}, we have $\lim_{L \to \infty} \mathcal{R}_L(t) = 0$, uniformly for all $t \in [0,T]$.
    Furthermore, using Proposition~\ref{prop:N-infty-hom-est}, we can pass to the limit $L \to \infty$ for the Riemann sum on the right hand side of~\eqref{eq:E-spec-decomp-hom} to converge to an integral, which yields
    \begin{equation*}
        \lim_{L \to \infty} \int \chi(\xi) \mathbb{E}\mathcal{W}\Bigl[u^{2n+1}\Bigl(\frac{t}{\lambda^2}\Bigr)\Bigr](x,\xi) \diff \xi = \int \chi(k) N^{n,\infty}_{t,k} \diff k.
    \end{equation*}
    Finally, we conclude~\eqref{eq:E-spec-wigner-lim} by summing up in~$n$.

\section{Effective dynamics: inhomogeneous and semi-homogeneous setting}

\label{sec:inhom}

In this section we prove Theorems~\ref{thm:hom-inhom} and~\ref{thm:inhom-2nd} in the inhomogeneous and semi-homogeneous setting.
Recall that we have $\lambda^{-2} = L^\alpha$ and $\epsilon^{-1} = L^\beta$ where $0 < \alpha < 2/(1+4/d^*)$ and $ \alpha \le \beta < \infty$.

\subsection{General couples}

In the inhomogeneous setting, the identity~\eqref{eq:N-I-id-hom} needs a fine tuning via an operator that describes the transport of wavepackets.
Let $\cp = (\tau^*,\wp) \in \mathscr{K}$.
For all $x,\zeta \in \mathbb{R}^d$ we define $\mathcal{T}^\cp_{x,\zeta} \phi = \mathcal{F}^{-1}_{\eta\to x} \widehat{\mathcal{T}}^\cp_{\eta,\zeta} \phi$ as a function on $\mathbb{R}^{\mathfrak{B}^\cp} \times \mathscr{D}^\cp$ by setting
\begin{equation}
    \label{eq:T-op-def}
    \widehat{\mathcal{T}}^{\cp}_{\eta,\zeta} \phi(\bm{t},\bm{k})
    = \iint_{\mathscr{C}^\cp_\eta \times \mathscr{D}^{\cp}_{\zeta}}
    e^{2\pi i \epsilon \lambda^{-2} \bm{t} \cdot (2\Omega^{\tau^*}(\bm{k};\bm{\zeta}+\iota\bm{\eta}/2) + \epsilon \Omega^{\tau^*}(\bm{\zeta}+\iota\bm{\eta}/2))}
    \widehat{\mathcal{W}}^\cp[\phi](\bm{k},\bm{\eta},\bm{\zeta}) \diff \bm{\zeta} \diff \bm{\eta}.
    \index{Operations and transforms!Phase space transforms!Transport operator on couples $\mathcal{T}^\cp$}
\end{equation}
\begin{proposition}
    \label{prop:E-upper-bound-inhom}    
    Then for all $t > 0$ and all $k \in \mathbb{R}^d$, we have the identity
    \begin{equation}
        \label{eq:E-spec-I-T-id}
        \mathcal{E}^\cp_{t,k}(x,\zeta) = \varsigma_\cp \mathcal{I}^\cp_{t,k} \mathcal{T}^{\cp}_{x,\zeta} \phi.
    \end{equation}
    Suppose that $q \in \mathscr{K}_n$ with $n \ge 1$.
    Let $\supp \widehat{\phi} $ be bounded by $R \ge 1$ and let $m \in \mathbb{N}^d$.
    Then uniformly in $(t,k) \in [0,1] \times \mathbb{R}^d$, we have the estimates
    \begin{align}
        \label{eq:E-upper-bound-inhom}
        \|k^m \mathcal{E}^\cp_{t/\lambda^2,k}\|_{L^\infty} & \lesssim C^n t^{2n} n^{|m|} (nR)^{4nd}\ln(1+nR^2L^\alpha) L^{-\nu\ind(\cp)} \|\widehat{\phi}\|_{W^{d,1}L^\infty_{|m|}}^{4n+2} \\
        \label{eq:N-upper-bound-inhom}
        \|k^m \mathcal{N}^\cp_{t/\lambda^2,k}\|_{L^\infty} & \lesssim C^n t^{2n} n^{|m|} (nR)^{4nd}\ln(1+nR^2L^\alpha) L^{-\nu\ind(\cp)} \|\widehat{\phi}\|_{L^1L^\infty_{|m|}}^{4n+2}.
    \end{align}
    Particularly, if $q \notin \mathscr{K}^\reg$, then we have
    \begin{equation*}
        \lim_{L \to \infty} \sup_{(t,k) \in [0,1] \times \mathbb{R}^d} \|k^m \mathcal{E}^\cp_{t/\lambda^2,k}\|_{L^\infty} 
        = \lim_{L \to \infty} \sup_{(t,k) \in [0,1] \times \mathbb{R}^d} \|k^m \mathcal{N}^\cp_{t/\lambda^2,k}\|_{L^\infty} = 0.
    \end{equation*}
\end{proposition}
\begin{proof}
    Set $\bm{\xi} = \bm{\zeta}+\iota\bm{\eta}/2$.
    Then~\eqref{eq:E-spec-I-T-id} is a direct consequence of the quadratic expansion:
    \begin{equation*}
        \Omega^{\tau^*}(\bm{k}+\epsilon\bm{\xi})
        = \Omega^\cp(\bm{k}) + 2 \epsilon \Omega^{\tau^*}(\bm{k};\bm{\xi}) + \epsilon^2 \Omega^{\tau^*}(\bm{\xi}).
    \end{equation*}
    The proof for the estimate~\eqref{eq:E-upper-bound-inhom} is similar to that of Proposition~\ref{prop:E-spec-general-upper-bound-hom}.
    In fact, by Lemma~\ref{lem:I-op-est},
    \begin{equation*}
        |k^m \mathcal{E}^\cp_{t/\lambda^2,k}(x,\zeta)|
        \lesssim C^n t^{2n} (nR)^{4nd}\ln(1+nR^2L^\alpha) L^{-\nu\ind(\cp)} \|\bm{k}_\cp^m \mathcal{T}^\cp_{x,\zeta}\phi\|_{L^\infty L^\infty}.
    \end{equation*}
    We conclude with Lemma~\ref{lem:T-op-upper-bound} below and the Sobolev embedding theorem.
    The identity~\eqref{eq:N-upper-bound-inhom} follows by Lemma~\ref{lem:T-op-upper-bound} and integrating the above estimate directly in $\zeta$.
\end{proof}

\begin{lemma}
    \label{lem:T-op-upper-bound}
    If $\cp \in \mathscr{K}_n$ with $n \ge 0$, then for all $m,\ell \in \mathbb{N}^d$, we have
    \begin{equation*}
        \sup_{x,\zeta \in \mathbb{R}^d} \int \| \bm{k}_\cp^m \partial_\zeta^\ell \mathcal{T}^\cp_{x,\zeta} \phi \|_{L^\infty} \diff \zeta \lesssim
        C^n n^{|m|} \prod_{\mathfrak{p} \in \wp} \|\widehat{\phi}\|_{W^{|\ell|,1}L^\infty_{|m|}}^{4n+2}.
    \end{equation*}
\end{lemma}
\begin{proof}
    The proof is similar to that of Proposition~\ref{prop:N-infty-hom-est}:
    \begin{align*}
        \iint | \bm{k}_\cp^m \partial_\zeta^\ell \widehat{\mathcal{T}}^\cp_{\eta,\zeta} \phi (\bm{t},\bm{k}) | \diff \eta\diff \zeta 
        & \le \int \Bigl( \iint_{\mathscr{C}^\cp_\eta \times \mathscr{D}^{\cp}_\zeta}
        \bigl| \bm{k}_\cp^m  \mathcal{D}_\zeta^\ell \widehat{\mathcal{V}}^\cp[\phi](\bm{k},\bm{\eta},\bm{\zeta}) \bigr| \diff \bm{\zeta} \diff \bm{\eta} \Bigr) \diff \eta \diff \zeta \\
        & \lesssim C^n n^{|m|} \|\widehat{\mathcal{V}}[\phi]\|_{L^\infty_{|m|}L^1W^{|\ell|,1}}^{2n+1}.
        \lesssim \tilde{C}^n n^{|m|} \prod_{\mathfrak{p} \in \wp} \|\widehat{\phi}\|_{W^{|\ell|,1}L^\infty_{|m|}}^{4n+2}.
        \qedhere
    \end{align*}
\end{proof}

\subsection{Regular couples}

Now we study the asymptotic behaviors of the energy spectra $\mathcal{E}^\cp_{t,k}$ and $\mathcal{N}^\cp_{t,k}$, defined in \S\ref{sec:energy-spec-non-periodic}, when $\cp \in \mathscr{K}^\reg$.
The analysis is similar to that of the homogeneous setting \S\ref{sec:reg-cp-hom}, but we need to take into account of the estimates for the operator $\mathcal{T}^\cp_{x,\zeta}$.

\subsubsection{Algebraic structures}

Let us first give the limit without proof:
\begin{equation*}
    \mathcal{T}^{\cp,\infty}_{x,\zeta} \phi = \lim_{L \to \infty} \mathcal{T}^\cp_{x,\zeta} \phi,
\end{equation*}
which clearly depends on the scaling parameters $\alpha$ and $\beta$.
Precisely, if $\cp = (\tau^*,\wp) \in \mathscr{K}$, then for all $x,\zeta \in \mathbb{R}^d$, all $\bm{t} \in \mathbb{R}^{\mathfrak{B}^\cp}$, and all $\bm{k} \in \mathscr{D}^\cp$, we define $\mathcal{T}^{\cp,\infty}_{x,\zeta} = \mathcal{F}^{-1}_{\eta\to x} \widehat{\mathcal{T}}^{\cp,\infty}_{\eta,\zeta}$ by setting
\begin{equation}
    \label{eq:T-op-infty-def}
    \widehat{\mathcal{T}}^{\cp,\infty}_{\eta,\zeta} \phi(\bm{t},\bm{k})
    =  \iint_{\mathscr{C}^\cp_\eta \times \mathscr{D}^{\cp}_{\zeta}} e^{\bm{1}_{\alpha=\beta} 4\pi i \bm{t} \cdot \Omega^{\tau^*}(\bm{k};\bm{\zeta}+\iota\bm{\eta}/2)}
    \widehat{\mathcal{V}}^\cp(\bm{k},\bm{\eta},\bm{\zeta}) \diff \bm{\eta} \diff \bm{\zeta}.
\end{equation}
We shall give an explicit expression of $\kappa_\cp^* \mathcal{T}^{\cp,\infty}_{\eta,\zeta} \phi(\bm{\mu},\bm{s},\bm{k})$.
Note that, by Lemma~\ref{lem:conju-node-char}, if $\mathfrak{m} \sim \mathfrak{n}$, then $\Omega^{\tau^*}_{\mathfrak{m}}(\bm{k},\iota\bm{\eta}) = \Omega^{\tau^*}_{\mathfrak{n}}(\bm{k},\iota\bm{\eta})$ for all $\bm{k} \in \mathscr{D}^\cp$ and all $\bm{\eta} \in \mathscr{C}^\cp$.
Particularly, if $\cp \in \mathscr{K}^\reg$, then we can define $\Upsilon^\cp = (\Upsilon^\cp_r)_{r \in \mathscr{I}_\cp} : \mathscr{D}^\cp \times \mathscr{C}^\cp \to \mathbb{R}^{\mathscr{I}_\cp}$ by setting
$\Upsilon^\cp_r(\cdot,\cdot) = \Omega^{\tau^*}_{\mathfrak{b}^\pm}(\cdot,\iota \cdot)$,
\index{Functions and random variables!Resonance functions!Transport factor $\Upsilon^\cp_r$}
where $\mathfrak{B}^r = \{\mathfrak{b}^+,\mathfrak{b}^-\}$.
Note that by~\eqref{eq:T-op-infty-def} and the definitions of $\mho^\cp$ and $\Upsilon^\cp$, if $\bm{t} = \kappa_\cp(\bm{\mu},\bm{s})$, then
\begin{align*}
    \bm{t} \cdot \Omega^{\tau^*}(\bm{k},\bm{\zeta}+\iota\bm{\zeta}/2)
    & = \bm{t} \cdot (\Omega^\cp(\bm{k},\bm{\zeta}) + \Omega^{\tau^*}(\bm{k},\iota\bm{\eta})/2) = \bm{s} \cdot \mho^\cp(\bm{k},\bm{\zeta}) /2 + \bm{\mu} \cdot \Upsilon^\cp(\bm{k},\bm{\eta})/2,
\end{align*}
which yields the identity
\begin{equation}
    \label{eq:T-op-phase-decomp-id}
    \kappa_\cp^* \widehat{\mathcal{T}}^{\cp,\infty}_{\eta,\zeta} \phi(\bm{\mu},\bm{s},\bm{k})
    = \iint_{\mathscr{C}^\cp_\eta \times \mathscr{D}^\cp_\zeta} e^{\bm{1}_{\alpha=\beta}2\pi i (\bm{s} \cdot \mho^\cp(\bm{k},\bm{\zeta}) + \bm{\mu} \cdot \Upsilon^\cp(\bm{k},\bm{\eta}))} \widehat{\mathcal{V}}^\cp(\bm{k},\bm{\eta},\bm{\zeta}) \diff \bm{\eta} \diff \bm{\zeta}
\end{equation}
Using Lemma~\ref{lem:conju-node-char}, we can give an explicit expression to $\Upsilon^\cp$, which reveals the transport structure of the phase factor $\bm{\mu} \cdot \Upsilon^\cp(\bm{k},\bm{\eta})$ in~\eqref{eq:T-op-phase-decomp-id}.

\begin{lemma}
    \label{lem:T-op-transport-id}
    Let $\cp \in \mathscr{K}^\reg$.
    For all $\bm{\mu} \in \mathbb{R}^{\mathscr{I}_\cp}$, all $\bm{k} \in \mathscr{D}^\cp$, and all $\bm{\eta} \in \mathscr{C}^\cp$, we have
    \begin{equation}
        \label{eq:transport-str}
        \bm{\mu} \cdot \Upsilon^\cp(\bm{k},\bm{\eta})
        = \sum_{\mathfrak{p} \in \wp} Z^\cp_{\mathfrak{p}}(\bm{\mu},\bm{k}) \cdot \bm{\eta}^\flat_{\mathfrak{p}},
        \quad \text{where}\
        Z^\cp_{\mathfrak{p}}(\bm{\mu},\bm{k}) = \frac{1}{2} \sum_{r \in \mathscr{I}_\cp} \bm{\mu}_{r} \sum_{\mathfrak{c} \in \wp(r)} \bm{1}_{\mathfrak{p} \preceq \mathfrak{c}} (\bm{k}^\flat_{\mathfrak{B}^r} - \bm{k}^\flat_{\mathfrak{c}}).
    \end{equation}
    In~\eqref{eq:transport-str}, the strict partial order on $\mathfrak{C}^\cp_2$ is defined as in Remark~\ref{rmk:I-q-ordering} and $\wp(r)$ is the pairing that defines~$r$.
    Consequently, for all $x,\zeta \in \mathbb{R}^d$ we have the identities
    \begin{align}
        \label{eq:T-op-transport-id}
        \kappa_\cp^* \mathcal{T}^{\cp,\infty}_{x,\zeta} \phi(\bm{\mu},0,\bm{k})
        = \int_{\mathscr{D}^\cp_\zeta} & \prod_{\mathfrak{p}\in\wp} \mathcal{V}[\phi](\bm{k}^\flat_{\mathfrak{p}},x + \bm{1}_{\alpha=\beta} Z^\cp_{\mathfrak{p}}(\bm{\mu},\bm{k}),\bm{\zeta}^\flat_{\mathfrak{p}}) \diff \bm{\zeta}; \\
        \label{eq:T-op-transport-id-int}
        \int \kappa_\cp^* \mathcal{T}^{\cp,\infty}_{x,\zeta} \phi(\bm{\mu},0,\bm{k}) \diff \zeta
        = & \prod_{\mathfrak{p}\in\wp} |\phi(x + \bm{1}_{\alpha=\beta} Z^\cp_{\mathfrak{p}}(\bm{\mu},\bm{k}),\bm{k}^\flat_{\mathfrak{p}})|^2.
    \end{align}
\end{lemma}
\begin{proof}
    In fact, by Lemmas~\ref{lem:conju-node-char} and~\ref{lem:deco-formula}, we have $\bm{\eta}^\flat_{\mathfrak{B}^r} = \sum_{\mathfrak{c} \in \wp(r)} \bm{\eta}^\flat_{\mathfrak{c}} $ and
    \begin{align*}
        2 \Upsilon^\cp_r(\bm{k},\bm{\eta})
        & = \bm{k}^\flat_{\mathfrak{B}^r} \cdot \bm{\eta}^\flat_{\mathfrak{B}^r} - \sum_{\mathfrak{c} \in \wp(r)} \bm{k}^\flat_{\mathfrak{c}} \cdot \bm{\eta}^\flat_{\mathfrak{c}}
        = \sum_{\mathfrak{c} \in \wp(r)} (\bm{k}^\flat_{\mathfrak{B}^r} - \bm{k}^\flat_{\mathfrak{c}}) \cdot \bm{\eta}^\flat_{\mathfrak{c}}\\
        & = \sum_{\mathfrak{c} \in \wp(r)} (\bm{k}^\flat_{\mathfrak{B}^r} - \bm{k}^\flat_{\mathfrak{c}}) \cdot \Bigl( \sum_{\mathfrak{p} \in \wp} \bm{1}_{\mathfrak{p}\preceq \mathfrak{c}} \bm{\eta}^\flat_{\mathfrak{p}} \Bigr)
        = \sum_{\mathfrak{p} \in \wp} \Bigl( \sum_{\mathfrak{c} \in \wp(r)} \bm{1}_{\mathfrak{p} \preceq \mathfrak{c}} (\bm{k}^\flat_{\mathfrak{B}^r} - \bm{k}^\flat_{\mathfrak{c}}) \Bigr) \cdot \bm{\eta}^\flat_{\mathfrak{p}}.
    \end{align*}
    We obtain~\eqref{eq:transport-str} by summing up in $r \in \mathscr{I}_\cp$.
    To prove~\eqref{eq:T-op-transport-id}, we use~\eqref{eq:transport-str} and Lemma~\ref{lem:deco-formula}.
    Taking the inverse Fourier transform on both sides of~\eqref{eq:T-op-phase-decomp-id}, we have
    \begin{align*}
        \kappa_\cp^* \mathcal{T}^{\cp,\infty}_{x,\zeta} \phi(\bm{\mu},0,\bm{k})
        & = \int e^{2\pi i x \cdot \eta} \Bigl(
        \iint_{\mathscr{C}^\cp_\eta \times \mathscr{D}^\cp_\zeta}
        \prod_{\mathfrak{p} \in \wp} e^{\bm{1}_{\alpha=\beta} 2\pi i Z^\cp_{\mathfrak{p}}(\bm{\mu},\bm{k}) \cdot \bm{\eta}^\flat_{\mathfrak{p}}} \widehat{\mathcal{V}}[\phi](\bm{k}^\flat_{\mathfrak{p}},\bm{\eta}^\flat_{\mathfrak{p}},\bm{\zeta}^\flat_{\mathfrak{p}}) \diff \bm{\eta} \diff \bm{\zeta} \Bigr) \diff \eta \\
        & = \int_{\mathscr{D}^\cp_\eta} \Bigl( \prod_{\mathfrak{p} \in \wp} \int e^{2\pi i (x + \bm{1}_{\alpha=\beta} Z^\cp_{\mathfrak{p}}(\bm{\mu},\bm{k})) \cdot \bm{\zeta}^\flat_{\mathfrak{p}}} \widehat{\mathcal{V}}[\phi](\bm{k}^\flat_{\mathfrak{p}},\bm{\eta}^\flat_{\mathfrak{p}},\bm{\zeta}^\flat_{\mathfrak{p}}) \diff \bm{\zeta}^\flat_{\mathfrak{p}} \Bigr) \diff \bm{\zeta},
    \end{align*}
    which proves~\eqref{eq:T-op-transport-id}.
    The identity~\eqref{eq:T-op-transport-id-int} follows by integrating in~$\zeta$.
\end{proof}

By Lemma~\ref{lem:I-S-op-id}, for all $\cp \in \mathscr{K}^\reg$, all $t > 0$ and all $k \in \mathbb{R}^d$, we have
\begin{equation}
    \label{eq:def-E-reg-inhom}
    \mathcal{E}^\cp_{t,k}(x,\zeta) = \mathcal{S}^\cp_L\bigl(\theta_L \Phi^\cp_t \bigl(\mathcal{T}^\cp_{x,\zeta}\phi|_{\mathscr{D}^\cp_k} \comp (\Lambda^\cp_k)^{-1}\bigr)\bigr).
\end{equation}
It suggests that, for all $n \ge 0$, we define $\mathcal{E}^{n,\infty}_{t,k} = \mathcal{E}^{n,\infty}_{t,k}(x,\zeta)$ by setting
\begin{equation}
    \label{eq:def-E-infty-inhom}
    \mathcal{E}^{n,\infty}_{t,k} = \sum_{\cp \in \mathscr{K}^\reg_n} \mathcal{E}^{\cp,\infty}_{t,k}, 
    \quad \text{where}\
    \mathcal{E}^{\cp,\infty}_{t,k}(x,\zeta)
    = \mathcal{S}^\cp_\infty\bigl(\theta_\infty \Phi^\cp_t \bigl(\mathcal{T}^{\cp,\infty}_{x,\zeta}\phi|_{\mathscr{D}^\cp_k} \comp (\Lambda^\cp_k)^{-1}\bigr)\bigr).
\end{equation}
Then we define $\mathcal{N}^{n,\infty}_{t,k} = \mathcal{N}^{n,\infty}_{t,k}(x)$ by integrating $\mathcal{E}^{n,\infty}_{t,k}$ in $\zeta$:
\begin{equation}
    \label{eq:def-N-infty-inhom}
    \mathcal{N}^{n,\infty}_{t,k} = \sum_{\cp \in \mathscr{K}^\reg_n} \mathcal{N}^{\cp,\infty}_{t,k},
    \quad \text{where}\
    \mathcal{N}^{\cp,\infty}_{t,k}(x) = \int \mathcal{E}^{\cp,\infty}_{t,k}(x,\zeta) \diff \zeta.
\end{equation}
By a direct computation using~\eqref{eq:def-E-infty-inhom}, \eqref{eq:def-N-infty-inhom}, and Lemma~\ref{lem:T-op-transport-id}, we have expressions for $\mathcal{E}^{\cp,\infty}_{t,k}$ and $\mathcal{N}^{\cp,\infty}_{t,k}$ which are similar to that of $N^{\cp,\infty}_{t,k}$ given in Lemma~\ref{lem:N-infty-expression}, particularly when $\alpha < \beta$, in which case $\kappa_\cp^* \mathcal{T}^{\cp,\infty}_{x,\zeta} \phi(\bm{\mu},0,\bm{k})$ is independent of the variable~$\bm{\mu}$.

\begin{lemma}
    \label{lem:E-N-infty-expression-compact}
    Let $\cp \in \mathscr{K}^\reg$.
    For all $t > 0$, $x,\zeta \in \mathbb{R}^d$, and all $\bm{k} \in \mathscr{D}^\cp$, let
    \begin{equation}
        \label{eq:def-U-V-op}
        \mathcal{U}^\cp_{t,x,\zeta} \phi(\bm{k}) = \int_{\mathscr{O}_t(\mathscr{I}_\cp)} \kappa_\cp^* \mathcal{T}^{\cp,\infty}_{x,\zeta} \phi(\bm{\mu},0,\bm{k}) \diff \bm{\mu}.
    \end{equation}
    Then for all $k \in \mathbb{R}^d$, we have the identities:
    \begin{equation*}
        \mathcal{E}^{\cp,\infty}_{t,k}(x,\zeta)
        = \int_{\mathscr{D}^\cp_k} \bm{\delta}(\mho^\cp(\bm{k})) \mathcal{U}^\cp_{t,x,\zeta} \phi(\bm{k}) \diff \bm{k}.
    \end{equation*}
\end{lemma}

Consequently, we have inductive formulas similar to the ones stated in Lemma~\ref{lem:N-infty-hom-induction}.

\begin{lemma}
    \label{lem:E-N-infty-induction}
    If $\cp = \otimes^\sigma(q_1,q_2,q_3) \in \mathscr{K}^\reg$ where $\sigma \in \{+,-\}$ and $q_1,q_2,q_3 \in \mathscr{K}^\reg$ , then for all $t > 0$ and all $k,x,\zeta \in \mathbb{R}^d$, we have the identity
    \begin{equation*}
        \mathcal{E}^{\cp,\infty}_{t,k}(x,\zeta)
        = \int_0^t \int_{\mathscr{D}_k} \bm{\delta}(\Omega(\bm{k})) \Bigl( \int_{\mathscr{D}_\zeta} \prod_{1 \le j \le 3} \mathcal{E}^{\cp_j,\infty}_{s,k_j}(x+s (k-k_j) \bm{1}_{\alpha=\beta},\bm{\zeta}_j) \diff\bm{\zeta} \Bigr) \diff \bm{k} \diff s.
    \end{equation*}
    Consequently, for all $n \ge 1$, there holds the inductive formula:
    \begin{equation*}
        \mathcal{E}^{n,\infty}_{t,k}(x,\zeta)
        = 2 \sum_{n_1+n_2+n_3=n-1} \int_0^t \int_{\mathscr{D}_k} \bm{\delta}(\Omega(\bm{k})) \Bigl( \int_{\mathscr{D}_\zeta} \prod_{1 \le j \le 3} \mathcal{E}^{n_j,\infty}_{s,k_j}(x+s (k-k_j) \bm{1}_{\alpha=\beta},\bm{\zeta}_j) \diff\bm{\zeta} \Bigr) \diff \bm{k} \diff s.
    \end{equation*}
    Integrating in $\zeta$ then yields the identity
    \begin{equation*}
        \mathcal{N}^{\cp,\infty}_{t,k}(x)
        = \int_0^t \int_{\mathscr{D}_k} \bm{\delta}(\Omega(\bm{k})) \prod_{1 \le j \le 3} \mathcal{N}^{\cp_j,\infty}_{s,k_j}(x+s (k-k_j) \bm{1}_{\alpha=\beta}) \diff \bm{k} \diff s,
    \end{equation*}
    as well as the inductive formula:
    \begin{equation*}
        \mathcal{N}^{n,\infty}_{t,k}(x)
        = 2 \sum_{n_1+n_2+n_3=n-1}
        \int_0^t \int_{\mathscr{D}_k} \bm{\delta}(\Omega(\bm{k})) \prod_{1 \le j \le 3} \mathcal{N}^{n_j,\infty}_{s,k_j}(x+s (k-k_j) \bm{1}_{\alpha=\beta}) \diff \bm{k} \diff s.
    \end{equation*}
\end{lemma}
\begin{proof}
    We shall only prove the formula for $\mathcal{E}^{\cp,\infty}_{t,k}$ as the others are direct consequences.
    By Lemma~\ref{lem:E-N-infty-expression-compact},
    \begin{equation*}
        \mathcal{E}^{\cp,\infty}_{t,k}(x,\zeta)
        = \int_{\mathscr{D}_k} \bm{\delta}(\Omega(\bm{k}))
        \Bigl( \prod_{1\le j \le 3} \int_{\mathscr{D}^{\cp_j}_{k_j}} \bm{\delta}(\mho^{\cp_j}(\bm{k}^j)) \diff \bm{k}^j \Bigr) \mathcal{U}^{\cp}_{t,x,\zeta}\phi(\overline{\bm{k}}) \diff \bm{k},
    \end{equation*}
    where $\overline{\bm{k}}$ is the unique element in $\mathscr{D}^\cp$ such that $\overline{\bm{k}}|_{\cp_j} = \bm{k}^j$ for $1\le j\le 3$.
    Therefore, it remains to establish the identity that, if $\bm{k}^j \in \mathscr{D}^{\cp_j}_{k_j}$, then
    \begin{equation}
        \label{eq:U-op-induction}
        \mathcal{U}^{\cp}_{t,x,\zeta}\phi(\overline{\bm{k}})
        = \int_0^t \int_{\mathscr{D}_\zeta} \prod_{1\le j\le 3} \mathcal{U}^{\cp_j}_{s,x+s(k-k_j)\bm{1}_{\alpha=\beta},\zeta_j} \diff \bm{\zeta} \diff s.
    \end{equation}
    In fact, by~\eqref{eq:T-op-transport-id} and~\eqref{eq:def-U-V-op}, we have
    \begin{equation*}
        \mathcal{U}^{\cp}_{t,x,\zeta}\phi(\overline{\bm{k}})
        = \int_{\mathscr{O}_t(\mathscr{I}_\cp)} \Bigl(\int_{\mathscr{D}^\cp_\zeta} \prod_{\mathfrak{p}\in\wp} \mathcal{W}[\phi(\cdot,\overline{\bm{k}}^\flat_{\mathfrak{p}})](x + \bm{1}_{\alpha=\beta} Z^\cp_{\mathfrak{p}}(\bm{\mu},\overline{\bm{k}}),\overline{\bm{\zeta}}_{\mathfrak{p}}) \diff \overline{\bm{\zeta}} \Bigr) \diff \bm{\mu}.
    \end{equation*}
    Next write $\bm{\mu} = (s,\bm{\mu}^1,\bm{\mu}^2,\bm{\mu}^3)$ with $\bm{\mu}^j = \bm{\mu}|_{\mathscr{I}_{\cp_j}}$.
    Then $\mathscr{O}_t(\mathscr{I}_\cp)$ is the set of all $\bm{\mu} \in \mathbb{R}^{\mathscr{I}_\cp}$ such that $0 < s < t$ and $\bm{\mu}^j \in \mathscr{O}_s(\mathscr{I}_{\cp_j})$ for $1\le j \le 3$.
    Write $\bm{k}^j = \overline{\bm{k}}|_{\cp_j}$ and $\bm{\zeta}^j = \overline{\bm{\zeta}}|_{\cp_j}$.
    Observe that if $\mathfrak{p} \in \wp(\cp_j) \subset \wp$, then we have
    \begin{equation*}
        Z^\cp_{\mathfrak{p}}(\bm{\mu},\overline{\bm{k}}) = s(k-k_j) + Z^{\cp_j}_{\mathfrak{p}}(\bm{\mu}^j,\bm{k}^j).
    \end{equation*}
    We therefore conclude~\eqref{eq:U-op-induction} with the identity
    \begin{align*}
        \prod_{\mathfrak{p}\in\wp} \mathcal{W}&[\phi(\cdot,\overline{\bm{k}}^\flat_{\mathfrak{p}})](x + \bm{1}_{\alpha=\beta} Z^\cp_{\mathfrak{p}}(\bm{\mu},\overline{\bm{k}}),\bm{\zeta}_{\mathfrak{p}}) \\
        & = \prod_{1\le j \le 3} \prod_{\mathfrak{p}\in\wp(\cp_j)} \mathcal{W}[\phi(\cdot,(\bm{k}^j)^\flat_{\mathfrak{p}})](x + \bm{1}_{\alpha=\beta} s(k-k_j) + \bm{1}_{\alpha=\beta} Z^\cp_{\mathfrak{p}}(\bm{\mu}^j,\bm{k}^j),\bm{\zeta}^j_{\mathfrak{p}}).
        \qedhere
    \end{align*}
\end{proof}

\subsubsection{Asymptotic estimates}

Now we establish the limits
\begin{equation*}
    \lim_{L \to \infty} \mathcal{E}^\cp_{t/\lambda^2,k} = \mathcal{E}^{\cp,\infty}_{t,k}, \quad
    \lim_{L \to \infty} \mathcal{N}^\cp_{t/\lambda^2,k} = \mathcal{N}^{\cp,\infty}_{t,k}.
\end{equation*}

\begin{proposition}
    \label{prop:E-approx-infty}
    Let $\cp \in \mathscr{K}^\reg_n$ with $n \in \mathbb{N}$ and let $m \in \mathbb{N}^d$.
    Then there exists $\nu > 0$ which is independent of $n$ such that uniformly for all $t > 0$ and all $x,k \in \mathbb{R}^d$, we have the estimate
    \begin{equation*}
        \sup_{x,k \in \mathbb{R}^d} \int |k^m \partial_\zeta^\ell (\mathcal{E}^\cp_{t/\lambda^2,k} - \mathcal{E}^{\cp,\infty}_{t,k})(x,\zeta)| \diff \zeta
        \lesssim C^n n^{d+4} (nR)^{4nd} t^{n-1} L^{-\nu} \|\mathcal{F}\widehat{\phi}\|_{W^{|\ell|,1}_2 W^{2,1}_1}^{4n+2}.
    \end{equation*}
    Consequently, there hold the limits:
    \begin{equation*}
        \lim_{L \to \infty} \sup_{(t,k) \in [0,1] \times \mathbb{R}^d} \|k^m(\mathcal{E}^\cp_{t/\lambda^2,k} - \mathcal{E}^{\cp,\infty}_{t,k})\|_{L^\infty}
        = \lim_{L \to \infty} \sup_{(t,k) \in [0,1] \times \mathbb{R}^d} \|k^m(\mathcal{N}^\cp_{t/\lambda^2,k} - \mathcal{N}^{\cp,\infty}_{t,k})\|_{L^\infty}
        = 0.
    \end{equation*}
\end{proposition}
\begin{proof}
    By~\eqref{eq:def-E-reg-inhom} and~\eqref{eq:def-E-infty-inhom}, we write
    \begin{align*}
        \partial_\zeta^\ell & (\mathcal{E}^\cp_{t/\lambda^2,k} - \mathcal{E}^{\cp,\infty}_{t,k})
        = \mathcal{S}^\cp_L\bigl(\theta_L \Phi^\cp_t \bigl(\partial_\zeta^\ell \mathcal{T}^\cp_{x,\zeta}\phi|_{\mathscr{D}^\cp_k} \comp (\Lambda^\cp_k)^{-1}\bigr)\bigr) - \mathcal{S}^\cp_\infty\bigl(\theta_\infty \Phi^\cp_t \bigl(\partial_\zeta^\ell \mathcal{T}^{\cp,\infty}_{x,\zeta}\phi|_{\mathscr{D}^\cp_k} \comp (\Lambda^\cp_k)^{-1}\bigr)\bigr) \\
        & = (\mathcal{S}^\cp_L-\mathcal{S}^\cp_\infty)\bigl(\theta_L \Phi^\cp_t \bigl(\partial_\zeta^\ell \mathcal{T}^\cp_{x,\zeta}\phi|_{\mathscr{D}^\cp_k} \comp (\Lambda^\cp_k)^{-1}\bigr)\bigr)
        + \mathcal{S}^\cp_\infty\bigl(\theta_\infty \Phi^\cp_t \bigl(\partial_\zeta^\ell (\mathcal{T}^{\cp}_{x,\zeta}-\mathcal{T}^{\cp,\infty}_{x,\zeta})\phi|_{\mathscr{D}^\cp_k} \comp (\Lambda^\cp_k)^{-1}\bigr)\bigr).
    \end{align*}
    Let us denote the two terms on the right hand side respectively by $\mathcal{R}_1$ and $\mathcal{R}_2$, and estimate them separately.
    By Propositions~\ref{prop:int-approx} and~\ref{prop:gauss-int-upper-bound}, we respectively have
    \begin{align*}
        |k^m \mathcal{R}_1(t,k,x,\zeta)|
        & \lesssim C^n n^d (nR)^{4nd} L^{-\nu} \|\mathcal{F} \Phi^\cp_t \bigl(\partial_\zeta^\ell \mathcal{T}^\cp_{x,\zeta} (\bm{k}_\cp^m\phi)|_{\mathscr{D}^\cp_k} \comp (\Lambda^\cp_k)^{-1}\bigr)\|_{W^{1,\infty} L^1_1}, \\
        |k^m\mathcal{R}_t(t,k,x,\zeta)|
        & \lesssim C^n (nR)^{2nd} \|\mathcal{F} \Phi^\cp_t \bigl(\partial_\zeta^\ell (\mathcal{T}^{\cp}_{x,\zeta}-\mathcal{T}^{\cp,\infty}_{x,\zeta}) (\bm{k}_\cp^m \phi)|_{\mathscr{D}^\cp_k} \comp (\Lambda^\cp_k)^{-1}\bigr)\|_{L^\infty L^1}.
    \end{align*}
    We conclude with Lemma~\ref{lem:T-op-est-Fourier} below and the Sobolev embedding theorem.
\end{proof}

\begin{lemma}
    \label{lem:T-op-est-Fourier}
    Let $\cp \in \mathscr{K}^\reg_n$ with $n \ge 1$ and let $m \in \mathbb{N}^d$.
    Then there exists $\nu > 0$ such that uniformly for all $t > 0$ and all $x,k \in \mathbb{R}^d$, there following estimates hold:
    \begin{align*}
        \int\bigl\|\mathcal{F} \Phi^\cp_t \bigl(\partial_\zeta^\ell \mathcal{T}^\cp_{x,\zeta} (\bm{k}_\cp^m \phi)|_{\mathscr{D}^\cp_k} \comp (\Lambda^\cp_k)^{-1}\bigr)\bigr\|_{W^{1,\infty} L^1_1} \diff \zeta
        & \lesssim C^n n^{4+|m|} t^{n-1} \|\mathcal{F}\widehat{\phi}\|_{W^{|\ell|,1}W^{|m|,1}_1}^{4n+2}; \\
        \int \|\mathcal{F} \Phi^\cp_t \bigl(\partial_\zeta^\ell (\mathcal{T}^{\cp}_{x,\zeta}-\mathcal{T}^{\cp,\infty}_{x,\zeta})(\bm{k}_\cp^m \phi)|_{\mathscr{D}^\cp_k} \comp (\Lambda^\cp_k)^{-1}\bigr)\|_{L^\infty L^1} \diff \zeta & \lesssim C^n n^{4+|m|} t^{n-1} L^{-\nu} \|\mathcal{F}\widehat{\phi}\|_{W^{|\ell|,1}_2 W^{|m|,1}_1}^{4n+2}.
    \end{align*}
\end{lemma}
\begin{proof}
    Note that, since $\alpha \le \beta$, the $\bm{t}$-derivative of the oscillatory kernel in the definition~\eqref{eq:T-op-def} produces the factor $\epsilon \lambda^{-2} (2\Omega^{\tau^*}(\bm{k};\bm{\zeta}+\iota\bm{\eta}/2) + \epsilon \Omega^{\tau^*}(\bm{\zeta}+\iota\bm{\eta}/2)) = \mathcal{O}(|\bm{k}|^2 + |\bm{\zeta}|^2 + |\bm{\eta}|^2)$.
    Therefore, by Lemma~\ref{lem:trans-time-int-est} and the proof of Lemma~\ref{lem:leaf-data-fourier-est}, using the notations therein, we have
    \begin{align*}
        \int \bigl\|\mathcal{F} \Phi^\cp_t \bigl(\partial_\zeta^\ell \mathcal{T}^\cp_{x,\zeta} & (\bm{k}_\cp^m\phi)|_{\mathscr{D}^\cp_k} \comp (\Lambda^\cp_k)^{-1}\bigr)\bigr\|_{W^{1,\infty} L^1_1} \diff \zeta \\
        & \lesssim 2^n n^2 t^{n-1} \iint \bigl\| \mathcal{F}_{\wp} \pi^\cp_* \partial_\zeta^\ell \widehat{\mathcal{T}}^\cp_{\eta,\zeta} (\bm{k}_\cp^m\phi)|_{(\mathbb{R}^d)^\wp}\bigr\|_{W^{1,\infty} L^1_1} \diff \eta \diff \zeta \\
        & \lesssim 2^n n^{4+|m|} t^{n-1} \|\mathcal{F}\widehat{\mathcal{V}}\|_{W^{|m|,1}_1 L^1 W^{|\ell|,1}}^{2n+1}
        \lesssim C^n n^{4+|m|} t^{n-1} \|\mathcal{F}\widehat{\phi}\|_{W^{|\ell|,1}W^{|m|,1}_1}^{4n+2},
    \end{align*}
    which yields the first estimate.
    The proof for the second estimate is similar.
    In fact, it suffices to use the following estimate of the difference of the oscillatory integral kernels:
    \begin{equation*}
        \bigl| e^{2\pi i \epsilon \lambda^{-2} \bm{t} \cdot (2\Omega^{\tau^*}(\bm{k};\bm{\xi}) + \epsilon \Omega^{\tau^*}(\bm{\xi}))} - e^{\bm{1}_{\alpha=\beta} 4\pi i \bm{t} \cdot \Omega^{\tau^*}(\bm{k};\bm{\xi})} \bigr|
        \lesssim 
        \begin{cases}
            \epsilon |\bm{\xi}|^2, & \alpha = \beta
            ;\\
            \epsilon \lambda^{-2} (|\bm{k}|^2 + |\bm{\xi}|^2), & \alpha < \beta.
        \end{cases}\qedhere
    \end{equation*}
\end{proof}

\begin{proposition}
    \label{prop:E-infty-upper-bound}
    Let $\cp \in \mathscr{K}^\reg_n$ with $n \ge 1$, then for all $m =(m_1,m_2), \ell = (\ell_1,\ell_2) \in \mathbb{N}^d \times \mathbb{N}^d$, and uniformly for all $(t,k,x,\zeta) \in [0,1] \times \mathbb{R}^{3d}$, we have the estimate
    \begin{equation}
        \label{eq:E-infty-upper-bound}
        \int |k^{m_1} \zeta^{m_2} \partial_k^{\ell_1} \partial_\zeta^{\ell_2} \mathcal{E}^{\cp,\infty}_{t,k}(x,\zeta)| \diff \zeta \lesssim \tilde{C}^n t^n \langle t \rangle^{|\ell|} n^{|m|+|\ell|} \|\widehat{\phi}\|_{W^{|\ell_2|,1}_{|m_2|+|\ell_1|}W^{|\ell_1|+d,\infty}_{|m_1|+d/2}}^{4n+2}.
    \end{equation}
\end{proposition}
\begin{proof}
    Adopting the notations in the proof of Proposition~\ref{prop:N-infty-hom-est}, we let $(\chi_\gamma)_{\gamma \in (\mathbb{Z}^d)^\wp}$ be the partition of unity of $(\mathbb{R}^d)^\wp$, and let $\varrho^1_k$ and $\varrho^2_\zeta$ denote respectively the translation operator with respect to the variables $\bm{k}$ and $\bm{\zeta}$, in the direction of $k$ and $\zeta$.
    The we define the differential operators $\mathcal{D}_k^{\ell_1} = (\varrho^1_k)^{-1} \partial_k^{\ell_1} \varrho^1_k $ and $\mathcal{D}_\zeta^{\ell_2} = (\varrho^2_\zeta)^{-1} \partial_k^{\ell_2} \varrho^2_\zeta $.
    Decompose
    \begin{equation*}
        \partial_k^{\ell_1} \partial_\zeta^{\ell_2} \mathcal{E}^{\cp,\infty}_{t,k}(x,\zeta)
        = \sum_{\gamma \in (\mathbb{Z}^d)^\wp} I_\gamma(t,k,x,\zeta),
    \end{equation*}
    where, by Lemma~\ref{lem:E-N-infty-expression-compact}, we have
    \begin{equation*}
        I_\gamma(t,k,x,\zeta) = \int_{\mathscr{D}^\cp} \bm{\delta}(\mho^\cp(\bm{k})) \mathcal{D}_k^{\ell_1} \mathcal{D}_\zeta^{\ell_2} \chi_\gamma \mathcal{U}^\cp_{t,x,\zeta} \phi(\bm{k}) \diff \bm{k}.
    \end{equation*}
    Let $\widehat{I}_\gamma(t,k,\eta,\zeta) = \mathcal{F}_{x\to\eta} I_\gamma(t,k,x,\zeta)$ and let $\widehat{\mathcal{U}}^\cp_{t,\eta,\zeta} = \mathcal{F}_{x\to\eta} \mathcal{U}^\cp_{t,x,\zeta}$.
    By~\eqref{eq:T-op-transport-id}, we have
    \begin{equation*}
        (\chi_\gamma \widehat{\mathcal{U}}^\cp_{t,\eta,\zeta}\phi) (\bm{k})
        = \int_{\mathscr{O}_t(\mathscr{I}_\cp)} \Bigl( \iint_{\mathscr{C}^\cp_\eta \times \mathscr{D}^\cp_\zeta} e^{\bm{1}_{\alpha=\beta}2\pi i \bm{\eta}^\flat|_\wp \cdot Z^\cp(\bm{\mu},\bm{k})} \prod_{\mathfrak{p} \in \wp} \widehat{\mathcal{V}}[\chi_{\gamma_{\mathfrak{p}}}\phi](\bm{k}^\flat_\mathfrak{p},\bm{\eta}^\flat_{\mathfrak{p}},\bm{\zeta}^\flat_{\mathfrak{p}}) \diff \bm{\eta} \diff \bm{\zeta} \Bigr) \diff \bm{\mu}.
    \end{equation*}
    Note that $Z^\cp(\bm{\mu},\bm{k})$ is linear in $\bm{k}$, therefore, we can write $\bm{\eta}^\flat|_\wp \cdot Z^\cp(\bm{\mu},\bm{k}) = \bm{k}^\flat|_\wp \cdot X^\cp(\bm{\mu},\bm{\eta})$ for some bilinear map~$V$. Therefore we have the identity
    \begin{align*}
        \mathcal{F}_\wp \pi_\cp^* & (\chi_\gamma \widehat{\mathcal{U}}^\cp_{t,\eta,\zeta}\phi)|_{(\mathbb{R}^d)^\wp}(\widehat{\bm{k}}+X^\cp(\bm{\mu},\bm{\eta})) \\
        & = \int_{\mathscr{O}_t(\mathscr{I}_\cp)} \Bigl\{ \iint_{\mathscr{C}^\cp_\eta \times \mathscr{D}^\cp_\zeta}  \mathcal{F}_\wp \Bigl( \prod_{\mathfrak{p} \in \wp}  \widehat{\mathcal{V}}[\chi_{\gamma_{\mathfrak{p}}}\phi](\bm{k}^\flat_\mathfrak{p},\bm{\eta}^\flat_{\mathfrak{p}},\bm{\zeta}^\flat_{\mathfrak{p}}) \Bigr)\bigl(\widehat{\bm{k}}\bigr) \diff \bm{\eta}\diff\bm{\zeta} \Bigr\} \diff \bm{\mu}.
    \end{align*}
    Following the proof of Proposition~\ref{prop:N-infty-hom-est}, using Proposition~\ref{prop:gauss-int-upper-bound}, we obtain
    \begin{align*}
        \int & |k^{m_1} \zeta^{m_2} I_\gamma(t,k,x,\zeta)| \diff \zeta
        \le \iint |k^{m_1} \zeta^{m_2} \widehat{I}_\gamma(t,k,\eta,\zeta)| \diff \zeta \diff \eta \\
        & \lesssim C^n \iint \bigl\| \mathcal{F}_\wp\bigl\{ \zeta^{m_2}  \pi_\cp^* (\bm{k}_\cp^{m_1} \mathcal{D}_k^{\ell_1} \mathcal{D}_{\zeta}^{\ell_2} \chi_\gamma \widehat{\mathcal{U}}^\cp_{t,\eta,\zeta}\phi)|_{(\mathbb{R}^d)^\wp} \bigr\} \bigr\|_{L^1} \diff \eta \diff \zeta \\
        & \lesssim C^n t^n \iint_{((\mathbb{R}^d)^\wp)^2} \Bigl\| \prod_{\mathfrak{p} \in \wp} |\widehat{\bm{z}}_{\mathfrak{p}}|^{|\ell_1|}  \mathcal{F}_\wp\Bigl\{ \bm{k}_\cp^{m_1} \bm{\zeta}_\cp^{m_2} \mathcal{D}_{\zeta}^{\ell_2} \prod_{\mathfrak{p} \in \wp} \widehat{\mathcal{V}}[\chi_{\gamma_{\mathfrak{p}}}\phi](\widehat{\bm{k}}_{\mathfrak{p}},\bm{\eta}^\flat_{\mathfrak{p}},\bm{\zeta}^\flat_{\mathfrak{p}}) \Bigr\} \Bigr\|_{L^1} \diff \bm{\eta}^\flat \diff \bm{\zeta}^\flat \\
        & \lesssim (C')^n t^n \langle t \rangle^{|\ell|} n^{|m|+|\ell|} \prod_{\mathfrak{p} \in \wp} \|\mathcal{F}\widehat{\mathcal{V}}[\chi_{\gamma_{\mathfrak{p}}}\phi]\|_{W^{|m_1|,1}_{|\ell_1|} L^1_{|\ell_1|} W^{|\ell_2|,1}_{|m_2|}} \\
        & \lesssim(C'')^n t^n \langle t \rangle^{|\ell|} n^{|m|+|\ell|} \Bigl( \prod_{\mathfrak{p} \in \wp} \|\chi_{\gamma_{\mathfrak{p}}}\|_{W^{|\ell_1|+d,2}_{-d}} \Bigr) \|\widehat{\phi}\|_{W^{|\ell_2|,1}_{|m_2|+|\ell_1|}W^{|\ell_1|+d,\infty}_{|m_1|+d/2}}^{4n+2},
    \end{align*}
    where $\widehat{\bm{z}}_{\mathfrak{p}} =\widehat{\bm{k}}^\flat_{\mathfrak{p}} - X^\cp_{\mathfrak{p}}(\bm{\mu},\bm{\eta})$
    We conclude~\eqref{eq:E-infty-upper-bound} by summing up in $\gamma$.
\end{proof}

\subsection{Wave kinetic equations}

Now we derive the wave kinetic equation WK when $\beta < \infty$ as well as the second microlocal version WK-2 and prove the convergence results in Theorems~\ref{thm:hom-inhom} and~\ref{thm:inhom-2nd}.

\subsubsection{Derivation of the wave kinetic equations}

By Proposition~\ref{prop:E-infty-upper-bound} and Lemma~\ref{lem:tree-regcp-counting}, if $t \in [0,T]$ where $T > 0$ is sufficiently small, then for all $k,x,\zeta \in \mathbb{R}^d$, we have finite sums
\begin{equation*}
    W_{k}(t,x) = \sum_{n \ge 0} \mathcal{N}^{n,\infty}_{t,k}(x-\bm{1}_{\alpha=\beta}tk),
    \quad
    E_{k,\zeta}(t,x) = \sum_{n \ge 0} \mathcal{E}^{n,\infty}_{t,k}(x-\bm{1}_{\alpha=\beta}tk,\zeta).
\end{equation*}
We will show that $(W_k)_{k \in \mathbb{R}^d}$ and $(E_{k,\zeta})_{k,\zeta \in \mathbb{R}^d}$ respectively solves the wave kinetic equations WK and WK-2, with initial data $W_k(0,x) = |\phi(x,k)|^2$ and $E_{k,\zeta}(0,x) = \mathcal{W}[\phi(\cdot,k)](x,\zeta)$, by establishing the Duhamel's formulas stated in Proposition~\ref{prop:WK-duhamel-inhom}.
Once it is proven, the estimates~\eqref{eq:A-W-convergence} and~\eqref{eq:A-E-convergence} follow from Proposition~\ref{prop:E-infty-upper-bound} and~\eqref{eq:u-wp-expansion}.
In fact, the proof of~\eqref{eq:u-wp-expansion} yields the identity
\begin{equation*}
    \mathbb{E}\mathcal{W}[A^{2n+1}_k(t)]\Bigl(x,\frac{\xi-k}{\epsilon}\Bigr) = \mathcal{E}^n_{t,k}\Bigl(x-\epsilon t \xi,\frac{\xi-k}{\epsilon}\Bigr).
\end{equation*}

\begin{proposition}
    \label{prop:WK-duhamel-inhom}
    If $T > 0$ is sufficiently small, then for all $t \in [0,T]$ and $k,x,\zeta \in \mathbb{R}^d$, we have
    \begin{align*}
        W_k(t,x) & = |\phi(x,k)|^2 + 2 \int_0^t \int_{\mathscr{D}_k} \bm{\delta}(\Omega(\bm{k})) \prod_{1 \le j \le 3} W_{k_j}(s,x-\bm{1}_{\alpha=\beta}(t-s)k) \diff \bm{k} \diff s,\\
        E_{k,\zeta}(t,x) & = \mathcal{V}[\phi](k,x,\zeta) + 2 \int_0^t \iint_{\mathscr{D}_k\times \mathscr{D}_\zeta} \bm{\delta}(\Omega(\bm{k})) \prod_{1 \le j \le 3} E_{k_j,\zeta_j}(s,x-\bm{1}_{\alpha=\beta}(t-s)k) \diff \bm{k} \diff \bm{\zeta} \diff s.
    \end{align*}
\end{proposition}
\begin{proof}
    The proof is the same as that of Proposition~\ref{prop:N-duhamel-hom}.
    In fact, by Proposition~\ref{prop:E-infty-upper-bound}, we can pass $n \to \infty$ in the inductive formulas in Lemma~\ref{lem:E-N-infty-induction} and obtain these Duhamel's formulas.
\end{proof}

\subsubsection{Microlocalization}

Let $\chi \in C_c^\infty(\mathbb{R}^{2d})$.
By~\eqref{eq:E-spec-E-id}, we write
\begin{align*}
    \iint \chi(x,\xi) & \mathbb{E} \mathcal{W}\bigl[ u^{2n+1}\bigl(t/\lambda^2\bigr)\bigr] \bigl(x/\epsilon,\xi\bigr) \diff x \diff \xi \\
    & = \frac{1}{L^d} \sum_{k \in \mathbb{Z}^d_L} \iint \chi\Bigl(x + \frac{\epsilon}{\lambda^2} t(k+\epsilon\zeta),k+\epsilon\zeta\Bigr) \mathcal{E}^n_{t/\lambda^2,k}(x,\zeta) \diff x \diff\zeta \\
    & = \frac{1}{L^d} \sum_{k \in \mathbb{Z}^d_L} \iint \chi(x +  tk \bm{1}_{\alpha=\beta},k) \mathcal{E}^{n,\infty}_{t,k}(x,\zeta) \diff x \diff \zeta + \frac{1}{L^d} \sum_{k \in \mathbb{Z}^d} \sum_{1\le j \le 3} \mathcal{R}^j_L(t,k).
\end{align*}
The remainders have the explicit expressions, where we write $\xi = k + \epsilon\zeta$ for short:
\begin{align*}
    \mathcal{R}^1_L(t,k)
    & =  \sum_{\cp \in \mathscr{K}_n \backslash \mathscr{K}^\reg_n}  \iint \chi\Bigl(x + \frac{\epsilon}{\lambda^2} t\xi,\xi\Bigr) \mathcal{E}^{\cp}_{t/\lambda^2,k}(x,\zeta) \diff x \diff \zeta,  \\
    \mathcal{R}^2_L(t,k) & = \sum_{\cp \in \mathscr{K}^\reg_n}  \iint \chi\Bigl(x + \frac{\epsilon}{\lambda^2} t\xi,\xi\Bigr) \bigl( \mathcal{E}^{\cp}_{t/\lambda^2,k}(x,\zeta) - \mathcal{E}^{\cp,\infty}_{t,k}(x,\zeta) \bigr)\diff x \diff \zeta, \\
    \mathcal{R}^3_L(t,k) 
    &= \sum_{\cp \in \mathscr{K}^\reg_n}  \iint \Bigl(\chi\Bigl(x + \frac{\epsilon}{\lambda^2} t\xi,\xi\Bigr) - \chi(x + \bm{1}_{\alpha=\beta} tk,k)\Bigr)\mathcal{E}^{\cp,\infty}_{t,k}(x,\zeta)\diff x \diff \zeta.
\end{align*}
By Propositions~\ref{prop:E-upper-bound-inhom}, \ref{prop:E-approx-infty}, and~\ref{prop:E-infty-upper-bound}, we have
\begin{equation*}
    \lim_{L \to \infty} \sup_{t \in [0,T]} \frac{1}{L^d} \sum_{k \in \mathbb{Z}^d} \sum_{1\le j \le 3} |\mathcal{R}^j_L(t,k)| = 0.
\end{equation*}
Finally by Proposition~\ref{prop:E-infty-upper-bound}, we can let $L \to \infty$ in the first term and make it converge to an integral, which establishes the limit~\eqref{eq:A-W-convergence}.

\subsubsection{Second microlocalization}

Let $\chi \in C_c^\infty(\mathbb{R}^{3d})$, for all $(t,k) \in \mathbb{R} \times \mathbb{Z}^d_L$, set:
\begin{equation*}
    \mathcal{G}^n_k(t) = \epsilon^d \iint \chi\Bigl(\epsilon x,\xi,\frac{\xi-k}{\epsilon}\Bigr) \mathcal{E}^n_{t/\lambda^2,k}(t/\lambda^2,x,\xi) \diff x \diff \xi.
\end{equation*}
By~\eqref{eq:E-spec-E-id}, when $\beta > 1$ and if $L > 0$ is sufficiently large, then for all $k \in \mathbb{Z}^d_L$, we have
\begin{align*}
    L^d \mathcal{G}^n_k(t) 
    & = \iint \chi\Bigl(x + \frac{\epsilon}{\lambda^2} t(k+\epsilon\zeta),k+\epsilon\zeta,\zeta\Bigr) \mathcal{E}^n_{t/\lambda^2,k}(x,\zeta) \diff x \diff\zeta \\
    & = \iint \chi(x +  tk \bm{1}_{\alpha=\beta},k,\zeta) \mathcal{E}^{n,\infty}_{t,k}(x,\zeta) \diff x \diff \zeta + \sum_{1\le j \le 3} \mathcal{R}^j_L(t,k).
\end{align*}
The remainders have the explicit expressions, where we write $\xi = k + \epsilon\zeta$ for short:
\begin{align*}
    \mathcal{R}^1_L(t,k)
    & =  \sum_{\cp \in \mathscr{K}_n \backslash \mathscr{K}^\reg_n}  \iint \chi\Bigl(x + \frac{\epsilon}{\lambda^2} t\xi,\xi,\zeta\Bigr) \mathcal{E}^{\cp}_{t/\lambda^2,k}(x,\zeta) \diff x \diff \zeta,  \\
    \mathcal{R}^2_L(t,k) & = \sum_{\cp \in \mathscr{K}^\reg_n}  \iint \chi\Bigl(x + \frac{\epsilon}{\lambda^2} t\xi,\xi,\zeta\Bigr) \bigl( \mathcal{E}^{\cp}_{t/\lambda^2,k}(x,\zeta) - \mathcal{E}^{\cp,\infty}_{t,k}(x,\zeta) \bigr)\diff x \diff \zeta, \\
    \mathcal{R}^3_L(t,k) 
    &= \sum_{\cp \in \mathscr{K}^\reg_n}  \iint \Bigl\{\chi\Bigl(x + \frac{\epsilon}{\lambda^2} t\xi,\xi,\zeta\Bigr) - \chi\Bigl(x + \bm{1}_{\alpha=\beta} tk,k,\zeta\Bigr)\Bigr\}\mathcal{E}^{\cp,\infty}_{t,k}(x,\zeta)\diff x \diff \zeta.
\end{align*}
By Propositions~\ref{prop:E-upper-bound-inhom}, \ref{prop:E-approx-infty}, and~\ref{prop:E-infty-upper-bound}, for all $j \in \{1,2,3\}$ we have
\begin{equation*}
    \lim_{L \to \infty} \frac{1}{L^d} \sup_{t \in [0,T]}\sup_{k \in \mathbb{Z}^d} |\mathcal{R}^j_L(t,k)| = 0,
\end{equation*}
which establishes the limits~\eqref{eq:u-second-ml} by summing up in~$n$.

\printindex

\printbibliography

\end{document}